\documentclass[12pt,reqno]{amsart}
\usepackage{amsmath,amssymb,amsthm,mathtools,wasysym,calc,verbatim,enumitem,tikz,url,hyperref,mathrsfs,cite,fullpage,bbm}

\usepackage{comment}
\usepackage{verbatim}

\usepackage{bbm}
\usepackage{mathtools}
\usepackage{enumitem}
\usepackage{amssymb}
\usepackage[utf8]{inputenc}
\usepackage{amsmath}
\usepackage{hyperref}
\usepackage{caption}
\usepackage{float}
\usepackage{listings}
\usepackage{hyperref}
\usepackage{mathtools}

\newtheorem{theorem}{Theorem}[section]
\newtheorem{proposition}[theorem]{Proposition}
\newtheorem{lemma}[theorem]{Lemma}
\newtheorem{corollary}[theorem]{Corollary}
\theoremstyle{definition}
\newtheorem{definition}{Definition}[section]
\newtheorem{remark}{Remark}[section]

\newcommand{\Bin}{\mathrm{Bin}}
\newcommand{\varepsilonghghghghg}{L}
\newcommand{\EEEEE}{\mathcal{D}}

\usepackage{fancyhdr}

\title{A new bound in Majority Dynamics on Random Graphs}
\author{Sean Jaffe}
\date{August 2024}
\address{Trinity College, University of Cambridge, CB2 1TQ, United Kingdom}
\email{scj47@cam.ac.uk}

\begin{document}

\begin{abstract}
    We study the evolution of majority dynamics on Erdős-Rényi $G(n,p)$ random graphs. In this process, each vertex of a graph is assigned one of two initial states. Subsequently, on every day, each vertex simultaneously updates its state to the most common state in its neighbourhood.

  If the difference in the numbers of vertices in each state on day $0$ is larger than \break$ \max \left\{\frac{1}{\sqrt{p}} \exp\left[A\sqrt{\log \left(\frac{1}{p}\right)}\right] , Bp^{-3/2} n^{-1/2} \right\}$ for constants $A$ and $B$, we demonstrate that the state with the initial majority wins with overwhelmingly high probability. This extends work by Linh Tran and Van Vu (2023), who previously considered this phenomenon.

We also study majority dynamics with a random initial assignment of vertex states. When each vertex is assigned to a state with equal probability, we show that unanimity occurs with high probability for every $p \geq \lambda n^{-2/3}$, for some constant $\lambda$. This improves work by Fountoulakis, Kang and Makai (2020).

Furthermore, we also consider a random initial assignment of vertex states where a vertex is slightly more likely to be in the first state than the second state.  Previous work by Zehmakan (2018) and Tran and Vu (2023) provided conditions on how big this bias needs to be for the first colour to achieve unanimity with high probability. We strengthen these results by providing a weaker sufficient condition.

\end{abstract}

\maketitle

\vspace{-2em}

\section{Introduction}

The exchange of information between individuals is key to understanding a broad range of social, political, and economic phenomena \cite{Mossel_2017}. Consequently, studying processes in social networks, where individuals update their opinions based on those around them, has become a significant area of research. As noted by Mossel and Tamuz in their 2017 survey \cite{Mossel_2017}, these models have valuable applications in biology, economics, and statistical physics. Furthermore, they have also led to important developments in computer science, particularly in the area of network security \cite{Alistarh2015FastAE, Alistarh2016TimeSpaceTI, Gacs, PhysRevLett}.

In this paper, we study a simple model of information exchange by individuals in a network, known as \textit{majority dynamics}. In this process, each vertex of a graph is assigned an initial colour. Subsequently, on every day, each vertex simultaneously updates its colour to the most common colour in its neighbourhood. In the event of a tie, the vertex's colour remains unchanged.

\begin{definition}
    (Majority Dynamics). Suppose we are given a simple undirected graph \(G = (V,E)\), and a partition \(V = C_1 \cup C_2\). The majority dynamics process is a process that begins on day 0, where each vertex in \(C_i\) is coloured with colour $i$ (for $i=1,2$). On each day \(t+1\), every vertex 
    adopts the colour the majority of its neighbours held on day $t$. In the event of a tie, the vertex's colour remains unchanged.
\end{definition}

We investigate the termination of the majority dynamics process. Our study centers on Erdős-Rényi $G(n,p)$ random graphs, with a particular focus on the concept of \textit{unanimity}: Given an initial colouring, what is the probability that, after some time, all vertices are eventually same colour? We define a colour as having \textit{won} if all vertices ultimately adopt that colour.

This topic has been extensively studied by many authors in recent studies \cite{Benjamini_2016, fountoulakis2020resolutionconjecturemajoritydynamics, Chakraborti_2023, tran2023powerfewphenomenonsparse, tran2020reachingconsensusrandomnetworks, zehmakan2018opinionformingerdosrenyirandom}. For $C_1$ and $C_2$ of fixed size, Linh Tran and Van Vu proved that when $p$ is constant, ${\Delta := \frac{1}{2} \big||C_1|-|C_2|\big|}$ only needs to be constant for the colour with the initial majority to win within four days with probability arbitrarily close to 1, regardless of the size $n$ of the population \cite{tran2020reachingconsensusrandomnetworks}. They referred to this phenomenon as ``the power of few''. In a later work, they improve this result to consider all values of $p$ above the connectivity threshold, $(\log n)/n$ \cite{tran2023powerfewphenomenonsparse}. They proved that, for $\lambda >0 $ fixed, if $1-\frac{10}{n}\geq p\geq \frac{(1+\lambda)\log n}{n}$ and $\Delta \geq \frac{10}{p}$ then the colour with the initial majority wins with high probability. In this paper, we extend their result and show that it holds for smaller values of $\Delta$. Our main theorem shows that \(\Delta =  \max \left\{\dfrac{1}{\sqrt{p}} \exp\left[\Omega\left(\sqrt{\log \left(\dfrac{1}{p}\right)}\right)\right] , \omega(p^{-3/2} n^{-1/2}) \right\}\) is sufficient to guarantee that the colour with the initial majority wins with probability $1-o(1)$ (for the same range of values of $p$).

This paper also considers the scenario where the initial colouring of the graph is random. Define the \textbf{random $\mathbf{1/2}$ initial colouring scheme} as the process where every vertex of a random graph is coloured one of two colours with probability $1/2$ each, and then majority dynamics is run. 

 Benjamini, Chan, O'Donnell, Tamuz and Tan \cite{Benjamini_2016} proved that the side with the initial majority wins with probability at least $0.4$ when ${p \geq \lambda n^{-1/2}}$ for a sufficiently large constant $\lambda$. The winning probability was improved to $1-\epsilon$ for any constant $\epsilon>0 $, provided that ${\lambda = \lambda_\epsilon \overset{\epsilon \to 0}{\longrightarrow} \infty}$, by Fountoulakis, Kang and Makai \cite{fountoulakis2020resolutionconjecturemajoritydynamics}. Furthermore, they show that unanimity is achieved within $4$ days. This theorem was reproven by Linh Tran and Van Vu \cite{tran2023powerfewphenomenonsparse}, albeit with the winning time taken to be $5$ days. Furthermore, Chakraborti, Han Kim, Lee and Tran \cite{Chakraborti_2023} reduced the value of $p$ for
which unanimity is reached to $p\gg n^{-3/5}\log n$. In this paper, we use our main theorem to further reduce the value of $p$ for which unanimity is reached to $p\gg n^{-2/3}$. 


One can also explore a more generalized random initial colouring scheme. For $q_1 + q_2 = 1$, define the \textbf{random $\mathbf{(q_1, q_2)}$ initial colouring scheme} as the process where each vertex of a random graph is assigned colour $1$ with probability $q_1$ and colour $2$ with probability $q_2$.
This has been studied by Zehmakan \cite{zehmakan2018opinionformingerdosrenyirandom}. They proved that for $p \geq (1+\lambda)(\log n)/n$, if $q_1-q_2 = \omega((pn)^{-1/2})$, then colour $1$ wins with probability $1-o(1)$. Linh Tran and Van Vu \cite{tran2023powerfewphenomenonsparse} were able to improve this result by showing that if $q_1-q_2 = \omega ( (pn)^{-1})$, then colour $1$ wins with probability $1-o(1)$. Our main theorem can be used to further improve these results, and show that $ q_1 - q_2 = \max \left\{\dfrac{1}{\sqrt{p}} \exp\left[\omega\left(\sqrt{\log \left(\dfrac{1}{p}\right)}\right)\right]n^{-1} , \omega\left(p^{-3/2} n^{-3/2} \right)\right\}$ is sufficient to ensure colour $1$ wins with  probability $1-o(1)$. 

We now present the statement of the main result of this paper. 
\section{Main Result}

\begin{theorem}\label{actual_theorem}
    Let $V$ be a set of $n$ vertices with arbitrary partition $V = C_{1} \cup C_{2}$, and suppose $|C_{1}| = \frac{n}{2} + \Delta $. Let $G$ be a $G(n,p)$ random graph drawn over $V$. For any $\epsilon,\lambda > 0 $ fixed, there are constants $A,B$ such that if $\Delta  \geq \max \left\{\dfrac{1}{\sqrt{p}} \exp\left[A\sqrt{\log \left(\dfrac{1}{p}\right)}\right] , Bp^{-3/2} n^{-1/2} \right\}$ and $1-\frac{10}{n}\geq p\geq \frac{(1+\lambda)\log n}{n}$ , then majority dynamics on $G$ with initial colouring $(C_1,C_2)$ results in colour $1$ winning with probability at least \(1-\epsilon\). Furthermore, colour $1$ wins in $\mathcal{O}(\log_{pn} n) $ days.
\end{theorem}


Our main theorem extends the work of Linh Tran and Van Vu  \cite{tran2023powerfewphenomenonsparse}. They proved \break that, provided $\Delta \geq \frac{10}{p}$, colour $1$ wins with high probability. Since \break\(\max \left\{\dfrac{1}{\sqrt{p}} \exp\left[A\sqrt{\log \left(\dfrac{1}{p}\right)}\right] , Bp^{-3/2} n^{-1/2} \right\} \ll \dfrac{10}{p}\) as $p\to 0$, our result is asymptotically better.

\section{Relations with the random colouring scheme.}
As previously noted, Theorem \ref{actual_theorem} can be applied to enhance recent results concerning unanimity when the initial colouring of the graph is random. We first present Theorem \ref{firstminitheorem}, which  represents an improvement over the work by Chakraborti, Han Kim, Lee and Tran. on the majority dynamics process with random $1/2$ initial colouring scheme \cite{Chakraborti_2023}.

\begin{theorem}\label{firstminitheorem}
    For any fixed $\epsilon >0$, there is a constant $\lambda >0$ such that: in majority dynamics on a $G(n,p)$ with $p\geq \lambda n^{-2/3}$ and a random $1/2$ initial colouring, the colour with the initial majority wins in $\mathcal{O}(1) $ days with probability at least $1-\epsilon$
\end{theorem}
\begin{proof}
    For $i=1,2$, let $C_i$ be the set of vertices of colour $i$ on day $0$, and define \(\Delta = \frac{1}{2}\big| |C_1| - |C_2|\big|\).  The quantity $|C_1| - |C_2|$ is a sum of $n$ i.i.d. Rademacher variables. Therefore, $|C_1| - |C_2|$ has mean $0$ and variance $n$. Therefore, by Berry-Esseen Theorem (see Corollary \ref{in_paper_A.2} in Appendix \ref{appendix_A}), for any $\epsilon >0$, there is a constant, $C_\epsilon >0 $, such that $\Delta \geq C_\epsilon \sqrt{n}$ with probability at least $1-\epsilon/2$. Suppose a fixed instance of $(C_1,C_2)$ satisfies this.
    By Theorem \ref{actual_theorem}, there is a constant $\lambda$ (depending on $\epsilon$ and $C_\epsilon$) such that if $p\geq \lambda n^{-2/3}$ and $\Delta \geq C_\epsilon n^{-1/2}$, then the colour with the initial majority wins within $\mathcal{O}(\log_{pn}n)$ days with probability at least $1- \epsilon/2$. By a union bound, this means that the probability that the colour with the initial majority wins is at least $1-\epsilon$. Furthermore, \(\mathcal{O}(\log_{pn}n) = \mathcal{O}(1)\).
    
\end{proof}

We now present Theorem \ref{secondminitheorem}, which improves results by Zehmakan \cite{zehmakan2018opinionformingerdosrenyirandom} and Linh Tran and Van Vu \cite{tran2023powerfewphenomenonsparse} on majority dynamics with the random $(q_1,q_2)$ initial colouring scheme.

\begin{theorem}\label{secondminitheorem}
    Let \(\lambda >0\) be any constant. On a $G(n,p)$ random graph with $p\geq \frac{(1+\lambda)\log n}{n}$, majority dynamics with a random $(q_1,q_2)$ initial colouring where $0\leq q_1,q_2\leq 1$, $q_1+q_2=1$, and $q_1 = \frac{1}{2} + \max \left\{\dfrac{1}{\sqrt{p}} \exp\left[\omega\left(\sqrt{\log \left(\dfrac{1}{p}\right)}\right)\right]n^{-1} , \omega\left(p^{-3/2} n^{-3/2} \right)\right\}$ results in colour $1$ winning with probability $1-o(1)$.
\end{theorem}
\begin{proof}
    By the Berry-Esseen Theorem, with probability $1-o(1)$, ${|C_1|>|C_2|}$ and the initial gap $\Delta = (|C_1|-|C_2|)/2$ satisfies $\Delta = \max \left\{\dfrac{1}{\sqrt{p}} \exp\left[\omega\left(\sqrt{\log \left(\dfrac{1}{p}\right)}\right)\right] ,  \omega\left(p^{-3/2} n^{-1/2} \right)\right\}$. Suppose fixed instance of $(C_1,C_2)$ satisfies this. Then, by Theorem \ref{actual_theorem}, colour $1$ wins with probability $1-o(1)$. This completes the proof.
\end{proof}


\section{Proof of Theorem \ref{actual_theorem}}
Linh Tran and Van Vu \cite{tran2023powerfewphenomenonsparse} proved that for any $\lambda > 0 $ fixed, if $1-\frac{10}{n}\geq p\geq \frac{(1+\lambda)\log n}{n}$ and $\Delta\geq \frac{10}{p}$, then majority dynamics on $G$ with initial colouring $(C_1,C_2)$ achieves a unanimity of colour $1$ with probability at least \(
     1 - \mathcal{O}\left(\frac{1}{n} + \frac{1}{p\Delta^2} + n^{-\lambda/2}\right).
     \)
As such, we do not need to consider $\Delta\geq \frac{10}{p}$. Furthermore, for sufficiently large constants $A,B$ we have: $\frac{10}{p}\leq\max \left\{\dfrac{1}{\sqrt{p}} \exp\left[A\sqrt{\log \left(\dfrac{1}{p}\right)}\right] , Bp^{-3/2} n^{-1/2} \right\} $ whenever $p\geq \frac{1}{4}$. Therefore, we also need not consider cases where $p\geq \frac{1}{4}$. As such, we will only consider $\frac{10}{p}\geq \Delta$ and $\frac{1}{4}\geq p$. We will prove Theorem \ref{the_power_of_few_2_colours}:


\begin{theorem}\label{the_power_of_few_2_colours}
Let $V$ be a set of $n$ vertices with arbitrary partition $V = C_{1} \cup C_{2}$, and suppose $|C_{1}| = \frac{n}{2} + \Delta $ for some $\Delta > 0$. Let $G$ be a $G(n,p)$ random graph drawn over $V$, and suppose $\frac{1}{4}\geq p\geq \frac{(1+\lambda)\log n}{n}$ for some $\lambda >0$ fixed. 
There are constants $A,B$ such that if $\frac{10}{p}\geq \Delta \geq \max \left\{\dfrac{1}{\sqrt{p}} \exp\left[A\sqrt{\log \left(\dfrac{1}{p}\right)}\right] , Bp^{-3/2} n^{-1/2} \right\}$, then majority dynamics on $G$ with initial colouring $(C_1,C_2)$ achieves a unanimity of colour $1$ within $\mathcal{O}(\log_{pn} n) $ days with probability at least\[
 1 - \mathcal{O}\left( \frac{1}{p\Delta^2} + \frac{1}{np^3\Delta^2} +n^{-\lambda/2}\right).
 \]
\end{theorem}

\section{Overview of the proof}
This proof uses several notations, as described below, as well as several technical lemmas whose proofs are included in the appendices.
\begin{itemize}
\item \(\Phi(a) := (2 \pi)^{-1/2} \int_{-\infty}^a e^{-x^2/2}dx\) and \(\Phi_0(a) := \Phi(a) - \frac{1}{2}\).
\item \(\phi(x) := \frac{1}{\sqrt{2\pi}}e^{-\frac{x^2}{2}}\). This is the probability density function of a standard normal random variable.
\item $C_{BE} = 0.56$ denotes the constant in the latest version of the Berry-Esseen Theorem (Theorem \ref{in_paper_A.1} in Appendix \ref{appendix_A}).
\item For \(i=1,2\), define \(C_{i,t}\) to be the set of vertices that are colour \(i\) on day \(t\). 
\item For a vertex $v \in V$, define \(L(v) := \begin{cases} 
      1 & v\in C_{1,0} \\
      -1 & v \in C_{2,0} 
   \end{cases}\).
\end{itemize}
%
\begin{definition}\label{definition__of__R__hat__v}
For a vertex $w\in V$, define 
\begin{multline*}
\hat{R}_w := \bigg\{ u \in V\backslash \{w\} : \bigg( |\Gamma(u)\cap( C_{1,0}\backslash\{w\})|+L(w) > |\Gamma(u)\cap (C_{2,0}\backslash\{w\})|\land u \in C_{2,0}\bigg)\\
\lor \bigg(|\Gamma(u)\cap(C_{1,0}\backslash\{w\})|+L(w) \geq |\Gamma(u)\cap (C_{2,0}\backslash\{w\})|\land u \in C_{1,0}\bigg)\bigg\}.
\end{multline*}

A careful inspection of the above definition allows one to see that $\hat{R}_w$ only depends on the edges of $G[V\backslash\{w\}]$ (and the colouring of $G$).
Furthermore, if \(v\neq w\), then \[v \in \hat{R}_w \land v \sim w \Leftrightarrow v \in C_{1,1} \land v \sim w.\]
This means that the set of colour-1 neighbours of $w$ on day 1 is precisely \(\hat{R}_w \cap \Gamma(w)\). 
\end{definition} 

The proof of Theorem \ref{the_power_of_few_2_colours} has eight steps, spanning the rest of this paper:
\begin{enumerate}
    \item We show that \(\mathbb{E}\left[|C_{1,1}|\right]\geq \frac{n}{2} + \frac{19\sqrt{pn}\Delta}{1920}\).
    \item  For $D>0$ an arbitrary constant, we show that \(
    \mathbb{P}\left(\Big||\hat{R}_v| - \mathbb{E}\left[|\hat{R}_v|\right]\Big| \geq D \sqrt{pn}\Delta\right) =\mathcal{O}\left( \frac{1}{\sqrt{(n-1)p(1-p)}} + {p} \Delta\right),\) where $v$ is any vertex.
\item Using the previous two steps, we show that \(\mathbb{E}\left[ |C_{1,2}|\right] \geq \frac{n}{2} +\Omega(pn\Delta).\)
\item We show that \(\mathbf{Var}\left(|{{{{{{{{{{C}}}}}}}}}}_{1,2}| \right) =\mathcal{O}\left(n^2p\right) + \mathcal{O}\left(\frac{n}{p}\right).\)
    \item Using the previous two steps and Chebyshev's inequality, we show that with probability at least \(1- \mathcal{O}\left(\frac{1}{p\Delta^2}\right) - \mathcal{O}\left(\frac{1}{np^3\Delta^2}\right)\), the number of nodes of colour $2$ decreases to at most \(\frac{n}{2} - \Omega(pn\Delta)\) after the second day. 
    \item With probability at least \( 1 - e^{-\Omega(n)} \), the number of nodes of colour $2$ decreases to at most \( \frac{n}{2} - \Omega(n) \) after the third day.
    \item With probability at least \( 1 - \mathcal{O}(n^{-\lambda/2}) \), \( G \) has “nice” properties that force the number of nodes of colour $2$ to shrink repeatedly until there are no nodes of colour $2$ remaining.
    \item Finally, we combine the results from the previous steps to prove Theorem \ref{the_power_of_few_2_colours}.
\end{enumerate}
Linh Tran and Van Vu proved a similar result to our main theorem \cite{tran2023powerfewphenomenonsparse}. Their proof consists of a strong analysis of the configuration of the graph after one day, followed by what we refer to as steps 6 and 7 in our own proof. We improve their result by conducting a strong analysis of the graph's configuration after two days (instead of one), which we do in steps 1-5. After this, we conclude the proof by following work by Linh Tran and Van Vu \cite{tran2023powerfewphenomenonsparse}.

We begin the proof in the next section with Step 1.

\section{Step 1 - Day 1, Part 1}

The goal of this section is to find a bound for \(\mathbb{E}\left[|C_{1,1}|\right]\). We achieve this by finding the probability that a single vertex is colour 1 on day 1, and then use linearity of expectation.

\begin{lemma}\label{expectation_lemma_two_colours}
If  $n \geq \exp \left[ 3*10^6\right]$, $\frac{\log n}{n} \leq p \leq \frac{1}{4}$ and $\frac{10}{p} \geq \Delta \geq 1$, 

\[\mathbb{E}[|C_{1,1}|] \geq \frac{n}{2} + \frac{19\sqrt{pn}\Delta}{1920}.\]

\end{lemma}

\begin{proof}

Define $n_i := |C_{i,0}|$ for $i=1,2$, and let \(X_1\sim \mathrm{Bin}(n_1-1,p)\), \(X_2 \sim \mathrm{Bin}(n_2,p)\) be independent random variables. Let $u$ be a vertex selected uniformly at random from $V$. By the law of total probability,
\begin{align*}
\mathbb{P}(u \in C_{1,1}) &= \mathbb{P}(u \in C_{1,1} | u \in C_{1,0}) \mathbb{P}( u \in C_{1,0}) + \mathbb{P}(u \in C_{1,1} | u \in C_{2,0}) \mathbb{P}( u \in C_{2,0})\\
&= \frac{n_1}{n}\mathbb{P}(u \in C_{1,1} | u \in C_{1,0}) + \frac{n_2}{n}\mathbb{P}(u \in C_{1,1} | u \in C_{2,0}).
\end{align*}
It is easy to see that \(|\Gamma(u) \cap C_{i,0}| = \sum_{w \in C_{i,0}\backslash \{u\}} \mathbbm{1}_{u \sim w}\) for $i =1,2$, and \(\{\mathbbm{1}_{u \sim w}\}_{w \in V\backslash \{u\}}\) is a set of i.i.d \(\mathrm{Ber}(p)\) random variables. Therefore, \(|\Gamma(u) \cap C_{i,0}| \sim \mathrm{Bin}(|C_{i,0}\backslash \{u\}|,p)\), and \(|\Gamma(u) \cap C_{1,0}|\) and \(|\Gamma(u) \cap C_{2,0}|\) are independent. By definition,
\begin{align*}
\mathbb{P}(u \in C_{1,1} | u \in C_{1,0}) &= \mathbb{P}\left(|\Gamma(u) \cap C_{1,0}| \geq |\Gamma(u) \cap C_{2,0}|\big|u \in C_{1,0}\right)\\
&\geq \mathbb{P}(X_1\geq X_2),
\end{align*}
where the final inequality follows because \(|C_{1,0}\backslash \{u\}|\geq n_1-1\) and \(|C_{2,0}\backslash \{u\}|\leq n_2\). Likewise, 
\begin{align*}
\mathbb{P}\left(u \in C_{1,1} |u \in C_{2,0}\right) & = \mathbb{P}\left(|\Gamma(u) \cap C_{1,0}|>|\Gamma(u) \cap C_{2,0}|\big|u \in C_{2,0}\right) \\
&\geq \mathbb{P}(X_1> X_2).
\end{align*}
\(
\text{Hence, }\mathbb{P}(u\in C_{1,1}) \geq \frac{n_1}{n}\mathbb{P}(X_1\geq X_2) + \frac{n_2}{n}\mathbb{P}(X_1>X_2).\) 
Since $n_1\geq n_2$ and $\mathbb{P}(X_1\geq X_2)\geq \mathbb{P}(X_1> X_2)$, we have:
\begin{align*}
\mathbb{P}(u\in C_{1,1})&\geq \frac{1}{2}\mathbb{P}(X_1\geq X_2) + \frac{1}{2}\mathbb{P}(X_1>X_2).
\end{align*}
Let $A \sim \mathrm{Bin}(n_2,p)$ and $B \sim \mathrm{Bin}(2 \Delta - 1,p)$ be independent random variables (that are also independent of $X_2$). It is easy to see that $A+B \overset{d}{=} X_1$. Therefore,
%
\begin{align}\label{firstthingy2colours}
    \mathbb{P}(u\in C_{1,1}) &\geq \frac{1}{2}\mathbb{P}(A+B\geq X_2) + \frac{1}{2}\mathbb{P}(A+B >X_2)\notag\\
    &\geq \frac{1}{2}\mathbb{P}(B=0) \bigg( \mathbb{P}(A \geq X_2) + \mathbb{P}(A>X_2)
    \bigg)\notag\\
    &+ \frac{1}{2}\sum_{d=1}^{2\Delta-1} \bigg(\mathbb{P}(A+d \geq X_2) + \mathbb{P}(A+d-1 \geq X_2)\bigg)\mathbb{P}(B=d).
\end{align}
For $d\geq 1$: \(\mathbb{P}(A+d \geq X_2) \geq \mathbb{P}(A+d \geq X_2 \geq A+1)+\frac{1}{2}\) and \(\mathbb{P}(A+d-1 \geq X_2)\geq \frac{1}{2}\). Therefore:

\begin{align*}
\mathbb{P}(u\in C_{1,1}) &\geq \frac{1}{2}\mathbb{P}(B=0) +  \frac{1}{2}\sum_{d=1}^{2\Delta-1} \bigg( \mathbb{P}(A+d \geq X_2 \geq A+1) +1\bigg)\mathbb{P}(B=d)\\
    &=\frac{1}{2} \sum_{d=0}^{2\Delta-1}\mathbb{P}(B=d) + \frac{1}{2}\sum_{d=1}^{2\Delta-1}\mathbb{P}(A+d \geq X_2 \geq A+1)\mathbb{P}(B=d)\\
    &= \frac{1}{2} + \frac{1}{2}\sum_{d=1}^{2\Delta-1} \mathbb{P}(A+d=X_2)\mathbb{P}(B \geq d).
\end{align*}
We will now proceed by bounding \(\sum_{d=1}^{2\Delta-1} \mathbb{P}(A+d=X_2)\mathbb{P}(B \geq d)\). Define \(J:= \left\lfloor \frac{\sqrt{np(1-p)}}{240C_{BE}} \right\rfloor\) and \(M := \min \{J,2\Delta-1\}\). By Corollary \ref{cttorollaryththththth} in Appendix \ref{appendix_A}, for every $1 \leq d \leq M$, \({\mathbb{P}(A+d = X_2) } \geq \frac{1}{12\sqrt{np(1-p)}} \). 
Therefore:

\begin{align}\label{NEWLYarabarukisetsuotomedomeyorepepr}
\sum_{d=1}^{2\Delta-1} \mathbb{P}(A+d=X_2)\mathbb{P}(B \geq d) &\geq \sum_{d=1}^{M} \mathbb{P}(A+d=X_2)\mathbb{P}(B \geq d)\notag\\
&\geq \frac{1}{12\sqrt{np(1-p)}} \sum_{d=1}^{M}  \mathbb{P}(B\geq d).
\end{align}
We now need to bound \(\sum_{d=1}^{M}  \mathbb{P}(B\geq d)\).
Since $B\sim \mathrm{Bin}(2\Delta-1,p)$, we can express $B$ as $\sum_{i=1}^{2\Delta-1} Z_i$, where $Z_1,\dots,Z_{2\Delta-1}$ is a sequence of i.i.d. $\mathrm{Ber}(p)$ random variables. 
Let \(S_1\cup S_2\cup \cdots \cup S_M\) be a partition of \(\{1,2,\dots,2\Delta-1\}\) such that for each \(1 \leq i \leq M\), \(|S_i| = \lfloor \frac{2\Delta-1}{M}\rfloor \) or \(\lceil \frac{2\Delta-1}{M} \rceil \). This means that \(|S_i|\geq 1\) for each $i$. For each \(1 \leq i \leq M\), define \(W_i = \mathbbm{1}(\sum_{j\in S_i} Z_j \geq 1)\). It is easy to see that \(B \geq \sum_{i=1}^M W_i\). Therefore, for each \(d \in \mathbb{R}\), \(\mathbb{P}(B\geq d)\geq \mathbb{P}\left(\sum_{i=1}^M W_i\geq d\right)\). Hence, 
\begin{align}\label{NOTONEconcentrationequalitiesaregoatedimeantinequalitiesone11one11}
\sum_{d=1}^{M} \mathbb{P}(B \geq d) &\geq \sum_{d=1}^{M} \mathbb{P}\left(\sum_{i=1}^M W_i\geq d\right) = \mathbb{E}\left[\sum_{i=1}^M W_i\right] = \sum_{i=1}^{M} \mathbb{P}\left(Z_j\geq 1\text{ for some }j\in S_i\right).
\end{align}

Wlog, \(|S_1| = \lfloor \frac{2\Delta-1}{M}\rfloor\). Then:
\begin{align}\label{concentrationequalitiesaregoatedimeantinequalitiesone11one11}
\sum_{d=1}^{M} \mathbb{P}(B \geq d) \geq M\cdot \mathbb{P}\left(Z_j\geq 1\text{ for some }j\in S_1\right).
\end{align}
We will now bound \( \mathbb{P}\left(Z_j\geq 1\text{ for some }j\in S_1\right)\) by splitting into cases based on whether or not  \(2\Delta -1 \geq 2 J\).\\
\textit{\underline{Case 1:} \(2\Delta -1 \geq 2 J\).}\\
Then, \(\frac{2\Delta - 1 }{J} \geq \frac{2\Delta-1}{2J}+1 \). By a Bonferroni inequality,
\begin{align*}
 \mathbb{P}\left(Z_j\geq 1\text{ for some }j\in S_1\right) &\geq \sum_{i \in S_1} \mathbb{P}(Z_i\geq 1) - \sum_{i<j\in S_1}\mathbb{P}(Z_i \geq 1 \land Z_j \geq 1) \\
&= \left\lfloor \frac{2\Delta-1}{M}\right\rfloor p - \binom{\left\lfloor \frac{2\Delta-1}{M}\right\rfloor}{2}p^2\\
&\geq \left\lfloor \frac{2\Delta-1}{M}\right\rfloor p - \frac{1}{2}\left\lfloor \frac{2\Delta-1}{M}\right\rfloor^2p^2 \\
&= \left\lfloor \frac{2\Delta-1}{J}\right\rfloor p \left( 1 - \frac{1}{2}\left\lfloor \frac{2\Delta-1}{J}\right\rfloor p\right).
\end{align*}
Since \(\frac{2\Delta - 1 }{J} \geq \frac{2\Delta-1}{2J}+1 \) and $\left\lfloor \frac{2\Delta-1}{J}\right\rfloor\leq \frac{2\Delta}{J}$, we have:
\begin{align*}
\mathbb{P}(Z_1 \geq 1) &\geq \left\lfloor \frac{2\Delta-1}{2J}+1\right\rfloor p \left( 1 -  \frac{p\Delta}{J} \right)\\
&\geq \left( \frac{120C_{BE}\sqrt{p}(2\Delta-1)}{\sqrt{n(1-p)}}\right) \left( 1 - \frac{120C_{BE} \sqrt{p} \Delta}{\sqrt{n(1-p)}}\right).
%
\end{align*}
Since \(\Delta \leq \frac{10}{p}\), 
\begin{align*}
\frac{120C_{BE}\sqrt{p}\Delta}{\sqrt{n(1-p)}}\leq \frac{1200C_{BE}}{\sqrt{np(1-p)}} .
\end{align*}
Since \(p \geq \frac{\log n}{n}\) and \(1-p \geq \frac{3}{4}\), 
\(
\frac{1200C_{BE}}{\sqrt{np(1-p)}} \leq \frac{800\sqrt{3}C_{BE}}{\sqrt{\log n}} \leq \frac{1}{2}.
\)
Therefore, if \(2\Delta - 1 \geq 2J\),
\begin{align*}
     \mathbb{P}\left(Z_j\geq 1\text{ for some }j\in S_1\right)\geq \frac{60C_{BE}\sqrt{p}(2\Delta-1)}{\sqrt{n(1-p)}}\geq \frac{60C_{BE}\sqrt{p}\Delta}{\sqrt{n(1-p)}}.
\end{align*}
Therefore, by Equations \ref{NOTONEconcentrationequalitiesaregoatedimeantinequalitiesone11one11} and \ref{concentrationequalitiesaregoatedimeantinequalitiesone11one11},
\(
\sum_{d=1}^{M} \mathbb{P}(B \geq d) \geq J \times \frac{60C_{BE}\sqrt{p}\Delta}{\sqrt{n(1-p)}} \). Since \(\frac{3}{4}\leq 1-p\) and \(\frac{\log n }{n} \leq p\), we have that \(J\geq \frac{19}{20}\times \frac{\sqrt{np(1-p)}}{240C_{BE}} .\) Therefore, \(\sum_{d=1}^{M} \mathbb{P}(B \geq d) \geq \frac{19}{80} p \Delta.\)\\
%
%
%
%
\textit{\underline{Case 2:} \(2\Delta - 1 < 2J\).}\\
Then \(\Delta \leq J\). Since \(|S_1|\geq 1\), \(\mathbb{P}(Z_j\geq 1\text{ for some $j\in S_1$})\geq p\). Therefore, \(\sum_{d=1}^{M} \mathbb{P}(B \geq d) \geq p\min \{J,\Delta\} = p\Delta\). 
Therefore, in both cases, 
\begin{equation*}
    \sum_{d=1}^{M} \mathbb{P}(B \geq d) \geq \frac{19}{80} p \Delta.
\end{equation*}
Upon substituting this into Equation \ref{NEWLYarabarukisetsuotomedomeyorepepr}, we find:
\begin{align*}
    \sum_{d=1}^{2\Delta-1} \mathbb{P}(A+d = X_2)\mathbb{P}(B \geq d) &\geq \frac{19}{960} \times \frac{\sqrt{p}\Delta}{\sqrt{n(1-p)}}\notag
    \geq \frac{19}{960} \times \frac{\sqrt{p}\Delta}{\sqrt{n}}.
    \end{align*}

Therefore, 
\(
    \mathbb{P}(u \in C_{1,1}) \geq \dfrac{1}{2} + \dfrac{19}{1920} \times \dfrac{\sqrt{p}\Delta}{\sqrt{n}}.
\)
Since 
\(    \mathbb{E}[|C_{1,1}|] = n \mathbb{E}\left[\mathbbm{1}_{u \in C_{1,1}} \right] = n \mathbb{P}(u \in C_{1,1}),
    \)
    we have that:
    \begin{align*}
        \mathbb{E}\left[|C_{1,1}|\right] 
        &\geq \frac{n}{2} + \frac{19\sqrt{pn}\Delta}{1920},
    \end{align*}
    which completes the proof.
\end{proof}
\section{Step 2 - Day 1, Part 2}

The main goal of this section is to prove Lemma \ref{lemma tail prob of |hatR_v|}. A large part of the proof of Lemma \ref{lemma tail prob of |hatR_v|} was adapted from work by Berkowitz and Devlin \cite{berkowitz2022central}, who proved a central limit theorem for the number of vertices of colour $1$ after the first day by using the method of moments. In this step, we adapt their calculations to evaluate higher moments of $|\hat{R}_v|$. 
\begin{lemma}\label{lemma tail prob of |hatR_v|} 
Let $D,\EEEEE>0 $ be arbitrary constants. Then, there is a constant $C_{D,\EEEEE}>0 $ such that if $n$ is sufficiently large, $p\leq \frac{1}{4}$ and \(\Delta \geq \dfrac{1}{\sqrt{p}} \exp\left[C_{D,\mathcal{D}}\sqrt{\log \left(\dfrac{1}{p}\right)}\right]\), then
    \begin{align*}
    \mathbb{P}\left(\Big||\hat{R}_v| - \mathbb{E}\left[|\hat{R}_v|\right]\Big| \geq D \sqrt{pn}\Delta\right) \leq \frac{1}{\sqrt{(n-1)p(1-p)}} + \EEEEE{p} \Delta.
\end{align*}
\end{lemma}
In order to prove Lemma \ref{lemma tail prob of |hatR_v|}, we will prove Lemma \ref{lemma tail prob of biased maj dyn} and show that it easily implies Lemma \ref{lemma tail prob of |hatR_v|}. The statement of  Lemma \ref{lemma tail prob of biased maj dyn} requires the following definition:

\begin{definition}
    (Biased Majority Dynamics). Suppose we are given a simple undirected graph \(G = (V,E)\), and a partition \(V = C_1 \cup C_2\). The biased majority dynamics process is a process that begins on day 0, where each vertex in \(C_i\) is coloured with colour $i$ (for $i=1,2$). As in the non-biased case, we define \(C_{i,t}\) to be the set of vertices that are colour \(i\) on day \(t\).   
    For each vertex $v\in V$ and each time $t+1$, if $|\Gamma(v)\cap C_{1,t}|> |\Gamma(v)\cap C_{2,t}| - 1 $, then $v \in C_{1,t+1}$. If $|\Gamma(v)\cap C_{1,t}| < |\Gamma(v)\cap C_{2,t}| - 1 $, then $v \in C_{2,t+1}$. If $|\Gamma(v)\cap C_{1,t}| = |\Gamma(v)\cap C_{2,t}| - 1 $, then $v$ remains the same colour. 
    
    As such, in non-biased majority dynamics, a vertex $v$ is becomes colour $1$ on day \break$t+1$ if and only if $|\Gamma(v)\cap C_{1,t}| - |\Gamma(v)\cap C_{2,t}| + \frac{1}{2} \mathbbm{1}_{v \in C_{1,t}} -  \frac{1}{2} \mathbbm{1}_{v \in C_{2,t}}\geq 0$. However, in biased \break majority dynamics,  a vertex $v$ is becomes colour $1$ on day $t+1$ if and only if\break ${|\Gamma(v)\cap C_{1,t}| - |\Gamma(v)\cap C_{2,t}| + \frac{1}{2} \mathbbm{1}_{v \in C_{1,t}} -  \frac{1}{2} \mathbbm{1}_{v \in C_{2,t}}\geq -1}$. 
\end{definition}
When considering the biased majority dynamics process, we will not assume that\break ${|C_{1,0}|\geq |C_{2,0}|}$. We will still denote \(\Delta = \Big| |C_{1,0}| - |C_{2,0}|\Big| /2\). 
We will now state Lemma \ref{lemma tail prob of biased maj dyn}.

\begin{lemma}\label{lemma tail prob of biased maj dyn} Consider the biased majority dynamics process. Let $D,\EEEEE>0 $ be arbitrary constants. Then, there is a constant $\widehat{C}_{D,\EEEEE}>0 $ such that whenever \(\Delta \geq \dfrac{1}{\sqrt{p}} \exp\left[\widehat{C}_{D,\mathcal{D}}\sqrt{\log \left(\dfrac{1}{p}\right)}\right]\), \(p\leq \frac{1}{4}\) and $n$ is sufficiently large, then
\begin{align*}
    \mathbb{P}\left(\Big||C_{1,1}| - \mathbb{E}\left[|C_{1,1}|\right]\Big| \geq D \sqrt{pn}\Delta\right) \leq \frac{1}{\sqrt{np(1-p)}} + \EEEEE{p}\Delta.
\end{align*}
\end{lemma}
We now present the proof of Lemma \ref{lemma tail prob of |hatR_v|}, assuming Lemma \ref{lemma tail prob of biased maj dyn}.
\begin{proof}[Proof of Lemma \ref{lemma tail prob of |hatR_v|}]
Consider a graph $G = (V,E)$, and let $v\in V$. First, suppose $v$ is colour 1 on day 0. Then, for any $u \in V\backslash\{v\}$, $u\in \hat{R}_v$ if and only if $u$ is colour $1$ on day 1 of the biased majority dynamics process on $G\backslash\{v\}$. Therefore, Lemma 
\ref{lemma tail prob of biased maj dyn} gives: For constants $D,\mathcal{D}>0$, if 
\(\Delta \geq\dfrac{1}{\sqrt{p}} \exp\left[\widehat{C}_{D,\mathcal{D}}\sqrt{\log \left(\dfrac{1}{p}\right)}\right]+1\), then:
\begin{align*}
    \mathbb{P}\left(\Big||\hat{R}_v| - \mathbb{E}\left[|\hat{R}_v|\right]\Big| \geq D \sqrt{p(n-1)}(\Delta-1)\right) \leq \frac{1}{\sqrt{(n-1)p(1-p)}} +\EEEEE{p}(\Delta-1),
\end{align*}
where we make note of the following two facts:
\begin{itemize}
    \item the value of $\Delta $ for $G$ is one more than the value of $\Delta$ for $G\backslash\{v\}$
    \item \(|G\backslash\{v\}| = n-1\).
\end{itemize}
After swapping the roles of colours 1 and 2, the same argument shows that when $v$ is colour $2$ on day $0$:

\begin{align}
    \mathbb{P}\left(\Big||\hat{R}_v| - \mathbb{E}\left[|\hat{R}_v|\right]\Big| \geq D \sqrt{p(n-1)}(\Delta+1)\right) \leq \frac{1}{\sqrt{(n-1)p(1-p)}} + \EEEEE{p}(\Delta+1),\notag
\end{align}
provided \(\Delta \geq\dfrac{1}{\sqrt{p}} \exp\left[\widehat{C}_{D,\mathcal{D}}\sqrt{\log \left(\dfrac{1}{p}\right)}\right]-1\). Therefore, in both cases, if \(\Delta \geq \) \break\(\dfrac{1}{\sqrt{p}} \exp\left[(\widehat{C}_{D,\mathcal{D}}+1)\sqrt{\log \left(\dfrac{1}{p}\right)}\right]\), then:
\begin{align}
    \mathbb{P}\left(\Big||\hat{R}_v| - \mathbb{E}\left[|\hat{R}_v|\right]\Big| \geq 2D \sqrt{pn}\Delta\right) \leq \frac{1}{\sqrt{(n-1)p(1-p)}} + 2\EEEEE {p}\Delta.\notag
\end{align}
Setting \(C_{D,\mathcal{D}} =\widehat{C}_{D/2,\mathcal{D}/2}+1\) completes the proof.

\end{proof}
In order to prove Lemma \ref{lemma tail prob of biased maj dyn}, we will bound $\mathbb{E}\Big[\big(|C_{1,1}| - \mathbb{E}\left[|C_{1,1}|\right]\big)^k\Big]$ when $k$ is even. Then, we will apply Markov's inequality to \(\big(|C_{1,1}| - \mathbb{E}\left[|C_{1,1}|\right]\big)^k\). As such, the first step is to prove the following lemma:

\begin{lemma}\label{moment_lemma}
Suppose $k\leq \sqrt{\frac{n}{20}}$ is even. Then,
\begin{equation} 
   \mathbb{E}\Big[\big(|C_{1,1}| - \mathbb{E}\left[|C_{1,1}|\right]\big)^k\Big] \leq  \mathcal{O}(1)^{k^2} \times  \mathcal{O}\left(n^{k/2}\left(\frac{\Delta}{n} + \frac{1}{\sigma}\right)\right) + \frac{(k-1)!!}{4} \cdot n^{k/2},\notag
\end{equation}

where
    \(
    \sigma = \sqrt{np(1-p)}
    \) and the implied constants are independent of both $k$ and $n$.
\end{lemma}
\begin{proof}[Proof of Lemma \ref{moment_lemma}]

As before, for each $v \in V$, let \(
    L(v) := \begin{cases} 
      1 & \text{if }v\in C_{1,0}, \\
      -1 & \text{if }v \in C_{2,0}. 
   \end{cases}
\)

Define:
\begin{equation*}
 \mu_i = 
    \begin{cases}
  2\mathbb{P}\left( \mathrm{Bin}(|C_{1,0}|-1,p) \geq \mathrm{Bin}(|C_{2,0}|,p)-1\right)-1 & \text{if } i=1,\\
  2\mathbb{P}\left( \mathrm{Bin}(|C_{2,0}|-1,p) \geq \mathrm{Bin}(|C_{1,0}|,p)-1\right) -1& \text{if } i=2.
    \end{cases}
\end{equation*}
Furthermore, denote \[
\mu_v := \begin{cases}
  \mu_1 & \text{if } v\in C_{1,0},\\
  \mu_2 & \text{if } v\in C_{2,0}.
    \end{cases}
\]
Let $Z_v$ denote the following centered indicator:
\begin{equation*}
   Z_v:= \begin{cases}
  1-\mu_v & \text{if } v\in C_{1,0}\cap C_{1,1} \text{ or } v\in C_{2,0}\cap C_{2,1},\\ 
  -1-\mu_v & \text{otherwise.}
    \end{cases}
\end{equation*}
By construction, \(\mathbb{E}\left[Z_v\right] = 0\). Finally, define $Z$ as
\begin{align*}
    Z = \sum_{v_1 \in C_{1,0}} Z_{v_1} - \sum_{v_2 \in C_{2,0}} Z_{v_2}  = \sum_{v \in V} L(v) Z_v =& 2|C_{1,1}| - n - \mu_{1}|C_{1,0}| + \mu_{2}|C_{2,0}|\\
    =& 2 |C_{1,1}| - 2\mathbb{E}\left[|C_{1,1}|\right].
\end{align*}
Clearly, Proposition \ref{oneinCLTpaper} implies Lemma \ref{moment_lemma}.
\begin{proposition}\label{oneinCLTpaper}
\begin{equation}
   \mathbb{E}[Z^k] \leq  \mathcal{O}(1)^{k^2} \times  \mathcal{O}\left(n^{k/2}\left(\frac{\Delta}{n} + \frac{1}{\sigma}\right)\right) + (k-1)!! \cdot n^{k/2}.\notag
\end{equation}
\end{proposition}

The proof of Proposition \ref{oneinCLTpaper} utilizes standard techniques from $p$-biased Fourier analysis. We begin by providing a concise overview and establishing the necessary framework; for a more detailed exposition, refer to O'Donnell's authoritative text \cite{o2014analysis}. 

Define $\mathcal{E} = \binom{V}{2}$, and consider the space of functions that map $\{-1,1\}^{\mathcal{E}}$ to $\mathbb{R}$. Each element $\vec{x} \in \{-1,1\}^{\mathcal{E}}$ can be viewed as representing a graph where the edge set is given by $\{e : x_e = 1\}$. We endow $\{-1,1\}^{\mathcal{E}}$ with the product measure, where each coordinate is independently equal to $1$ with probability $p$. This setup naturally gives rise to random variables that depend on $G \sim G(|V|, p)$ and functions of the form $f : \{-1,1\}^{\mathcal{E}} \to \mathbb{R}$.

For each $S\subseteq \mathcal{E}$, define the function $\Phi_S (\vec{x}) : \{-1,1\}^{\mathcal{E}} \to \mathbb{R}$ as:
\begin{equation}
   \Phi_S (\vec{x}) = \prod_{e \in S} \frac{x_e + 1 - 2p}{2\sqrt{p(1-p)}},
    \notag
\end{equation}
with $\Phi_\emptyset = 1$. It is easy to check that for all $S,T\subset \mathcal{E}$, we have:
\begin{equation*}
    \mathbb{E}\left[\Phi_S(\vec{x})\Phi_T(\vec{x})\right] = \begin{cases} 
      1 & \text{if }S=T, \\
      0 & \text{otherwise}. 
   \end{cases}
\end{equation*}
Therefore, \(\{\Phi_S : S\subset \mathcal{E}\}\) is an orthonormal basis for this space (with inner product \(\langle f,g\rangle = \mathbb{E}\left[fg\right]\)), and each $f$ can be written uniquely as:
\begin{equation*}
    f(\vec{x}) = \sum_{S \subseteq \mathcal{E}} \hat{f}(S) \Phi_S(\vec{x}), \text{          where          } \hat{f}(S) = \mathbb{E} \left[ f(\vec{x}) \Phi_S(\vec{x})\right].
\end{equation*}
We will need the following definition and technical lemma.
\begin{definition}
    For $u\in V$, define \[\Gamma_u := \left\{\{u,v\} : v \in V\backslash\{u\}\right\}\subseteq \mathcal{E}.\]
\end{definition}

\begin{lemma}\label{CLTPaperLemma10} Suppose $k \leq \sqrt{\frac{n}{20}}$. With notation as before, 
\begin{enumerate}[label=(\roman*)]
    \item We have \(\mathbb{E}\left[Z_v^2\right] =1 - \mu_v^2 \)
    \item If $S \not \subseteq \Gamma_v$, then $\widehat{Z_v}(S) = 0$. Moreover, \(\widehat{Z_v}(\emptyset) = 0\).
    \item For $|S|\leq 10k^2$, \(\widehat{Z_v}(S) = \mathcal{O}\left(\frac{1}{\sqrt{n}}\right)\).
    \item For $2\leq |S|\leq 10k^2$, \( \widehat{Z_v}(S) = \mathcal{O} \left(\frac{1}{n}\right).\)
    \item Suppose $v,u_1,u_2$ are distinct vertices with $v_1\in C_{1,0}$ and $v_2\in C_{2,0}$. Then, \[{\sum_{e \in E}} \Big(\widehat{Z_{u_1}} (e) - \widehat{Z_{u_2}}(e) \Big) \widehat{Z_v} (e) = {\mathcal{O}\left(\frac{1}{n\sigma}\right)}.\]
    \item If \(S\neq \emptyset\) and $L\geq 1$, then ${\Big| \widehat{Z_v^L}(S)\Big| \leq 2^L \Big| \widehat{Z_v}(S)\Big| \leq 2^{L+1}}$.
    \end{enumerate}
\end{lemma}
The proof to Lemma \ref{CLTPaperLemma10} can be found in Appendix \ref{appendix_B_2_colours}. 

We now present the proof of Proposition \ref{oneinCLTpaper}, where we estimate \(\mathbb{E}\left[ Z^k\right]\) using Lemma \ref{CLTPaperLemma10}.
\begin{proof}[Proof of Proposition \ref{oneinCLTpaper}]
\begin{definition}
For each tuple $\vec{w} = (w_1, \ldots, w_k) \in V^k$ consisting of (not necessarily distinct) vertices, set $Z_{\vec{w}} = \prod_{i=1} ^{k} Z_{w_i}$ and $\varepsilonghghghghg(\vec{w}) = \prod_{i=1} ^{k} \varepsilonghghghghg(w_i)$.  In this notation, we have ${\mathbb{E}[Z^k] = \sum_{\vec{w} \in V^k} \varepsilonghghghghg(\vec{w}) \mathbb{E}[Z_{\vec{w}}]}$.
\end{definition}
\begin{definition}
 For each $c \leq k$, let $\Lambda_{c} \subseteq V^k$ denote the strings in which there are exactly $c$ vertices that appear exactly once.
\end{definition}
This means that we can express:
\[
\mathbb{E}[Z^k] = \sum_{c\leq k}\sum_{\vec{w} \in \Lambda_c} \varepsilonghghghghg(\vec{w}) \mathbb{E}[Z_{\vec{w}}].
\]
To approximate \( \mathbb{E}[Z^k] \), we will make use of the above expression. We will establish that the term corresponding to \( \Lambda_0 \) is the primary contributor, and that the contribution from $\Lambda_c$ is small. To substantiate this, we will prove the following two key results:
\[
\sum_{\vec{w} \in \Lambda_0} \varepsilonghghghghg(\vec{w}) \mathbb{E}[Z_{\vec{w}}] \leq \mathcal{O}\left(n^{(k-1)/2}\right) \times \mathcal{O}(1)^{k^2} + (k-1)!! \cdot n^{k/2}.
\]
Furthermore, we'll show that for all $c \geq 1$, $$\sum_{\vec{w} \in \Lambda_c} \varepsilonghghghghg(\vec{w}) \mathbb{E}[Z_{\vec{w}}] =  \mathcal{O}(1)^{k^2} \times  \mathcal{O}\left(n^{k/2}\left(\frac{\Delta}{n} + \frac{1}{\sqrt{n}} + \frac{1}{\sigma}\right)\right).$$
After combining these equations, we find:
\[
\mathbb{E}[Z^k] \leq  \mathcal{O}(1)^{k^2} \times  \mathcal{O}\left(n^{k/2}\left(\frac{\Delta}{n} + \frac{1}{\sqrt{n}} + \frac{1}{\sigma}\right)\right) + (k-1)!! \cdot n^{k/2}
\]
Since \(\sigma \leq \sqrt{n}\), we have:
\[
\mathbb{E}[Z^k] \leq  \mathcal{O}(1)^{k^2} \times  \mathcal{O}\left(n^{k/2}\left(\frac{\Delta}{n} + \frac{1}{\sigma}\right)\right) + (k-1)!! \cdot n^{k/2},
\]
which is exactly the statement of Proposition \ref{oneinCLTpaper}. We proceed by expanding $\mathbb{E}\left[ Z_{\vec{w}}\right]$. Then, using this expansion, we will find bounds on \(\sum_{\vec{w} \in \Lambda_c} \varepsilonghghghghg(\vec{w}) \mathbb{E}[Z_{\vec{w}}] \) for $c=0$ and $c\geq 1$.

\subsection*{Expansion of \(\mathbb{E}\left[ Z_{\vec{w}}\right]\)}
Suppose $T$ and $U$ are disjoint sets of vertices such that $|T| + |U| \leq k$, and suppose for each $u \in U$, there is an integer $\alpha_u \in [2,k]$.  Then we have
\[
\mathbb{E}\left[ \prod_{t \in T} Z_{t} \cdot \prod_{u \in U} Z_{u} ^{\alpha_u} \right] = \sum_{(S_v)_{v \in T \cup U}} \prod_{t \in T} \widehat{\ Z_{t} \ }(S_t) \prod_{u \in U} \widehat{\ Z_{u} ^{\alpha_u} \ }(S_u) \cdot \mathbb{E}\left[ \prod_{v \in T \cup U} \Phi_{S_v} \right],
\]
where the sum is taken over all choices of $(S_v)_{v \in T \cup U}$ in $\mathcal{E}^{|T| + |U|}$ (so that each $v \in T \cup U$ is independently given a subset of $\mathcal{E}$).

Let $\mathcal{H} \subseteq \mathcal{E}$ denote the collection of $2$-element subsets formed from $T \cup U$, and assume that $(S_v)_{v \in T \cup U}$ corresponds to a non-zero term in the above summation. By Lemma \ref{CLTPaperLemma10}(ii), for each $v$, it follows that $S_v \subseteq \Gamma_v$ (otherwise one of the Fourier coefficients would vanish). Furthermore, every edge in $\bigcup_{u} S_u$ must appear in at least two of the sets $S_v$, as otherwise the expectation of the $\Phi$ functions would be zero. These conditions imply that each edge in $\bigcup_{u} S_u$ is included in precisely two sets, that $\bigcup_{u} S_u \subseteq \mathcal{H}$, and that for every $v$, we have $S_v = \left[ \bigcup_{u} S_u \right] \cap \Gamma_v$.

Thus, the above summation becomes
\[
\mathbb{E}\left[ \prod_{t \in T} Z_{t} \cdot \prod_{u \in U} Z_{u} ^{\alpha_u} \right] = \sum_{H \subseteq \mathcal{H}} \prod_{t \in T} \widehat{\ Z_{t} \ }(H \cap \Gamma_t) \prod_{u \in U} \widehat{\ Z_{u} ^{\alpha_u} \ }(H \cap \Gamma_u).
\]

Since $|\mathcal{H}| \leq 10 k^2$, Lemma \ref{CLTPaperLemma10}(iii) implies that each term $\widehat{Z_t}(H \cap \Gamma_t)$ is $\mathcal{O}(1/\sqrt{n})$.  Hence, if $H \subseteq \mathcal{H}$ is fixed, we have
\begin{equation}
\prod_{t \in T} \widehat{\ Z_{t} \ }(H \cap \Gamma_t) \prod_{u \in U} \widehat{\ Z_{u} ^{\alpha_u} \ }(H \cap \Gamma_u) = \mathcal{O}\left(\dfrac{1}{\sqrt{n}} \right)^{|T|} \prod_{u \in U} \widehat{\ Z_{u} ^{\alpha_u} \ }(H \cap \Gamma_u). \notag
\end{equation}
By Lemma \ref{CLTPaperLemma10}(vi),  $\big|\widehat{\ Z_{u} ^{\alpha_u} \ }(H \cap \Gamma_u)\big| \leq 2^{\alpha_u + 1}\leq 2^{k+1}\leq 4^k$. Therefore, we have:
\begin{align}\label{crude term-wise Fourier bound}
    \prod_{t \in T} \widehat{\ Z_{t} \ }(H \cap \Gamma_t) \prod_{u \in U} \widehat{\ Z_{u} ^{\alpha_u} \ }(H \cap \Gamma_u)& = \mathcal{O}\left(\dfrac{1}{\sqrt{n}} \right)^{|T|} \times \mathcal{O}\left(1\right)^{k^2}\notag\\
    &= \mathcal{O}\left(n^{-|T|/2}\right)\times \mathcal{O}\left(1\right)^{|T|} \times \mathcal{O}\left(1\right)^{k^2}\notag\\
    & = \mathcal{O}\left(n^{-|T|/2}\right)\times \mathcal{O}\left(1\right)^{k^2},
\end{align}
where the final line follows because $|T|\leq k$. 
If there exists some $u \in U$ such that $H \cap \Gamma_u \neq \emptyset$, then by applying part (vi) [and (iii)] of Lemma \ref{CLTPaperLemma10}, we can conclude that the above expression is bounded by $\mathcal{O}(n^{-(|T|+1)/2}) \times \mathcal{O}(1)^{k^2}$. Likewise, if there is a vertex $t \in T$ such that $|H \cap \Gamma_t| \neq 1$, another application of part (iv) of the same lemma allows us to deduce that the expression is $\mathcal{O}(n^{-(|T|+1)/2}) \times \mathcal{O}(1)^{k^2}$.
Together, since there are at most $2^{|\mathcal{H}|} = \mathcal{O}(1)^{k^2}$ choices for $H$, this implies
\begin{equation}\label{Fourier expansion equation}
\mathbb{E}\left[ \prod_{t \in T} Z_{t} \cdot \prod_{u \in U} Z_{u} ^{\alpha_u} \right] = \mathcal{O}(n^{-(|T|+1)/2}) \times \mathcal{O}\left(1\right)^{k^2}+ \left[ \prod_{u \in U} \widehat{\ Z_{u} ^{\alpha_u} \ }(\emptyset) \right]\cdot  \sum_{H \in \mathcal{M}} \prod_{t \in T} \widehat{Z_t}(H \cap \Gamma_t),
\end{equation}
where $\mathcal{M}$ is the set of all perfect matchings on the vertex set $T$. We now move on by bounding \(\sum_{\vec{w} \in \Lambda_0} \varepsilonghghghghg(\vec{w}) \mathbb{E}[Z_{\vec{w}}] \).
\subsection*{Contribution from \(\Lambda_0\)}

Let $\Lambda_0 ^{+} \subseteq \Lambda_0$ denote the set of strings where each vertex appears at most twice.  Then, since each term of the summation is at most $\mathcal{O}(1)$, we have
\[
\sum_{\vec{w} \in \Lambda_0} \varepsilonghghghghg(\vec{w}) \mathbb{E}[Z_{\vec{w}}] = \sum_{\vec{w} \in \Lambda_0 ^+} \mathbb{E}[Z_{\vec{w}}] + \mathcal{O}\left(|\Lambda_0 \setminus \Lambda_0 ^+|\right) .
\]
Every string in $\Lambda_0 \setminus \Lambda_0^+ \subseteq V^k$ contains at least one element that appears more than twice, with no element appearing exactly once. As a result, the number of possible ways of choosing which vertices appear in a string is bounded by $\mathcal{O}\left(n^{(k-1)/2}\right)$. Once the set of vertices is chosen, there are at most $\mathcal{O}\left(k^k\right)$ possible strings that can be formed. Therefore, the total number of such strings is at most at most $\mathcal{O}\left(k^k\right)$. Therefore,  $|\Lambda_0 \setminus \Lambda_0 ^+| = \mathcal{O}\left(k^kn^{(k-1)/2}\right)$.

Therefore, \[
\sum_{\vec{w} \in \Lambda_0} \varepsilonghghghghg(\vec{w}) \mathbb{E}[Z_{\vec{w}}] =  \sum_{\vec{w} \in \Lambda_0 ^+} \mathbb{E}[Z_{\vec{w}}] + \mathcal{O}\left(k^kn^{(k-1)/2}\right).
\]

Since $k$ is even, Equation \ref{Fourier expansion equation} and Lemma \ref{CLTPaperLemma10}(i) imply that

\begingroup
\allowdisplaybreaks
\begin{eqnarray*}
\sum_{\vec{w} \in \Lambda_0 ^+} \mathbb{E}[Z_{\vec{w}}] &=& \sum_{U \subseteq V,\ \ |U|=k/2} \dfrac{k!}{2^{k/2}}\mathbb{E}\left[ \prod_{u \in U} Z_{u} ^{2} \right]\\
&=&  \sum_{U \subseteq V,\ \ |U|=k/2} \dfrac{k!}{2^{k/2}} \bigg[\prod_{u \in U} \widehat{\ Z_{u} ^{2}\ } (\emptyset) + \mathcal{O}\left(n^{-1/2}\right)\times \mathcal{O}(1)^{k^2}\bigg]\\
&=& \mathcal{O}\left(n^{-1/2}\right)\times \mathcal{O}(1)^{k^2} \times \binom{n}{k/2} \frac{k!}{2^{k/2}} + \sum_{U \subseteq V,\ \ |U|=k/2} \dfrac{k!}{2^{k/2}} \prod_{u \in U} \widehat{\ Z_{u} ^{2}\ } (\emptyset)\\
&=& \mathcal{O}\left(n^{(k-1)/2}\right)\times \mathcal{O}(1)^{k^2} + \frac{k!}{2^{k/2}} \sum_{j=0}^{k/2} {\binom{|C_{1,0}| }{ j}} \binom{|C_{2,0}| }{ {k/2 -j}} (1-\mu_1 ^2)^{j} (1-\mu_2 ^2)^{k/2 -j}.
\end{eqnarray*}
\endgroup
Since \({\binom{|C_{1,0}| }{ j}}\leq \frac{|C_{1,0}|^j}{j!}\) and $\binom{|C_{2,0}| }{ {k/2 -j}}\leq \frac{|C_{2,0}|^{k/2-j}}{(k/2-j)!}$, we have:
\begingroup
\allowdisplaybreaks
\begin{eqnarray*}
\sum_{\vec{w} \in \Lambda_0 ^+} \mathbb{E}[Z_{\vec{w}}] &\leq&\mathcal{O}\left(n^{(k-1)/2}\right)\times \mathcal{O}(1)^{k^2}\\
&\quad&+ \frac{k!}{2^{k/2}}\bigg(\frac{1}{({k/2})!}\bigg)\sum_{j=0}^{k/2} \binom{k/2}{j} |C_{1,0}|^j(1-\mu_1 ^2)^{j}|C_{2,0}|^{k/2-j}(1-\mu_2 ^2)^{k/2 -j}\\
&=& \mathcal{O}\left(n^{(k-1)/2}\right)\times \mathcal{O}(1)^{k^2} + \dfrac{k!}{2^{k/2}} \left( \dfrac{1}{(k/2)!} \Big[ |C_{1,0}| (1-\mu_1 ^2) + |C_{2,0}| (1- \mu_2 ^2) \Big]^{k/2}  \right)\\
&=& \mathcal{O}\left(n^{(k-1)/2}\right) \times \mathcal{O}(1)^{k^2} + (k-1)!! \cdot \Big[ |C_{1,0}| (1-\mu_1 ^2) + |C_{2,0}| (1- \mu_2 ^2) \Big]^{k/2}\\
&\leq& \mathcal{O}\left(n^{(k-1)/2}\right) \times \mathcal{O}(1)^{k^2} + (k-1)!! \cdot n^{k/2}.
\end{eqnarray*}
\endgroup

\subsection*{Contribution from \(\Lambda_c\) for \(c \geq 1\)}

Define \(m := \min \{|C_{1,0}|,|C_{2,0}|\}\). We can write \(C_{1,0} = \{r_1,r_2,\dots,r_m\} \cup X_1\) and \(C_{2,0} = \{b_1,b_2,\dots,b_m\}\cup X_2\). It is easy to see that \(|X_i| = |C_{i,0} | - m\), so either \(X_1\) or \(X_2\) is empty. 

\begin{definition}
 For each $c \geq 1$, let $\Lambda_c ^+$ denote the strings $\vec{w} \in \Lambda_c$ such that:
 \begin{enumerate}
 \item[(a)] $\vec{w}$ contains no elements of $X_1\cup X_2$.
 \item[(b)] $\vec{w}$ contains no elements more than twice.
 \item[(c)] For each $i$, if $r_i$ appears in $\vec{w}$, then $b_i$ does not. 
 \end{enumerate}
 \end{definition}
  Therefore, by Equation \ref{crude term-wise Fourier bound}, we have
\begin{eqnarray*}
\sum_{\vec{w} \in \Lambda_c \setminus \Lambda_c ^+} \varepsilonghghghghg(\vec{w}) \mathbb{E}[Z_{\vec{w}}] &=& |\Lambda_c \setminus \Lambda_c ^+| \cdot \mathcal{O}(n^{-c/2}) = \mathcal{O}\left(n^{k/2 - 1} (|X_1\cup X_2| +\sqrt{n}) k!\right),
\end{eqnarray*}
which follows since
\[
|\Lambda_c \setminus \Lambda_c ^+| = \mathcal{O}\left(n^{(k+c)/2 -1} |X_1\cup X_2| k! \right) + \mathcal{O}\left(n^{(k+c-1)/2} k! \right) + \mathcal{O}\left(n^{(k+c)/2 -1} k! \right).
\]

Moreover, we have 
\begin{equation}\label{fjoiwhgiogeroncowzxnvlncm,nvcvm,zxcv,cxv}
\sum_{\vec{w} \in \Lambda_c ^+} \varepsilonghghghghg(\vec{w}) \mathbb{E}[Z_{\vec{w}}] = \dfrac{k!}{2^{(k-c)/2}} \sum_{T \in \binom{V}{c}} \sum_{U \in \binom{V\setminus T}{(k-c)/2}} \mathbb{E}\left[ \prod_{t \in T} (Z_{r_t} - Z_{b_t}) \prod_{u \in U} (Z_{r_u} ^2 + Z_{b_u} ^2) \right].
\end{equation}

For fixed $U$ and $T$, consider \(\mathbb{E}\left[ \prod_{t \in T} (Z_{r_t} - Z_{b_t}) \prod_{u \in U} (Z_{r_u} ^2 + Z_{b_u} ^2) \right]\). After some manipulations as in \eqref{Fourier expansion equation}
, we see

\begin{multline}\label{fjoingiorngoeriojoiijjioeogiojrefjjf}
\mathbb{E}\left[ \prod_{t \in T} (Z_{r_t} - Z_{b_t}) \prod_{u \in U} (Z_{r_u} ^2 + Z_{b_u} ^2) \right] = \mathcal{O}\left(n^{-(|T|+1)/2} \right)\times \mathcal{O}\left(1\right)^{k^2}\\
+ \left[ \prod_{u \in U} \widehat{\ Z_{r_u} ^{2} \ }(\emptyset) + \widehat{\ Z_{b_u} ^{2} \ }(\emptyset)\right]\cdot \sum_{H \in \mathcal{M}} \prod_{t \in T} [\widehat{Z_{r_t}}(H \cap \Gamma_{r_t}) - \widehat{Z_{b_t}} (H \cap \Gamma_{b_t})] ,
\end{multline}
where $\mathcal{M}$ is the set of graphs $H$ on the vertex set $\bigcup_{t \in T} \{r_t, b_t\}$ such that for all $t \in T$, we have $r_t \not \sim b_t$ and there is exactly one edge of $H$ meeting $\{r_t, b_t\}$. Therefore,
\begin{align*}
\mathbb{E}\left[ \prod_{t \in T} (Z_{r_t} - Z_{b_t}) \prod_{u \in U} (Z_{r_u} ^2 + Z_{b_u} ^2) \right] & = \mathcal{O}\left(n^{-(|T|+1)/2} \right)\times \mathcal{O}\left(1\right)^{k^2}\\
& \quad+ \mathcal{O}\left(1\right)^k\cdot \sum_{H \in \mathcal{M}} \prod_{t \in T} [\widehat{Z_{r_t}}(H \cap \Gamma_{r_t}) - \widehat{Z_{b_t}} (H \cap \Gamma_{b_t})] .
\end{align*}

For distinct $i,j \in T$, let $\mathcal{M}_{i,j} \subseteq \mathcal{M}$ denote the subset of graphs containing one of the edges in $\{r_i, b_i\} \times \{r_j, b_j\}$.  Then we have
\begin{eqnarray*}
\sum_{H \in \mathcal{M}} \prod_{t \in T} [\widehat{Z_{r_t}}(H \cap \Gamma_{r_t}) - \widehat{Z_{b_t}} (H \cap \Gamma_{b_t})]  &=& \dfrac{1}{|T|} \sum_{i,j} \sum_{H \in \mathcal{M}_{i,j}} \prod_{t \in T} [\widehat{Z_{r_t}}(H \cap \Gamma_{r_t}) - \widehat{Z_{b_t}} (H \cap \Gamma_{b_t})]\\
&=& \dfrac{1}{|T|} \sum_{i,j} \sum_{H \in \mathcal{M}_{i,j}} \prod_{t \notin \{i,j\}} [\widehat{Z_{r_t}}(H \cap \Gamma_{r_t}) - \widehat{Z_{b_t}} (H \cap \Gamma_{b_t})]\\
& &  \times \left(\dfrac{1}{4} \sum_{e} \Big[\widehat{Z_{r_i}}(e) - \widehat{Z_{b_i}}(e) \Big] \cdot \Big[\widehat{Z_{r_j}}(e) - \widehat{Z_{b_j}}(e) \Big] \right).
\end{eqnarray*}
By Lemma \ref{CLTPaperLemma10}(iii), \( \Big|\widehat{Z_{r_t}}(H \cap \Gamma_{r_t}) - \widehat{Z_{b_t}} (H \cap \Gamma_{b_t}) \Big| = \mathcal{O}\left(\frac{1}{\sqrt{n}}\right)\). Therefore,
\begin{align*}
    \dfrac{1}{|T|} \sum_{i,j} \sum_{H \in \mathcal{M}_{i,j}} \prod_{t \notin \{i,j\}} [\widehat{Z_{r_t}}(H \cap \Gamma_{r_t}) - \widehat{Z_{b_t}} (H \cap \Gamma_{b_t})] &= \dfrac{1}{|T|} \sum_{i,j} \sum_{H \in \mathcal{M}_{i,j}} \prod_{t \notin \{i,j\}} \mathcal{O}\left(n^{1 - |T|/2}\right)\\
    &= \mathcal{O}\left(|\mathcal{M}|n^{1 - |T|/2}\right).
\end{align*}
Since \(|\mathcal{M}| = 2^{{|T|}} {\times} \) (the number of perfect matchings on \(\{ \{r_t,b_t\} : t \in T \}\)), we must have \(|\mathcal{M}| \leq 2^{{|T|}} {\times} (|T|-1)!! \leq \mathcal{O}(1)^{k^2}\). Furthermore, by Lemma \ref{CLTPaperLemma10}(v), 
\[
\dfrac{1}{4} \sum_{e} \Big[\widehat{Z_{r_i}}(e) - \widehat{Z_{b_i}}(e) \Big] \cdot \Big[\widehat{Z_{r_j}}(e) - \widehat{Z_{b_j}}(e) \Big] = \mathcal{O}\left(\frac{1}{n\sigma}\right).
\]
Therefore, combining everything, we find:
\begin{eqnarray*}
\sum_{H \in \mathcal{M}} \prod_{t \in T} [\widehat{Z_{r_t}}(H \cap \Gamma_{r_t}) - \widehat{Z_{b_t}} (H \cap \Gamma_{b_t})] &=& \mathcal{O}\left(1\right)^{k^2} \times \mathcal{{O}}\left(n^{-|T|/2}/\sigma\right).
\end{eqnarray*}
Upon substituting this into Equation \ref{fjoingiorngoeriojoiijjioeogiojrefjjf}, we find:
\[
\mathbb{E}\left[ \prod_{t \in T} (Z_{r_t} - Z_{b_t}) \prod_{u \in U} (Z_{r_u} ^2 + Z_{b_u} ^2) \right] = \mathcal{O}\left(n^{-|T|/2} /\sigma\right)\times \mathcal{O}\left(1\right)^{k^2} .
\]
Therefore, by Equation \ref{fjoiwhgiogeroncowzxnvlncm,nvcvm,zxcv,cxv},
\begin{align*}
\sum_{\vec{w} \in \Lambda_c ^+} \varepsilonghghghghg(\vec{w}) \mathbb{E}[Z_{\vec{w}}] & = \dfrac{k!}{2^{(k-c)/2}} \binom{n}{c} \binom{n-c}{(k-c)/2}\times  \mathcal{O}\left(n^{-c/2} /\sigma\right)\times \mathcal{O}\left(1\right)^{k^2}\\
& = \mathcal{O}(1)^{k^2} \times \mathcal{O}\left(n^{k/2}\right).
\end{align*}

Therefore, we have
\begin{align*}
\sum_{\vec{w} \in \Lambda_c} \varepsilonghghghghg(\vec{w}) \mathbb{E}[Z_{\vec{w}}] &= \sum_{\vec{w} \in \Lambda_c \setminus \Lambda_c ^{+}} \varepsilonghghghghg(\vec{w}) \mathbb{E}[Z_{\vec{w}}] + \sum_{\vec{w} \in \Lambda_c ^+} \varepsilonghghghghg(\vec{w}) \mathbb{E}[Z_{\vec{w}}] \\
&= \mathcal{O}\left(n^{k/2 - 1} \left(|X_1\cup X_2| +\sqrt{n}\right) k!\right)+  \mathcal{O}(1)^{k^2} \times\mathcal{O}\left(n^{k/2} / \sigma \right)\\
&= \mathcal{O}(1)^{k^2} \times  \mathcal{O}\left(n^{k/2}\left(\frac{\Delta}{n} + \frac{1}{\sqrt{n}} + \frac{1}{\sigma}\right)\right),
\end{align*}
as desired. This completes the proof of Proposition \ref{oneinCLTpaper} and Lemma \ref{moment_lemma}
\end{proof}
\end{proof}
Now, we finally present the proof of Lemma \ref{lemma tail prob of biased maj dyn}. 
\begin{proof}[Proof of Lemma \ref{lemma tail prob of biased maj dyn}. ]
Suppose $k\leq \sqrt{\frac{n}{20}}$ is even. By Markov's inequality,
\begin{align}\label{dafdsfdfewfewfwefwefewfdanantntnatndanat}
\mathbb{P}\left(\Big||C_{1,1}| - \mathbb{E}\left[|C_{1,1}|\right]\Big| \geq D \sqrt{pn}\Delta\right) & = \mathbb{P}\left(\Big||C_{1,1}| - \mathbb{E}\left[|C_{1,1}|\right]\Big|^k \geq( D \sqrt{pn}\Delta)^k\right)\notag\\
&\leq \frac{\mathbb{E}\left[\Big(|C_{1,1}| - \mathbb{E}\left[|C_{1,1}|\right]\Big)^k\right]}{( D \sqrt{pn}\Delta)^k}.
\end{align}

By Lemma \ref{moment_lemma},
\begin{equation} 
   \mathbb{E}\Big[\big(|C_{1,1}| - \mathbb{E}\left[|C_{1,1}|\right]\big)^k\Big] \leq  \mathcal{O}(1)^{k^2} \times  \mathcal{O}\left(n^{k/2}\left(\frac{\Delta}{n} + \frac{1}{\sigma}\right)\right) + \frac{(k-1)!!}{4} \cdot n^{k/2}.\notag
\end{equation}
We can express \(\sigma = \frac{1}{\sqrt{np(1-p)}} = \frac{1}{n} \times \frac{\sqrt{np}}{11}\times \frac{11}{p} \geq \frac{1}{n}\times \frac{\sqrt{\log n}}{10}\times \Delta\). Therefore, $\frac{\Delta}{n} = \mathcal{O}\left(\frac{1}{\sigma}\right).$ This means that:
\begin{equation} 
   \mathbb{E}\Big[\big(|C_{1,1}| - \mathbb{E}\left[|C_{1,1}|\right]\big)^k\Big] \leq     \frac{a^{k^2}n^{k/2}}{\sigma} + \frac{(k-1)!!}{4} \cdot n^{k/2}\notag
\end{equation}
for some constant $a.$ To simplify the proceeding computations, we may also assume that $a\geq 3$. By substituting this into Equation \ref{dafdsfdfewfewfwefwefewfdanantntnatndanat}:
\begin{align*}
\mathbb{P}\left(\Big||C_{1,1}| - \mathbb{E}\left[|C_{1,1}|\right]\Big| \geq D \sqrt{pn}\Delta\right) & \leq \frac{a^{k^2}}{\sigma(D\sqrt{p}\Delta)^k} + \left(\frac{\sqrt{k}}{D\sqrt{p}\Delta}\right)^k.
\end{align*}
Let $k =2 \left\lfloor\dfrac{k_2}{2}\right\rfloor$, where $k_2 =\dfrac{\log (D\sqrt{p}\Delta)}{\log a} $. We will quickly check that $k$ satisfies the conditions for applying Lemma \ref{moment_lemma}. Since $k$ is clearly even, we only need to verify that $k\leq \sqrt{\frac{{n}}{{20}}}$. Using the fact that \(\Delta \leq \frac{10}{p}\) and $p\geq \frac{1}{n}$, we can deduce \(k\leq k_2\leq \dfrac{\log (10D \sqrt{n})}{\log a} \leq \sqrt{\dfrac{n}{20}}\) (where the final inequality holds if $n$ is sufficiently large). Now, one can easily see that
\(\dfrac{a^{k^2}}{\sigma(D\sqrt{p}\Delta)^k} =\dfrac{1}{\sigma}\left( \dfrac{a^{k}}{D\sqrt{p}\Delta} \right)^k\leq \dfrac{1}{\sqrt{np(1-p)}} \).

Furthermore, 
\(
\left(\dfrac{\sqrt{k}}{D\sqrt{p}\Delta}\right)^k  \leq \dfrac{\sqrt{k_2}^{k_2}}{(D\sqrt{p}\Delta)^{k_2-2}} \)
It is easy to see that $k_2 \leq D\sqrt{p}{\Delta}$. Therefore,
\(
\dfrac{\sqrt{k_2}^{k_2}}{(D\sqrt{p}\Delta)^{k_2-2}} \leq \left(\dfrac{1}{D\sqrt{p}\Delta}\right)^{\frac{k_2}{2}-2} = (D^2 p\Delta^2) \exp\left[ -\frac{\log(D\sqrt{p}\Delta)^2}{2\log a}\right].
\) 
We will find a condition on $\Delta$ that ensures \(\exp\left[- \frac{\log({D\sqrt{p}\Delta})^2}{2\log a }
\right]\leq \dfrac{\mathcal{D}p  }{10D^2}\). 
\begin{align*}
\exp\left[- \frac{\log({D\sqrt{p}\Delta})^2}{2\log a }
\right]\leq \frac{\EEEEE p}{10 D^2} & \Leftrightarrow \frac{\log(D\sqrt{p}\Delta)^2}{2\log a} \geq \log \left(\frac{10 D^2}{\EEEEE p}\right)\\
&\Leftrightarrow \log\left(\sqrt{p}\Delta \right) \geq \sqrt{2\log a} \sqrt{ \log\left(\frac{10D^2}{\EEEEE}\right) + \log \left(\frac{1}{p}\right)} - \log (D).
\end{align*}

Since $p\leq \frac{1}{4}$, the right hand side of the above expression can be bounded above by $\widehat{C}_{D,\mathcal{D}} \sqrt{\log\left(\dfrac{1}{p}\right)}$ for some constant $\widehat{C}_{D,\mathcal{D}}$ (which depends on $D$ and $\EEEEE$). Therefore, it suffices to have $\log\left(\sqrt{p}\Delta \right) \geq \widehat{C}_{D,\mathcal{D}}\sqrt{\log \left(\frac{1}{p}\right)}$. 

Therefore,  for \(\Delta \geq \dfrac{1}{\sqrt{p}} \exp\left[\widehat{C}_{D,\mathcal{D}}\sqrt{\log \left(\dfrac{1}{p}\right)}\right]\),
\begin{align*}
    \mathbb{P}\left(\Big||C_{1,1}| - \mathbb{E}\left[|C_{1,1}|\right]\Big| \geq D \sqrt{pn}\Delta\right) \leq \frac{1}{\sqrt{np(1-p)}} + \EEEEE p\Delta.
\end{align*}
This completes the proof.
\end{proof}

\section{Step 3 - Day 2 Part 1}

The goal of this section is to find a bound on \(\mathbb{E}\left[|C_{1,2}|\right]\). We can achieve this by bounding \(\mathbb{P}\left(v \in C_{1,2}\right)\). However, this probability differs depending on whether \( v \in C_{1,0} \) or \( v \in C_{2,0} \). This distinction arises because if a vertex is initially of colour \(i\), its neighbours are more likely to be colour \(i\) on day 1. Consequently, \(v\) is more likely to be colour \(i\) on day 2, as it is more likely to have a higher proportion of neighbours of colour \(i\).

For this reason, we calculate \(\mathbb{P}(v_1 \in C_{1,2}) + \mathbb{P}(v_2 \in C_{1,2})\) for \(v_1 \in C_{1,0}\) and \(v_2 \in C_{2,0}\). This allows us to cancel out the respective biases that \(v_i\) may have towards being in \(C_{i,2}\).

\begin{lemma}\label{RED.BIAS}
Suppose \(\Delta \geq 2\), $n \geq 3$ and \(\frac{\log n}{n}\leq p \leq \frac{1}{4}\). Furthermore, let \(p\Delta_2\geq  2*10^8\). Let $v_1\in C_{1,0}$ and $v_2\in C_{2,0}$. Let \(\mathcal{E}_1\) the event that 
\begin{equation*}
    |\hat{R}_{v_1}| = M_1 \text{ for some } M_1\geq \frac{n-1}{2} + \Delta_2 + \beta_1,
\end{equation*} 
 where \(\beta_1:= \mathbb{E}[|\hat{R}_{v_1}|] - \mathbb{E}[|C_{1,1}|]\). Similarly, let $\mathcal{E}_2$ the event that 
\begin{equation*}
    |\hat{R}_{v_2}| = M_2 \text{ for some } M_2 \geq \frac{n-1}{2} + \Delta_2 - \beta_2,
\end{equation*} 
where \(\beta_2:= \mathbb{E}[|C_{1,1}|] - \mathbb{E}[|\hat{R}_{v_2}|]\). Then,
\begin{align*}
    \mathbb{P}(v_1 \in {{{{C}}}}_{1,2}\big|\mathcal{E}_1)+  \mathbb{P}(v_2 \in {{{{C}}}}_{1,2}\big|\mathcal{E}_2)\geq 1 + \phi(6)\min \left\{1,\frac{\sqrt{p}\Delta_2}{\sqrt{n}}\right\}.
\end{align*}

\end{lemma}
In order to prove Lemma \ref{RED.BIAS}, we will need the following
proposition, which we prove in Appendix \ref{appendix_C_2_colours}.
\begin{proposition}\label{beta__proposition}
   Suppose \(\Delta \geq 2\), $n \geq \exp[6*10^{13}]$ and \(\frac{\log n}{n}\leq p \leq \frac{1}{4}\). Then,
   \begin{enumerate}[label=(\roman*)]
\item \(\beta_2 - \beta_1 \leq \frac{80C_{BE}+3}{p}\)
\item \(\beta_2 \leq \frac{3\sqrt{(1-p)(n-1)}    }{\sqrt{p}}\)
   \end{enumerate}
\end{proposition}

Now, we present the proof to Lemma \ref{RED.BIAS}.

\begin{proof}[Proof of Lemma \ref{RED.BIAS}. ]
By definition, 
\[
\mathbb{P}(v_1 \in {C}_{1,2}\big|\mathcal{E}_1) \geq \mathbb{P}\left( |\Gamma(v_1)\cap C_{1,1}| - |\Gamma(v_1)\cap C_{2,1}| > 0 \big|\mathcal{E}_1 \right).
\]
As noted in Definition \ref{definition__of__R__hat__v}, we have 
\[
\Gamma(v_1)\cap C_{1,1} = \Gamma(v_1)\cap \hat{R}_{v_1}.
\]
Similarly, 
\[
\Gamma(v_1)\cap C_{2,1} = \Gamma(v_1) \cap (V\backslash (\{v_1\} \cup \hat{R}_{v_1})).
\]
Since \(\Gamma(v_1)\) is independent of \(\mathcal{E}_1\) (as \(\mathcal{E}_1\) is an event on \(G[V\backslash\{v_1\}]\)), it follows that 
\[
|\Gamma(v_1)\cap \hat{R}_{v_1}|\sim \mathrm{Bin}(|\hat{R}_{v_1}|,p)
\]
and 
\[
|\Gamma(v_1) \cap (V\backslash (\{v_1\} \cup R_{v_1}))|\sim \mathrm{Bin}(n-1-|\hat{R}_{v_1}|,p).
\]
Furthermore, conditioned on \(\mathcal{E}_1\), the value of \(|\hat{R}_{v_1}|\) is fixed, implying that these are two independent binomial random variables. Therefore, by Corollary \ref{in_paper_A.2} in Appendix \ref{appendix_A},
\begin{align}\label{youdonothavetoliveit_one}
    \mathbb{P}(v_1 \in {{{{C}}}}_{1,2}\big|\mathcal{E}_1) &\geq \Phi\left(\frac{\mathbb{E}\left[|\Gamma(v_1)\cap C_{1,1}| - |\Gamma(v_1)\cap C_{2,1}|\big|\mathcal{E}_1\right]}{\sqrt{(n-1)p(1-p)}}\right) - \frac{C_{BE}}{\sqrt{(n-1)p(1-p)}}\notag\\
    &= \Phi\left(\frac{p(2M_1-n+1)}{\sqrt{(n-1)p(1-p)}}\right) - \frac{C_{BE}}{\sqrt{(n-1)p(1-p)}}\notag\\
    &\geq \Phi \left( \frac{2\sqrt{p}(\Delta_2+\beta_1)}{\sqrt{(n-1)(1-p)}}\right)- \frac{C_{BE}}{\sqrt{(n-1)p(1-p)}}\notag\\
    &= \frac{1}{2} + \Phi_0 \left( \frac{2\sqrt{p}(\Delta_2+\beta_1)}{\sqrt{(n-1)(1-p)}}\right)- \frac{C_{BE}}{\sqrt{(n-1)p(1-p)}}.
\end{align}

Likewise, \(\mathbb{P}(v_2 \in {{{{C}}}}_{1,2}\big|\mathcal{E}_2) = \mathbb{P}(|\Gamma(v_2)\cap C_{1,1}| - |\Gamma(v_2)\cap C_{2,1}|>0\big|\mathcal{E}_2)\), and \(|\Gamma(v_2)\cap C_{1,1}| \sim \mathrm{Bin}(|\hat{R}_{v_2}|,p)\) and, conditioned on $\mathcal{E}_2$, \(|\Gamma(v_2)\cap C_{2,1}| \sim \mathrm{Bin}(n-1-|\hat{R}_{v_2}|,p)\) are independent random variables. Therefore, by a very similar calculation to before:

\begin{align}\label{youdonothavetoliveit_two}
    \mathbb{P}(v_2 \in {{{{C}}}}_{1,2}\big|\mathcal{E}_2) 
    &\geq 
    \frac{1}{2} - \Phi_0 \left( \frac{2\sqrt{p}(\beta_2-\Delta_2)}{\sqrt{(n-1)(1-p)}}\right)- \frac{C_{BE}}{\sqrt{(n-1)p(1-p)}}.
\end{align}
Let \(Z\sim \mathcal{N}(0,1)\). Then, by Equations \ref{youdonothavetoliveit_one} and \ref{youdonothavetoliveit_two},\begin{align*}
    \mathbb{P}(v_1 \in {{{{C}}}}_{1,2}\big|\mathcal{E}_1)+  \mathbb{P}(v_2 \in {{{{C}}}}_{1,2}\big|\mathcal{E}_2)& \geq 1 +\Phi_0 \left( \frac{2\sqrt{p}(\Delta_2+\beta_1)}{\sqrt{(n-1)(1-p)}}\right) -\Phi_0 \left( \frac{2\sqrt{p}(\beta_2-\Delta_2)}{\sqrt{(n-1)(1-p)}}\right)\\
    &\quad - \frac{2C_{BE}}{\sqrt{(n-1)p(1-p)}}\\
 &   = 1 + \mathbb{P}\left(\frac{2\sqrt{p}(\Delta_2+\beta_1)}{\sqrt{(n-1)(1-p)}} \geq Z\geq \frac{2\sqrt{p}(\beta_2-\Delta_2)}{\sqrt{(n-1)(1-p)}}\right)\\
 &\quad- \frac{2C_{BE}}{\sqrt{(n-1)p(1-p)}},
\end{align*}
provided \( \beta_1 +\Delta_2> \beta_2-\Delta_2\). This is equivalent to saying \(\Delta_2 \geq \frac{1}{2}(\beta_2-\beta_1)\). Since \(\Delta_2 \geq \frac{2C_{BE}}{\phi(6)p} \geq \frac{80C_{BE}+3}{p}\), we have \(\Delta_2 \geq \beta_2-\beta_1\) (by Proposition \ref{beta__proposition}(i[)). Therefore, \(\Delta_2+\beta_1 \geq \beta_2\), so

\begin{align}\label{wfoiwnhiognvncoewnofojwogiwjognhorwounonobnownbiorneiobniorenborenobnreonboienbiorenbioerniobnreiobnreoibnorenborengorengoernbbejov ls cwmnfwefiwpefw}
    \mathbb{P}(v_1 \in {{{{C}}}}_{1,2}\big|\mathcal{E}_1)+  \mathbb{P}(v_2 \in {{{{C}}}}_{1,2}\big|\mathcal{E}_2) &\geq 1 + \mathbb{P}\left(\frac{2\sqrt{p}\beta_2}{\sqrt{(n-1)(1-p)}} \geq Z\geq \frac{2\sqrt{p}(\beta_2-\Delta_2)}{\sqrt{(n-1)(1-p)}}\right)\notag\\
    &\quad- \frac{2C_{BE}}{\sqrt{(n-1)p(1-p)}}
    \end{align}
    Define \(\mathcal{M} := \max\left\{0,\frac{2\sqrt{p}(\beta_2-\Delta_2)}{\sqrt{(n-1)(1-p)}}\right\}\). Then, by Equation \ref{wfoiwnhiognvncoewnofojwogiwjognhorwounonobnownbiorneiobniorenborenobnreonboienbiorenbioerniobnreiobnreoibnorenborengorengoernbbejov ls cwmnfwefiwpefw},
    \begin{align}\label{subitnoodidid}
    \mathbb{P}(v_1 \in {{{{C}}}}_{1,2}\big|\mathcal{E}_1)+  \mathbb{P}(v_2 \in {{{{C}}}}_{1,2}\big|\mathcal{E}_2)
    &\geq     1 + \mathbb{P}\left(\frac{2\sqrt{p}\beta_2}{\sqrt{(n-1)(1-p)}} \geq Z\geq \mathcal{M}\right)\notag\\
   &\quad - \frac{2C_{BE}}{\sqrt{(n-1)p(1-p)}}.
\end{align}

It is easy to see that \(\mathbb{P}\left(\frac{2\sqrt{p}\beta_2}{\sqrt{(n-1)(1-p)}} \geq Z\geq \mathcal{M}\right)\) is a decreasing function of \(\beta_2\). Therefore, by Proposition \ref{beta__proposition}(ii), 
 \begin{equation*}
    \mathbb{P}\left(\frac{2\sqrt{p}\beta_2}{\sqrt{(n-1)(1-p)}} \geq Z\geq \mathcal{M}\right) \geq      \mathbb{P}\left(6 \geq Z\geq \max\left\{0, 6 - \frac{2\sqrt{p}\Delta_2}{\sqrt{(n-1)(1-p)}}\right\}\right).
 \end{equation*}
 Furthermore, since the probability density function of a standard Normal random variable decreases the further away you go from $0$, we have:
 \begin{align*}
     \mathbb{P}\left(6 \geq Z\geq \max\left\{0, 6 - \frac{2\sqrt{p}\Delta_2}{\sqrt{(n-1)(1-p)}}\right\}\right) &\geq      
     \phi(6)\min\left\{ 6,\frac{2\sqrt{p}\Delta_2}{\sqrt{(n-1)(1-p)}}\right\}.
 \end{align*}
 Therefore,
 \begin{align*}
     \mathbb{P}\left(\frac{2\sqrt{p}\beta_2}{\sqrt{(n-1)(1-p)}} \geq Z\geq \frac{2\sqrt{p}(\beta_2-\Delta_2)}{\sqrt{(n-1)(1-p)}}\right) \geq      
     \phi(6)\min\left\{ 6,\frac{2\sqrt{p}\Delta_2}{\sqrt{(n-1)(1-p)}}\right\}.
 \end{align*}
 Substituting this into Equation \ref{subitnoodidid}, we find:
\begin{align}\label{fhjfjfjfhghgbbbbbnvnfnfnffjfjofjwfjwefowejfweiojoiwjefoiwjefoijiowjfiojweoifjiowejfiowejfiowejifojweiofjiowejfioweoifj}
    \mathbb{P}(v_1 \in {{{{C}}}}_{1,2}\big|\mathcal{E}_1)+  \mathbb{P}(v_2 \in {{{{C}}}}_{1,2}\big|\mathcal{E}_2) \geq 1 +\min \bigg\{6\phi(6)-\frac{2C_{BE}}{\sqrt{(n-1)p(1-p)}}\notag,\\
    \frac{2 \phi(6)p\Delta_2 - 2C_{BE}}{\sqrt{(n-1)p(1-p)}} \bigg\}.
    \end{align}
    We can simplify the arguments of $\min $ in the above expression as follows:\\
Since \(p\geq \frac{\log n}{n}\) and $1-p\geq \frac{3}{4}$, we have:
\begin{align*}
    {\sqrt{(n-1)p(1-p)}} &\geq \sqrt{\left(1-\frac{1}{n}\right)\log n\left(\frac{3}{4}\right)}
\end{align*}
Since \(n\geq \exp [6*10^{13}]\), we have:
\[
6\phi(6)-\frac{2C_{BE}}{\sqrt{(n-1)p(1-p)}}\geq \phi(6).
\]
On the other hand, since \(p\Delta_2 \geq 2*10^8\), we have that \(\phi(6)p\Delta_2\geq 2C_{BE}\). Therefore,
\begin{align*}
        \frac{2 \phi(6)p\Delta_2 - 2C_{BE}}{\sqrt{(n-1)p(1-p)}} & \geq     \frac{ \phi(6)p\Delta_2}{\sqrt{(n-1)p(1-p)}}\geq     \frac{ \phi(6)\sqrt{p}\Delta_2}{\sqrt{n}} .
\end{align*}
Upon substituting these results into Equation \ref{fhjfjfjfhghgbbbbbnvnfnfnffjfjofjwfjwefowejfweiojoiwjefoiwjefoijiowjfiojweoifjiowejfiowejfiowejifojweiofjiowejfioweoifj}, we find:
    \begin{align*}
    \mathbb{P}(v_1 \in {{{{C}}}}_{1,2}\big|\mathcal{E}_1)+  \mathbb{P}(v_2 \in {{{{C}}}}_{1,2}\big|\mathcal{E}_2)&\geq 1 + \phi(6)\min \left\{1,\frac{\sqrt{p}\Delta_2}{\sqrt{n}}\right\}.
\end{align*}
%
%
%
%
%
%
\end{proof}

\begin{lemma}\label{COrollary_after_RED.BIAS}
    Suppose \(\frac{10}{p}\geq \Delta \), $n \geq \exp[6*10^{13}]$ and \(\frac{\log n}{n}\leq p \leq \frac{1}{4}\). Let $v_1\in C_{1,0}$ and $v_2\in C_{2,0}$. There are constants $A,B$ such that if $\Delta \geq \max \left\{\dfrac{1}{\sqrt{p}} \exp\left[A\sqrt{\log \left(\dfrac{1}{p}\right)}\right] , Bp^{-3/2} n^{-1/2} \right\},$ we have:
    \begin{align*}
\mathbb{P}(v_1 \in {{{{C}}}}_{1,2})+\mathbb{P}(v_2 \in {{{{C}}}}_{1,2})
&\geq 1 + 1.5*10^{-11} p\Delta .
\end{align*}

\end{lemma}
\begin{proof}
    Let \(\mathcal{E}_1\) be the event that 
\(
|\hat{R}_{v_1}|\geq \mathbb{E}\left[|\hat{R}_{v_1}|\right] -  \frac{19\sqrt{pn}\Delta}{3840}.
    \)
 Similarly, let $\mathcal{E}_2$ the event that 
\(
|\hat{R}_{v_2}|\geq \mathbb{E}\left[|\hat{R}_{v_2}|\right] -  \frac{19\sqrt{pn}\Delta}{3840}.
    \) 
By the law of total probability,
\begin{align*}
\mathbb{P}(v_i \in {{{{C}}}}_{1,2}) &= \mathbb{P}(v_i \in {{{{C}}}}_{1,2}\big|\mathcal{E}_i)\mathbb{P}(\mathcal{E}_i) + \mathbb{P}(v_i \in {{{{C}}}}_{1,2}\big|\neg\mathcal{E}_i)\mathbb{P}(\neg\mathcal{E}_i)\\
   &\geq \mathbb{P}(v_i \in {{{{C}}}}_{1,2}\big|\mathcal{E}_i)\mathbb{P}(\mathcal{E}_i),
   \end{align*}
for $i=1,2$. Let $D = 19/3840$ and $\EEEEE$ be an arbitrary positive constant. Let $C_{D,\EEEEE}$ be the constant defined in Lemma \ref{lemma tail prob of |hatR_v|}. Suppose $\Delta \geq \dfrac{1}{\sqrt{p}} \exp\left[C_{D,\mathcal{D}}\sqrt{\log \left(\dfrac{1}{p}\right)}\right]$. Then, by Lemma \ref{lemma tail prob of |hatR_v|},
\begin{align*}
\mathbb{P}(v_i \in {{{{C}}}}_{1,2}) &\geq \mathbb{P}(v_i \in {{{{C}}}}_{1,2}\big|\mathcal{E}_i)\left(1 - \frac{1}{\sqrt{p(1-p)(n-1)}} - \EEEEE p \Delta\right)\\
&\geq \mathbb{P}(v_i \in {{{{C}}}}_{1,2}\big|\mathcal{E}_i) - \frac{1}{\sqrt{p(1-p)(n-1)}} - \EEEEE p \Delta.
\end{align*}
Therefore,
\begin{align*}
\mathbb{P}(v_1 \in {{{{C}}}}_{1,2})+\mathbb{P}(v_2 \in {{{{C}}}}_{1,2})
&\geq \mathbb{P}(v_1 \in {{{{C}}}}_{1,2}\big|\mathcal{E}_1) + \mathbb{P}(v_2 \in {{{{C}}}}_{1,2}\big|\mathcal{E}_2) - \frac{2}{\sqrt{p(1-p)(n-1)}}\\
&\quad- 2\EEEEE p \Delta.
\end{align*}
By Lemma \ref{expectation_lemma_two_colours}, \(\mathcal{E}_1\) implies \(|\hat{R}_{v_1}|\geq \frac{n-1}{2} + \beta_1 + \frac{19\sqrt{pn}\Delta}{3840}\). In addition to this, \(\mathcal{E}_2\) implies \(|\hat{R}_{v_2}|\geq \frac{n-1}{2} + \beta_2 + \frac{19\sqrt{pn}\Delta}{3840}\). 
Therefore, by Lemma \ref{RED.BIAS} (if $\Delta \geq \frac{3840}{19}*2*10^8 p ^{-3/2} n^{-1/2}$),
\begin{align*}
\mathbb{P}(v_1 \in {{{{C}}}}_{1,2})+\mathbb{P}(v_2 \in {{{{C}}}}_{1,2})
&\geq 1 + \min\left\{\frac{19\phi(6)}{3840}p \Delta , \phi(6) \right\}-  \frac{2}{\sqrt{p(1-p)(n-1)}} - 2 \EEEEE p \Delta.\end{align*}
Since \(p\Delta\leq 10\), we must have \(\frac{19\phi(6)}{3840}p \Delta \leq\phi(6)\). Therefore,

\begin{align*}
\mathbb{P}(v_1 \in {{{{C}}}}_{1,2})+\mathbb{P}(v_2 \in {{{{C}}}}_{1,2})
&\geq 1 + \frac{19\phi(6)}{3840}p \Delta -\frac{2}{\sqrt{p(1-p)(n-1)}} - 2 \EEEEE p \Delta.
\end{align*}
Suppose \(\dfrac{2}{\EEEEE}p^{-3/2}n^{-1/2} \leq \Delta\). Then, \(\dfrac{2}{\sqrt{p(1-p)(n-1)}} \leq \dfrac{2}{\sqrt{pn}} \leq \EEEEE p \Delta\). Therefore, 
\begin{align*}
\mathbb{P}(v_1 \in {{{{C}}}}_{1,2})+\mathbb{P}(v_2 \in {{{{C}}}}_{1,2}) &\geq 1 + \frac{19\phi(6)}{3840}p \Delta  - 3 \EEEEE p \Delta.
\end{align*}

Let $\EEEEE = \frac{19\phi(6)}{3840}\times \frac{1}{6}$. Then, combining everything, there are constants $A,B$ such that if $\Delta \geq \max \left\{\dfrac{1}{\sqrt{p}} \exp\left[A\sqrt{\log \left(\dfrac{1}{p}\right)}\right] , Bp^{-3/2} n^{-1/2} \right\},$
we have:
\begin{align*}
\mathbb{P}(v_1 \in {{{{C}}}}_{1,2})+\mathbb{P}(v_2 \in {{{{C}}}}_{1,2})&\geq 1 + 1.5*10^{-11} p\Delta.
\end{align*}

\end{proof}

\begin{lemma}\label{main.lemma.expectation}
  Suppose \(\frac{10}{p} \geq \Delta \geq \Delta \geq \max \left\{\dfrac{1}{\sqrt{p}} \exp\left[A\sqrt{\log \left(\dfrac{1}{p}\right)}\right] , Bp^{-3/2} n^{-1/2} \right\}\), $n\geq \exp \left[ 6*10^{13}\right]$ and \(\frac{\log n}{n}\leq p \leq \frac{1}{4}\). Then, 
    \[
\mathbb{E}\left[ |C_{1,2}|\right] \geq \frac{n}{2} + 5*10^{-12}*pn\Delta.
\]
\end{lemma}
\begin{proof}
    We can write:
    \begin{align*}
        \mathbb{E}\left[ |C_{1,2}|\right] &=\mathbb{E}\left[ \sum_{v \in V} \mathbbm{1}_{v\in C_{1,2}}\right]= \sum_{v\in C_{1,0}}\mathbb{P}(v\in C_{1,0}) + \sum_{v\in C_{2,0}}\mathbb{P}(v\in C_{2,0}).
    \end{align*}
   Let $v_1\in C_{1,0}$ and $v_2 \in C_{2,0}$. Then,
\begin{align*}
    \mathbb{E}\left[ |C_{1,2}|\right] &=\left(\frac{n}{2}+\Delta\right)\mathbb{P}(v_1\in C_{1,2})+\left(\frac{n}{2}-\Delta\right)\mathbb{P}(v_2\in C_{2,2})\\
    &\geq \frac{n}{2}\left(\mathbb{P}(v_1\in C_{1,2})+\mathbb{P}(v_2\in C_{1,2})    \right) - 2\Delta.
\end{align*}
Therefore, by Lemma \ref{COrollary_after_RED.BIAS},
\begin{align*}
    \mathbb{E}\left[ |C_{1,2}|\right] &\geq \frac{n}{2} + 7.5*10^{-12}*pn\Delta - 2 \Delta.
\end{align*}

Since \(pn\geq \log n\) and $n\geq \exp \left[ 6*10^{13}\right]$, we have that \(2.5*10^{-12}*pn\geq 2 \). Therefore, 
\[
\mathbb{E}\left[ |C_{1,2}|\right] \geq \frac{n}{2} + 5*10^{-12}*pn\Delta.
\]
   
\end{proof}

\section{Step 4 - Day 2 Part 2}
\subsection{Some definitions required for Step 4. }
\begin{definition}\label{definition_of_S_one}
For $u\neq v \in C_{1,0}$, define 
    \begin{align*}
    S^{(1)}_{uv}:= \bigg\{w\in V\backslash \{u,v\} :\bigg(( |\Gamma(w)\cap C_{1,0}\backslash \{u,v\}| - |\Gamma(w)\cap C_{2,0}|\geq -1)\land (w \in C_{1,0})\bigg)\\
    \lor \bigg( (|\Gamma(w)\cap C_{1,0}\backslash \{u,v\}| - |\Gamma(w)\cap C_{2,0}|\geq 0 )\land (w \in C_{2,0})\bigg)\bigg\}.    
    \end{align*}
    This is the set of vertices in \(V\backslash \{u,v\}\) that are in \(C_{1,1}\) if they are connected to $u$ or $v$ (or both) by an edge. 
 \end{definition}
 \begin{definition}\label{definition_of_S_two}
     For $u\neq v \in C_{1,0}$, define
    \begin{align*}
    S^{(2)}_{uv}:= \bigg\{w\in V\backslash \{u,v\} :\bigg(( |\Gamma(w)\cap C_{1,0}\backslash \{u,v\}| - |\Gamma(w)\cap C_{2,0}|\leq -3)\land (w \in C_{1,0})\bigg)\\
    \lor \bigg( (|\Gamma(w)\cap C_{1,0}\backslash \{u,v\}| - |\Gamma(w)\cap C_{2,0}|\leq -2)\land (w \in C_{2,0})\bigg)\bigg\}.
    \end{align*}
    This is the set of vertices in \(V\backslash \{u,v\}\) that are in \(C_{2,1}\) no matter what the set of edges adjoining $u$ and $v$ is. 
   
 \end{definition}
 \begin{definition}\label{definition_of_S_star_of_u_and_v}
Suppose \(u\neq v \in C_{1,0}\). Define
\begin{multline*}
S^*_{uv} := \bigg\{ w \in V \backslash \{u,v\}: \bigg(|(\Gamma(w)\backslash\{u,v\})\cap C_{1,0}| = |(\Gamma(w)\backslash \{u,v\} ) \cap C_{2,0}|-2 \land w \in C_{1,0} \bigg)\\
\lor \bigg(|(\Gamma(w)\backslash\{u,v\})\cap C_{1,0}| = |(\Gamma(w)\backslash \{u,v\} ) \cap C_{2,0}|-1 \land w \in C_{2,0} \bigg)\bigg\}.
\end{multline*}
This is the set of vertices, $w$, in \(V\backslash \{u,v\}\) that are in \(C_{1,1}\) if and only if $w\sim u$ and $w\sim v$.
\end{definition}

A careful inspection of the above definitions allows one to see that \(S^{(1)}_{uv}\), \(S^{(2)}_{uv}\) and \(S^*_{uv}\) are solely determined by the edges of \(G[V\backslash \{u,v\}]\), so are independent of the neighbourhoods of $u$ and $v$. This means that, for a vertex \(w \in V\backslash \{u,v\}\), the event that $w$ is an element of \(S^{(1)}_{uv}\), \(S^{(2)}_{uv}\), or \(S^*_{uv}\) , \([u\sim w]\) and \([v \sim w]\) are independent events. 
  Furthermore, \(S^{(1)}_{uv} \cup S^{{{{{*}}}}}_{uv} \cup S^{(2)}_{uv}\) is a partition of $V\backslash \{u,v\}$. 
\begin{definition}For $u\neq v \in C_{1,0}$, define the following random variables:
\[\mathcal{G} := \bigg(|S^{(1)}_{uv}|,|S^{(2)}_{uv}|,|S^{{{{{*}}}}}_{uv}|,|S^{{{{{*}}}}}_{uv}\cap\Gamma(u)\cap\Gamma(v)|\bigg).\]
      \[
I(\mathcal{G}):= |S^{{{{{*}}}}}_{uv}\cap\Gamma(u)\cap\Gamma(v)|.
\]
\end{definition}

\subsection{A sketch of Step 4. }
The primary objective of this section is to derive a bound for \( \mathbf{Var}\left(|C_{1,2}|\right) \). To do so, we express it in the following way:
\begin{equation*}
    \mathbf{Var}\left(|C_{1,2}|\right) = \mathbf{Var}\left( \sum_{v \in C_{1,0}} \mathbbm{1}_{C_{1,2}}(v) + \sum_{v \in C_{2,0}} \mathbbm{1}_{C_{1,2}}(v)\right),
\end{equation*}
which we can bound above by  \(2 \mathbf{Var}\left( \sum_{v \in C_{1,0}} \mathbbm{1}_{C_{1,2}}(v)\right) +2 \mathbf{Var}\left( \sum_{v \in C_{2,0}} \mathbbm{1}_{C_{1,2}}(v)\right) \). We can express \(\mathbf{Var}\left( \sum_{v \in C_{1,0}} \mathbbm{1}_{C_{1,2}}(v)\right) = \sum_{v\in C_{1,0}}\mathbf{Var}\left(\mathbbm{1}_{u\in C_{1,2}}\right) + \sum_{u
\neq v \in C_{1,0}} \mathbf{Cov}\left( \mathbbm{1}_{C_{1,2}}(u),\mathbbm{1}_{C_{1,2}}(v)\right).\) As such, we will find upper bounds for \(\mathbf{Var}\left(\mathbbm{1}_{C_{1,2}}(v)\right)\) and \(\mathbf{Cov}\left( \mathbbm{1}_{C_{1,2}}(u),\mathbbm{1}_{C_{1,2}}(v)\right)\). This will allow us to find an upper bound for \(\mathbf{Var}\left( \sum_{v \in C_{1,0}} \mathbbm{1}_{C_{1,2}}(v)\right)\). By symmetry, we will also have an upper bound for \(\mathbf{Var}\left( \sum_{v \in C_{2,0}} \mathbbm{1}_{C_{1,2}}(v)\right)\), thereby completing the main goal of this step.

It is relatively easy to bound $\mathbf{Var}\left(\mathbbm{1}_{u\in C_{1,2}}\right) $, so most of the work in this section will be focused on approximating \(\mathbf{Cov}\left( \mathbbm{1}_{C_{1,2}}(u),\mathbbm{1}_{C_{1,2}}(v)\right)\). To do so, we use the law of total covariance which allows us to express:
\begin{align*}
   \mathbf{Cov}\left(\mathbbm{1}_{{{{{{{{{{{C}}}}}}}}}}_{1,2}}(u),\mathbbm{1}_{{{{{{{{{{{C}}}}}}}}}}_{1,2}}(v)\right) = \mathbb{E}\left[
   \mathbf{Cov}\left(\mathbbm{1}_{{{{{{{{{{{C}}}}}}}}}}_{1,2}}(u),\mathbbm{1}_{{{{{{{{{{{C}}}}}}}}}}_{1,2}}(v)\big|\mathcal{G}\right)\right] + \mathbf{Var}\left(\mathbb{P}\left(v \in {{{{{{{{{{C}}}}}}}}}}_{1,2} \big|\mathcal{G}\right)\right) .
\end{align*}
As such, it suffices to bound the two terms in the above sum. In order to bound \break\(\mathbb{E}\left[
   \mathbf{Cov}\left(\mathbbm{1}_{{{{{{{{{{{C}}}}}}}}}}_{1,2}}(u),\mathbbm{1}_{{{{{{{{{{{C}}}}}}}}}}_{1,2}}(v)\big|\mathcal{G}\right)\right]\), we first notice that 
   \begin{align*}
   \mathbf{Cov}\left(\mathbbm{1}_{{{{{{{{{{{C}}}}}}}}}}_{1,2}}(u),\mathbbm{1}_{{{{{{{{{{{C}}}}}}}}}}_{1,2}}(v)\big|\mathcal{G}\right) &= \mathbb{P}\left(u,v\in C_{1,2}\big|\mathcal{G}\right)-\mathbb{P}\left(u\in C_{1,2}\big|\mathcal{G}\right)\mathbb{P}\left(v\in C_{1,2}\big|\mathcal{G}\right) \\
   &= \mathbb{P}\left(u,v\in C_{1,2}\big|\mathcal{G}\right) - \mathbb{P}\left(u\in C_{1,2}\big|\mathcal{G}\right)^2
   \end{align*}
   For values of \(\mathcal{G}\) inside of some set $S$ such that $\mathcal{G}\in S$ with very high probability, we upper bound $\mathbb{P}\left(u,v\in C_{1,2}\big|\mathcal{G}\right) - \mathbb{P}\left(u\in C_{1,2}\big|\mathcal{G}\right)^2$. Using this result, we are finally able to bound \(\mathbb{E}\left[
   \mathbf{Cov}\left(\mathbbm{1}_{{{{{{{{{{{C}}}}}}}}}}_{1,2}}(u),\mathbbm{1}_{{{{{{{{{{{C}}}}}}}}}}_{1,2}}(v)\big|\mathcal{G}\right)\right]\). By combining this with a bound for the value of \(\mathbf{Var}\left(\mathbb{P}\left(v \in {{{{{{{{{{C}}}}}}}}}}_{1,2} \big|\mathcal{G}\right)\right) \), one obtains an upper bound for \( \mathbf{Var}\left(|C_{1,2}|\right) \).

   The key observation that enables us to bound \( \mathbf{Cov}\left(\mathbbm{1}_{{{{{{{{{{{C}}}}}}}}}}_{1,2}}(u),\mathbbm{1}_{{{{{{{{{{{C}}}}}}}}}}_{1,2}}(v)\right)\) (for $u\neq v\in C_{1,0}$) is understanding the reason behind the dependence between \( \mathbbm{1}_{C_{1,2}}(u) \) and \( \mathbbm{1}_{C_{1,2}}(v) \). This dependence arises because there are certain vertices that are colour 1 on day 1 if and only if they are neighbours of both \( u \) and \( v \). We denote the set of such vertices as \( S^*_{uv} \).

If \( u \in C_{1,2} \), then the intersection \( \Gamma(u) \cap \Gamma(v) \cap S^*_{uv} \) is more likely to be large. This implies that \( v \) is more likely to have an increased number of neighbours in \( C_{1,1} \), consequently making it more probable that \( v \) will also belong to \( C_{1,2} \). 

\subsection{Overview of Step 4. }

We split the arguments in this section up into steps as follows. Note that in this section, we no longer assume that $|C_{1,0}|\geq |C_{2,0}|$ (but will still denote $\Delta := \frac{1}{2} \left( |C_{1,0}|-|C_{2,0}|\right)$).
\begin{enumerate}
    \item First, we prove Lemma \ref{short_lemma_on_importance_of_day_2_p_2_+definitions}, which demonstrates why Definitions \ref{definition_of_S_one}, \ref{definition_of_S_two} and \ref{definition_of_S_star_of_u_and_v} are useful.
    \item Proposition \ref{whi_why+_do_I+ke_kp_keep_lesbel} is a technical result that presents bounds for the variance and expected value of $|S^*_{uv}|$. We will use these results for bounding $|S^*_{uv}|$ via Chebyshev's inequality later in the proof. 
    \item Next, we prove Propositions \ref{tinyproopvariacne} and \ref{currentlygoinginappendixBvariance4parttechnical}. These results give bounds on the sizes of various probabilities/expectations/variances that will be used later in the proof.
    \item Next, we prove Proposition \ref{cov_calc_in_appendix_C}, which bounds $\mathbb{P}(u,v \in {{{{{{{{{{C}}}}}}}}}}_{1,2} |\mathcal{G})
- \mathbb{P}(u \in {{{{{{{{{{C}}}}}}}}}}_{1,2} |\mathcal{G} )^2$, for certain values of $\mathcal{G}$.
    \item After this, we prove Lemmas \ref{SOMETHING__TH__TH+__++_+IN_PLAN} and \ref{THIRD_IN_PLAN_LEMMA}, which provide bounds for \hfill\break\(\mathbb{E}\mathbf{Cov}\left(\mathbbm{1}_{{{{{{{{{{{C}}}}}}}}}}_{1,2}}(u),\mathbbm{1}_{{{{{{{{{{{C}}}}}}}}}}_{1,2}}(v)\bigg|\mathcal{G}\right)\) and \(\mathbf{Var}\left(\mathbb{P}\left(v \in {{{{{{{{{{C}}}}}}}}}}_{1,2} \big|\mathcal{G}\right)\right) \), for $u\neq v\in C_{1,0}$, respectively. 
    \item We then are able to bound \(\mathbf{Cov}\left(\mathbbm{1}_{{{{{{{{{{{C}}}}}}}}}}_{1,2}}(u),\mathbbm{1}_{{{{{{{{{{{C}}}}}}}}}}_{1,2}}(v)\right)\) for $u\neq v \in C_{1,0}$ by expressing it in terms of \(\mathbb{E}\mathbf{Cov}\left(\mathbbm{1}_{{{{{{{{{{{C}}}}}}}}}}_{1,2}}(u),\mathbbm{1}_{{{{{{{{{{{C}}}}}}}}}}_{1,2}}(v)\bigg|\mathcal{G}\right)\) and \(\mathbf{Var}\left(\mathbb{P}\left(v \in {{{{{{{{{{C}}}}}}}}}}_{1,2} \big|\mathcal{G}\right)\right) \). 
    \item Since the same results hold for $u\neq v\in C_{2,0}$, we are able to find a bound for \(\mathbf{Var}\left(|{{{{{{{{{{C}}}}}}}}}}_{1,2}|\right) \) by expressing it in terms of \(\mathbf{Cov}\left(\mathbbm{1}_{{{{{{{{{{{C}}}}}}}}}}_{1,2}}(u),\mathbbm{1}_{{{{{{{{{{{C}}}}}}}}}}_{1,2}}(v)\right)\), for $u$ and $v$ distinct vertices of the same colour. We do this in Lemma \ref{mainRESULTstep4VariaNCEdayTWO}, which is the main result of this section.
    \end{enumerate}


\begin{lemma}\label{short_lemma_on_importance_of_day_2_p_2_+definitions}
    If $u\not\sim v$
\begin{multline*}
 |\Gamma(u) \cap C_{1,1}| - |\Gamma(v)\cap C_{2,1}|=|S^{(1)}_{uv} \cap \Gamma(u) | - |S^{(2)}_{uv} \cap \Gamma(u) | - |S^{{{{{*}}}}}_{uv}\cap (\Gamma(u)\backslash\Gamma(v)) | + I\left(\mathcal{G}\right).
\end{multline*}
\end{lemma}
\begin{proof}
    By definition, 
\[
\Gamma(u)\backslash \{v\} \cap C_{1,1} = \left(\Gamma(u)\cap S^{(1)}_{uv}\right)\cup \left(\Gamma(u)\cap\Gamma(v) \cap S^{{{{{*}}}}}_{uv}\right).
\]
Similarly, 
\[
\Gamma(u)\backslash \{v\} \cap C_{2,1} = \left((\Gamma(u)\cap S^{{{{{*}}}}}_{uv})\cup ( \Gamma(u)\cap S^{(2)}_{uv})\right)\backslash (\Gamma(u)\cap\Gamma(v) \cap S^{{{{{*}}}}}_{uv}).
\]
Therefore, if \( u \not \sim v\),
\begin{align*}
\big|\Gamma(u)\cap C_{1,1} \big|= \big|\Gamma(u)\cap S^{(1)}_{uv}\big| +  I(\mathcal{G})
\end{align*}
and 
\[
\big|\Gamma(u) \cap C_{2,1} \big| = \big|S^{{{{{*}}}}}_{uv}\cap (\Gamma(u)\backslash\Gamma(v) ) \big| +  \big| \Gamma(u)\cap S^{(2)}_{uv}\big|.
\]
Therefore, if $u\not\sim v$
\begin{multline*}
 |\Gamma(u) \cap C_{1,1}| - |\Gamma(v)\cap C_{2,1}|=|S^{(1)}_{uv} \cap \Gamma(u) | - |S^{(2)}_{uv} \cap \Gamma(u) | - |S^{{{{{*}}}}}_{uv}\cap (\Gamma(u)\backslash\Gamma(v)) | + I\left(\mathcal{G}\right).
\end{multline*}
\end{proof}

\begin{proposition}\label{whi_why+_do_I+ke_kp_keep_lesbel}
Suppose \(\frac{1}{4}\geq p\geq \frac{\log n}{n}\).
 For $u \neq v \in C_{1,0}$,
\begin{enumerate}[label=(\roman*)]
    \item
    \begin{equation*}
    \mathbb{E}[|S_{u,v}^*|]  \leq \frac{2\sqrt{n}}{\sqrt{p}}.
\end{equation*}
\item If \(n\geq 1048\), then 
\begin{align}
   \mathbf{Var}(|S^*_{uv}|)& \leq \frac{289\sqrt{n}}{\sqrt{p}}.\notag
\end{align}
  \end{enumerate}
\end{proposition}
\begin{proof}[Proof of part (i). ] We can write 
\begin{align*}
\mathbb{E}[|S_{u,v}^*|]
&= \sum_{w \in V \backslash \{u,v\}}\mathbb{E}[\mathbbm{1}_{w \in S^*_{uv}}]
= \sum_{w \in C_{1,0}\backslash \{u,v\}}\mathbb{P}(w \in S^*_{uv}) + \sum_{w \in C_{2,0}} \mathbb{P}(w \in S^*_{uv}).
\end{align*}
If \( w\in C_{1,0}\backslash \{u,v\}\), then 
\begin{align*}
\mathbb{P}(w \in S^*_{uv} ) &= \mathbb{P}\bigg(|(\Gamma(w)\backslash\{u,v\})\cap C_{1,0}| = |(\Gamma(w)\backslash \{u,v\} ) \cap C_{2,0}|-2\bigg).
\end{align*}
Since \(|(\Gamma(w)\backslash\{u,v\})\cap C_{1,0}|\sim \mathrm{Bin}(|C_{1,0}|-2,p)\) and \(|(\Gamma(w)\backslash \{u,v\} ) \cap C_{2,0}|\sim \mathrm{Bin}(|C_{2,0}|,p)\) are independent, by Lemma \ref{in_paper_A.3}, 
\begin{align}\label{equalion_at_katy_house}
\mathbb{P}(w \in S^*_{uv}) \leq \frac{1.12}{\sqrt{(n-2)p(1-p)}}.
\end{align}
Similarly, if \(w \in C_{2,0},\) then 
\begin{align}\label{equalion_at_katy_house_2}
\mathbb{P}(w \in S^*_{uv}) &= \mathbb{P}\bigg(|(\Gamma(w)\backslash\{u,v\})\cap C_{1,0}| = |(\Gamma(w)\backslash \{u,v\} ) \cap C_{2,0}|-1\bigg)\notag\\
&\leq \frac{1.12}{\sqrt{(n-2)p(1-p)}}.
\end{align}
Therefore, 
\begin{align*}
    \mathbb{E}[|S_{u,v}^*|] &\leq \sum_{w\in V\backslash \{u,v\}}\frac{1.12}{\sqrt{(n-2)p(1-p)}}
    = \frac{1.12\sqrt{n-2}}{\sqrt{p(1-p)}}
    \leq \frac{2\sqrt{n}}{\sqrt{p}}.
\end{align*}
%
%
%
%
%
%
%
\end{proof}
\begin{proof}[Proof of part (ii). ]
We can express \(\mathbf{Var}(|S_{u,v}^*|)\) as:
\begin{equation}\label{S_star_variance_calculation_first_equation_for_overall}
\mathbf{Var}(|S^*_{uv}|) = \sum_{w \in V \backslash \{u,v\}}\mathbf{Var}(\mathbbm{1}_{w \in S^*_{uv}}) + 2\sum_{\substack{w_1 <w_2\\ w_1,w_2 \in V\backslash \{u,v\}}} \mathbf{Cov}(\mathbbm{1}_{w_1 \in S^*_{uv}},\mathbbm{1}_{w_2 \in S^*_{uv}}).
\end{equation}
First, we consider the \(\mathbf{Var}(\mathbbm{1}_{w \in S^*_{uv}})\) term. 
\begin{align*}
\mathbf{Var}(\mathbbm{1}_{w \in S^*_{uv}}) = \mathbb{P}(w \in S_{u,v}^* ) - \mathbb{P}(w \in S_{u,v}^* ) ^2 \leq \mathbb{P}(w \in S_{u,v}^* ) .
\end{align*}
Therefore, by Equations \ref{equalion_at_katy_house} and \ref{equalion_at_katy_house_2}, 
\(\mathbf{Var}(\mathbbm{1}_{w \in S^*_{uv}})\leq \frac{1.12}{\sqrt{(n-2)p(1-p)}}.\)

Now, consider the \(\mathbf{Cov}(\mathbbm{1}_{w_1 \in S^*_{uv}},\mathbbm{1}_{w_2 \in S^*_{uv}})\) term, which we can express as: 
\begin{equation} \label{first_equation_variance_calculation_s*}
    \mathbf{Cov}(\mathbbm{1}_{w_1 \in S^*_{uv}},\mathbbm{1}_{w_2 \in S^*_{uv}}) = \mathbb{P}(w_1,w_2 \in S^*_{uv}) - \mathbb{P}(w_1 \in S^*_{uv}) \mathbb{P}(w_2 \in S^*_{uv}).
\end{equation}
We can further express \(\mathbb{P}(w_1,w_2 \in S^*_{uv})\) in terms of independent events in the following way:
\begin{equation}\label{second_equation_variance_calculation_s*}
    \mathbb{P}(w_1,w_2 \in S^*_{uv}) = p \mathbb{P}(w_1,w_2 \in S^*_{uv} | w_1 \sim w_2) + (1-p)\mathbb{P}(w_1,w_2 \in S^*_{uv}| w_1 \not \sim w_2).
\end{equation}
Let \(F_{1} := \{ w_1 \in S^*_{uv} | w_1 \sim w_2\}\) and vice versa for \(F_2\). Since \(w_1\) and \(w_2\) have no influence over each other's neighbours in \(V \backslash \{ u,v,w_1,w_2\}\), \(F_1\) and \(F_2\) are independent events. Therefore, \(\mathbb{P}(w_1,w_2\in S^*_{uv}| w_1 \sim w_2) = \mathbb{P}(F_1) \mathbb{P}(F_2)\).

Let \(E_1 := \{ w_1 \in S^*_{uv} | w_1 \not \sim w_2\}\), and vice versa for \(E_2\). Then, by the same logic as the previous case, \(E_1\) and \(E_2\) are independent events and \(\mathbb{P}(w_1,w_2 \in S^*_{uv}| w_1 \not \sim w_2) = \mathbb{P}(E_1) \mathbb{P}(E_2)\).

It is easy to express the probabilities \(\mathbb{P}(w_1 \in S^*_{uv})\) and \(\mathbb{P}(w_2 \in S^*_{uv})\) as \(p \mathbb{P}(F_1) + (1-p) \mathbb{P}(E_1)\) and \(p \mathbb{P}(F_2) + (1-p) \mathbb{P}(E_2)\) respectively. Substituting these expressions into Equations \ref{first_equation_variance_calculation_s*} and \ref{second_equation_variance_calculation_s*}, we find that, after some algebraic manipulation,
\begin{equation}\label{DAY1forrepeatuseappendixB_s*}
    \mathbf{Cov}(\mathbbm{1}_{w_1 \in S^*_{uv}},\mathbbm{1}_{w_2 \in S^*_{uv}}) = p(1-p)(\mathbb{P}(E_1) - \mathbb{P}(F_1))(\mathbb{P}(E_2) - \mathbb{P}(F_2)).
\end{equation}
For $i=1,2$, define \(\alpha_i := |\Gamma(w_1) \cap (C_{i,0}\backslash\{u,v,w_2\})|\). By definition, 
\[
\alpha_i = \sum_{s \in C_{i,0}\backslash \{u,v,w_1,w_2\}} \mathbbm{1}_{w_1 \sim s}.
\]

Furthermore, \(\{\mathbbm{1}_{w_1 \sim s}\}_{s\in V \backslash \{w_1\}}\) is a set of independent \( \mathrm{Ber}(p)\) random variables. Therefore, \(\alpha_i \sim \mathrm{Bin}(|C_{i,0}\backslash \{u,v,w_1,w_2\}|,p)\), and \(\alpha_1\) and \(\alpha_2\) are independent. We will now evaluate the term \(\mathbb{P}(E_1) - \mathbb{P}(F_1)\) via case analysis on the colour of \(w_1\).\\\\
\textit{Case 1: \( w_1 \in C_{1,0}. \)} In this case, 
\begin{align*}
\mathbb{P}(E_1) - \mathbb{P}(F_1) &= \mathbb{P}(\alpha_1-\alpha_2 = -L(w_2)-2) - \mathbb{P}(\alpha_1-\alpha_2 = -2).
\end{align*}
Therefore, by Lemma \ref{taken_from_other_paper}, 
\begin{align*}
|\mathbb{P}(E_1) - \mathbb{P}(F_1)|&\leq \frac{20C_{BE}}{(\max \{|C_{1,0}\backslash \big\{u,v,w_1,w_2\}|,|C_{2,0}\backslash \{u,v,w_1,w_2\}|\big\})p(1-p)}\\
&\leq \frac{20C_{BE}}{(\frac{n}{2}-4)p(1-p)}\\
&\leq \frac{23}{np(1-p)}.
\end{align*}
\textit{Case 2: \( w_1 \in C_{1,0}. \)} In this case, \begin{align*}
\mathbb{P}(E_1) - \mathbb{P}(F_1) = \mathbb{P}(\alpha_1-\alpha_2 = -L(w_2)-1) - \mathbb{P}(\alpha_1-\alpha_2 = -1).\end{align*} Therefore, by the same logic as before, $|\mathbb{P}(E_1) - \mathbb{P}(F_1)|\leq \frac{23}{np(1-p)}.$
%
Hence, in both cases, \(|\mathbb{P}(E_1) - \mathbb{P}(F_1)| \leq \frac{23}{np(1-p)} \). Since we get the same inequality for \(w_2\), by Equation \ref{DAY1forrepeatuseappendixB_s*},
\begin{equation*}
    \mathbf{Cov}[\mathbbm{1}_{w_1 \in S^*_{uv}},\mathbbm{1}_{w_2 \in S^*_{uv}}] \leq \frac{529}{n^2p(1-p)}.
\end{equation*}

Therefore, by substituting into Equation \ref{S_star_variance_calculation_first_equation_for_overall}, we find:
\begin{align*}
    \mathbf{Var}(|S^*_{uv}|)& \leq \sum_{w \in V \backslash \{u,v\}} \frac{1.12}{\sqrt{(n-2)p(1-p)}} + 2\sum_{\substack{w_1 < w_2\\ w_1,w_2 \in V \backslash \{u,v\}}}\frac{529}{n^2p(1-p)}\\
    &\leq \frac{1.12\sqrt{n-2}}{\sqrt{p(1-p)}} +  \frac{529}{p(1-p)}\\
    & \leq \frac{1.12\sqrt{n}}{\sqrt{p(1-p)}} + \frac{529}{p(1-p)}\\
    &\leq \frac{2.24\sqrt{3n}}{\sqrt{p}} + \frac{2116}{3p},
\end{align*}
where the final line follows because $p\leq \frac{1}{4}$. Since \(p \geq \frac{\log n}{n}\) and $n\geq 1048$,
\begin{align*}
\frac{1}{p}&\leq \frac{\sqrt{n}}{\sqrt{p}} \times \sqrt{\frac{1}{\log n}}
\leq \frac{\sqrt{n}}{\sqrt{p}} \times \sqrt{\frac{1}{\log 1048}}.
\end{align*}
Therefore, \begin{align}
   \mathbf{Var}(|S^*_{uv}|)& \leq \frac{289\sqrt{n}}{\sqrt{p}}.\notag
\end{align}
\end{proof}

\begin{proposition}\label{tinyproopvariacne}
Suppose \(\frac{1}{4}\geq p\geq \frac{\log n}{n}\) and $n\geq 1048$. Then,
    \begin{equation*}
        \mathbb{P}\left( |S^{{{{{*}}}}}_{uv}| > \frac{4\sqrt{n}}{\sqrt{p}}  \right) = \mathcal{O}\left( p\right).
\end{equation*}
\end{proposition}
\begin{proof}
    By Proposition \ref{whi_why+_do_I+ke_kp_keep_lesbel}, \(\mathbb{E}[|S^{{{{{*}}}}}_{uv}|] \leq \frac{2\sqrt{n}}{\sqrt{p}}\) and \(\mathbf{Var}(|S^{{{{{*}}}}}_{uv}|) \leq \frac{289\sqrt{n}}{\sqrt{p}}\). Therefore, by Chebyshev's inequality,
\begin{equation*}
        \mathbb{P}\left( |S^{{{{{*}}}}}_{uv}| > \frac{4\sqrt{n}}{\sqrt{p}}  \right) = \mathcal{O}\left( \frac{\sqrt{p}}{\sqrt{n}}\right) = \mathcal{O}\left( p\right).
\end{equation*}
\end{proof}
\begin{proposition}\label{currentlygoinginappendixBvariance4parttechnical}
Suppose $n \geq \exp \left[ 3*10^6\right]$, $\frac{\log n}{n} \leq p \leq \frac{1}{4}$ and $\frac{10}{p} \geq \Delta \geq 1$. Then,
\begin{enumerate}[label=(\roman*)]
    \item \[\mathbf{Var}\left(|S^{(1)}_{uv}| - |S^{(2)}_{uv}| - |S^{{{{{*}}}}}_{uv}| \right) \leq {4n}.\]
    \item \[ \mathbb{E}\left[\left(|S^{(1)}_{uv}| - |S^{(2)}_{uv}| - |S^{{{{{*}}}}}_{uv}| \right)^2\right] \leq \frac{328n}{p}.\]
    \item \[\mathbf{Var}\left(I(\mathcal{G})\right) \leq 21 p^{3/2}n^{1/2}.\]
    \item \[\mathbb{E}\left[I(\mathcal{G})  \right] \leq  2p^{3/2}n^{1/2}.\]
    \item \[\mathbb{E}\left[I(\mathcal{G})^2  \right] \leq  21 p^{3/2} n^{1/2} + 4p^3 n.\]
    \item \[
    \mathbb{E}\left[|S^{{{{{*}}}}}_{uv}|^2
    \right]\leq \frac{293n}{p}.
    \]
\end{enumerate}

\end{proposition}
We prove this in Appendix \ref{appendix_D_2_colours}.

\begin{proposition}\label{cov_calc_in_appendix_C}
Suppose \(u \neq v \in C_{1,0}\), $n\geq \exp \left[ 3*10^6\right]$ and $\frac{\log n}{n} \leq p \leq \frac{1}{4}$. \\

Let \(r,s,t\) be positive integers such that \(r+s+t = n-2\) and $t\leq \frac{4\sqrt{n}}{\sqrt{p}}$. Define the event \[H := \bigg[|S^{(1)}_{uv}|= r,|S^{(2)}_{uv}|=s,|S^{{{{{*}}}}}_{uv}|=t\bigg].\]

Then, for any $0\leq a\leq t$,
\begin{equation*}
    \mathbb{P}(u,v \in {{{{{{{{{{C}}}}}}}}}}_{1,2} |I(\mathcal{G})= a ,H)
\leq \mathbb{P}(u \in {{{{{{{{{{C}}}}}}}}}}_{1,2} |I\left(\mathcal{G}\right)= a,H )^2 +  \mathcal{O}(p)+\exp \left[-4\sqrt{np}\right].
\end{equation*}

\end{proposition}
\begin{remark}

    One may wonder why \([u \in C_{1,2} \mid I(\mathcal{G}) = a, H]\) and \([C_{1,2} \mid I(\mathcal{G}) = a, H]\) are not independent events. One reason is that \( u \) and \( v \) may or may not be neighbours, which means that \( \Gamma(u) \) and \( \Gamma(v) \) are not independent. However, this is not the only reason for their dependence, as the events \([u \in C_{1,2} \mid I(\mathcal{G}) = a, H, u \not\sim v]\) and \([C_{1,2} \mid I(\mathcal{G}) = a, H, u \not\sim v]\) are still not independent. If \( |S^{*}_{uv} \cap \Gamma(u) \backslash \Gamma(v)| \) is larger, then \( |S^{*}_{uv} \cap \Gamma(v) \backslash \Gamma(u)| \) is more likely to be smaller. This is because there are fewer vertices remaining in \( S^{*}_{uv} \) for \( v \) to be a neighbour of.

    \end{remark}
\begin{proof}
In the entirety of this proof, for notational convenience, we will be assuming that the event $H$ occurs (i.e. we will be working in the probability space defined by conditioning on $H$).

By Lemma \ref{short_lemma_on_importance_of_day_2_p_2_+definitions},
\begin{multline*}
    \mathbb{P}(u,v \in {{{{{{{{{{C}}}}}}}}}}_{1,2} |u \not \sim v,I\left(\mathcal{G}\right)= a  )\\
    = \mathbb{P}\bigg(\left[|S^{(1)}_{uv} \cap \Gamma(u) | - |S^{(2)}_{uv} \cap \Gamma(u) | - |S^{{{{{*}}}}}_{uv}\cap (\Gamma(u)\backslash\Gamma(v))|\geq -a \right]\land\\
    \left[|S^{(1)}_{uv} \cap \Gamma(v) | - |S^{(2)}_{uv} \cap \Gamma(v) ) | - |S^{{{{{*}}}}}_{uv}\cap (\Gamma(v)\backslash\Gamma(u))|\geq -a \right] \bigg| I\left(\mathcal{G}\right)= a
    \bigg).
\end{multline*}
Since \(S^{(1)}_{uv}\), \(S^{{{{{*}}}}}_{uv}\) and \(S^{(2)}_{uv}\) are disjoint, and since the neighbourhoods of $u$ and $v$ do not affect each other (if we condition on \( u\not \sim v\)), \(|S^{(1)}_{uv} \cap \Gamma(u) | \), \(|S^{(2)}_{uv} \cap \Gamma(u) |\), \(|S^{(1)}_{uv} \cap \Gamma(v) |\), \( |S^{(2)}_{uv} \cap \Gamma(v) ) | \) and \(\left(|S^{{{{{*}}}}}_{uv}\cap (\Gamma(u)\backslash\Gamma(v))|, |S^{{{{{*}}}}}_{uv} \cap (\Gamma(v)\backslash\Gamma(u))|\right) \) are independent random variables.

Each element of \(S^{{{{{*}}}}}_{uv}\backslash (\Gamma(u)\cap\Gamma(v))\), independently, is adjacent to \(u\) with probability $p^*$, where \(p^* := \mathbb{P}(w \in \Gamma(u) | w \not \in \Gamma(u) \cap \Gamma(v))\). By Bayes' rule,
\begin{align*}
p^* = \frac{\mathbb{P}(w \in \Gamma(u) \land w \not \in \Gamma(v))}{\mathbb{P}(w \not \in \Gamma(v) \cap \Gamma(u))} = \frac{\mathbb{P}(w \in \Gamma(u))\mathbb{P}(w \not \in \Gamma(v))}{1-\mathbb{P}(w  \in \Gamma(v))\mathbb{P}(w\in \Gamma(u))}= \frac{p(1-p)}{1-p^2} = \frac{p}{1+p}.
\end{align*}

Therefore, \begin{equation}\label{I_NEEN_NENND_EFEF_1}
|S^{{{{{*}}}}}_{uv}\cap \left(\Gamma(u)\backslash\Gamma(v)\right)| \sim \mathrm{Bin}\left(|S^{{{{{*}}}}}_{uv}\backslash\left(\Gamma(u)\cap\Gamma(v)\right) |,\frac{p}{1+p}\right).
\end{equation}

Now, conditioned on the value of \(S^{{{{{*}}}}}_{uv}\cap \Gamma(u)\), we will calculate the distribution of \\\({|S^{{{{{*}}}}}_{uv}\cap (\Gamma(v)\backslash\Gamma(u))|}\). Let $w$ be an element of  $S^{{{{{*}}}}}_{uv}\backslash \Gamma(u)$ and let \(\hat{p}\) be the probability that $w\in $\(S^{{{{{*}}}}}_{uv}\cap (\Gamma(v)\backslash\Gamma(u))\). It is easy to see that \(|S^{{{{{*}}}}}_{uv}\cap (\Gamma(v)\backslash\Gamma(u))| \sim \mathrm{Bin}(|S^{{{{{*}}}}}_{uv}\backslash \Gamma(u)|, \hat{p})\). By definition,
\[
\hat{p} = \mathbb{P}(w \in \Gamma(v) |w \not \in \Gamma(u)) = \mathbb{P}(w \in \Gamma(v)) = p.
\]
Therefore, \(|S^{{{{{*}}}}}_{uv}\cap (\Gamma(v)\backslash\Gamma(u))| \sim \mathrm{Bin}(|S^{{{{{*}}}}}_{uv}\backslash \Gamma(u)|, {p})\), which we can rewrite as: 
\begin{equation}\label{I_NEEN_NENND_EFEF_2}
|S^{{{{{*}}}}}_{uv}\cap (\Gamma(v)\backslash\Gamma(u))| \sim \mathrm{Bin}\big(|S^{{{{{*}}}}}_{uv}|-|S^{{{{{*}}}}}_{uv}\cap\Gamma(u)\cap\Gamma(v) | - |S^{{{{{*}}}}}_{uv}\cap (\Gamma(u)\backslash\Gamma(v))|, {p}\big).
\end{equation}
Let \(a\) be an arbitrary constant and let \(X_1\sim \mathrm{Bin}(|S^{{{{{*}}}}}_{uv}|-a,\frac{p}{1+p})\) and \(X_2\sim \mathrm{Bin}(|S^{{{{{*}}}}}_{uv}|-a - X_1, {p})\). If \(|S^{{{{{*}}}}}_{uv}\cap\Gamma(u)\cap\Gamma(v) | = a\), then by Equations \ref{I_NEEN_NENND_EFEF_1} and \ref{I_NEEN_NENND_EFEF_2}, \(|S^{{{{{*}}}}}_{uv}\cap \left(\Gamma(u)\backslash\Gamma(v)\right)|\overset{d}{=}X_1\)
 and \(|S^{{{{{*}}}}}_{uv}\cap (\Gamma(v)\backslash\Gamma(u))|\overset{d}{=}X_2\). Therefore,
\begin{multline*}
    \mathbb{P}(u,v \in {{{{{{{{{{C}}}}}}}}}}_{1,2} |u \not \sim v,I\left(\mathcal{G}\right)= a)
   = \mathbb{P}\bigg(\left[|S^{(1)}_{uv} \cap \Gamma(u) | - |S^{(2)}_{uv} \cap \Gamma(u) | - X_1 \geq -a \right]\land\\
   \left[|S^{(1)}_{uv} \cap \Gamma(v) | - |S^{(2)}_{uv} \cap \Gamma(v) ) | - X_2 \geq -a \right] \bigg| I\left(\mathcal{G}\right)= a
    \bigg).
    \end{multline*}
Since the joint distribution of \(|S^{(1)}_{uv} \cap \Gamma(u) | - |S^{(2)}_{uv} \cap \Gamma(u) | - X_1 \) and \(|S^{(1)}_{uv} \cap \Gamma(v) | - |S^{(2)}_{uv} \cap \Gamma(v) | - X_2\) is independent of \(I\left(\mathcal{G}\right)\), we must have:
\begin{align*}
    \mathbb{P}\big(u,v \in {{{{{{{{{{C}}}}}}}}}}_{1,2} |u \not \sim v,I\left(\mathcal{G}\right)= a \big)
   & = \mathbb{P}\bigg(\left[|S^{(1)}_{uv} \cap \Gamma(u) | - |S^{(2)}_{uv} \cap \Gamma(u) | - X_1 \geq -a \right]\land\notag\\
   &\qquad \quad\left[|S^{(1)}_{uv} \cap \Gamma(v) | - |S^{(2)}_{uv} \cap \Gamma(v) ) | - X_2 \geq -a \right]  \bigg).
    \end{align*}
Let \(0\leq A \leq |S^{{{{{*}}}}}_{uv}|-a\) be an arbitrary constant. Then,
    \begin{align}\label{UUEEGGEEUERFEFHF|Z}
    &\mathbb{P}\big(u,v \in {{{{{{{{{{C}}}}}}}}}}_{1,2} |u \not \sim v,I\left(\mathcal{G}\right)= a \big)\notag\\
    &= \mathbb{P}\bigg(\left[ X_1\leq A \right] \land \left[|S^{(1)}_{uv} \cap \Gamma(u) | - |S^{(2)}_{uv} \cap \Gamma(u) | - X_1 \geq -a \right]\land\notag\\
    &\qquad\quad\left[|S^{(1)}_{uv} \cap \Gamma(v) | - |S^{(2)}_{uv} \cap \Gamma(v) ) | - X_2 \geq -a \right]\bigg)\notag\\
    &+\mathbb{P}\bigg(\left[ X_1 > A \right]\land \left[|S^{(1)}_{uv} \cap \Gamma(u) | - |S^{(2)}_{uv} \cap \Gamma(u) | - X_1 \geq -a \right]\land\notag\\
    &\quad\qquad\left[|S^{(1)}_{uv} \cap \Gamma(v) | - |S^{(2)}_{uv} \cap \Gamma(v) ) | - X_2 \geq -a \right]\bigg)\notag\\
    &\leq \mathbb{P}\bigg(\left[ X_1\leq A \right] \land \left[|S^{(1)}_{uv} \cap \Gamma(u) | - |S^{(2)}_{uv} \cap \Gamma(u) | - X_1 \geq -a \right]\land\notag\\
    &\quad\qquad\left[|S^{(1)}_{uv} \cap \Gamma(v) | - |S^{(2)}_{uv} \cap \Gamma(v) ) | - X_2 \geq -a \right]    \bigg) +\mathbb{P}(X_1 > A).
\end{align}

Suppose \(X_1\leq A\). Let \(B_1,B_2,\dots,B_{A - X_1}\) be a sequence of i.i.d. $\mathrm{Ber}(p)$ random variables which are independent of \(X_2\). Then, define \(X_3 := X_2 +B_1+B_2 + \dots +B_{A-X_1}\). Then, \(X_3\geq X_2\) and \(X_3 \sim \mathrm{Bin}(|S^{{{{{*}}}}}_{uv}|-a -A,p)\). Therefore, \(X_3\) is independent of \(X_1\). By Equation \ref{UUEEGGEEUERFEFHF|Z}, we have:
\begin{align*}
    &\mathbb{P}(u,v \in {{{{{{{{{{C}}}}}}}}}}_{1,2} |u \not \sim v,I\left(\mathcal{G}\right)= a )\\
&\leq \mathbb{P}\bigg(\left[ X_1\leq A \right] \land \left[|S^{(1)}_{uv} \cap \Gamma(u) | - |S^{(2)}_{uv} \cap \Gamma(u) | - X_1 \geq -a \right]\land\\
&   \qquad\quad \left[|S^{(1)}_{uv} \cap \Gamma(v) | - |S^{(2)}_{uv} \cap \Gamma(v) ) | - X_3 \geq -a \right] 
    \bigg)+\mathbb{P}(X_1 > A)\\
    &\leq \mathbb{P}\bigg(\left[|S^{(1)}_{uv} \cap \Gamma(u) | - |S^{(2)}_{uv} \cap \Gamma(u) | - X_1 \geq -a \right]\land\\
& \qquad\quad   \left[|S^{(1)}_{uv} \cap \Gamma(v) | - |S^{(2)}_{uv} \cap \Gamma(v) ) | - X_3 \geq -a \right]  \bigg)   +\mathbb{P}(X_1 > A).
\end{align*}

We know that \(|S^{(1)}_{uv} \cap \Gamma(u) |\), \(|S^{(2)}_{uv} \cap \Gamma(u) |\), \(X_1\), \(|S^{(1)}_{uv} \cap \Gamma(v) |\), \(|S^{(2)}_{uv} \cap \Gamma(v) ) |\) and \( X_3\) are independent random variables. Therefore,\\
\(\left[|S^{(1)}_{uv} \cap \Gamma(u) | - |S^{(2)}_{uv} \cap \Gamma(u) | - X_1 \geq -a  \right]\) and \(\left[|S^{(1)}_{uv} \cap \Gamma(v) | - |S^{(2)}_{uv} \cap \Gamma(v) ) | - X_3 \geq -a \right]\) are independent events. Therefore,
\begin{align} \label{forst_vorstion_pepepepepepeppefhwdelfnjwrefjwbgjwrjwnjfkjknwdjvnjwk_forst_vorsion}
    \mathbb{P}(u,v \in {{{{{{{{{{C}}}}}}}}}}_{1,2} |u \not \sim v,I\left(\mathcal{G}\right)= a )\notag
&\leq \mathbb{P}\bigg(|S^{(1)}_{uv} \cap \Gamma(u) | - |S^{(2)}_{uv} \cap \Gamma(u) | - X_1 \geq -a \bigg)\notag\\
   &\times \mathbb{P}\bigg( |S^{(1)}_{uv} \cap \Gamma(v) | - |S^{(2)}_{uv} \cap \Gamma(v) ) | - X_3 \geq -a  \bigg) \notag\\
   &+\mathbb{P}(X_1 > A)
   \end{align}
Since 
\begin{align}\label{why___am____I+++++++so++__++__angry}
    &\mathbb{P}\bigg(|S^{(1)}_{uv} \cap \Gamma(u) | - |S^{(2)}_{uv} \cap \Gamma(u) | - X_1 \geq -a \bigg)\notag\\
    &= \mathbb{P}\bigg(|S^{(1)}_{uv} \cap \Gamma(u) | - |S^{(2)}_{uv} \cap \Gamma(u) | - |S^{{{{{*}}}}}_{uv}\cap \left(\Gamma(u)\backslash\Gamma(v)\right)| \geq -a \bigg)\notag\\
    &= \mathbb{P}\bigg(|S^{(1)}_{uv} \cap \Gamma(u) | - |S^{(2)}_{uv} \cap \Gamma(u) | - |S^{{{{{*}}}}}_{uv}\cap \left(\Gamma(u)\backslash\Gamma(v)\right)| \geq -I(\mathcal{G})\notag\bigg|I(\mathcal{G})=a \bigg)\notag\\
    &= \mathbb{P}\big(u \in {{{{{{{{{{C}}}}}}}}}}_{1,2}\big| u \not \sim v \land I(\mathcal{G} ) =a\big),
    \end{align}
 by Equation \ref{forst_vorstion_pepepepepepeppefhwdelfnjwrefjwbgjwrjwnjfkjknwdjvnjwk_forst_vorsion}, we must have:
 \begin{multline}\label{pepepepepepeppefhwdelfnjwrefjwbgjwrjwnjfkjknwdjvnjwk}
    \mathbb{P}(u,v \in {{{{{{{{{{C}}}}}}}}}}_{1,2} |u \not \sim v,I\left(\mathcal{G}\right)= a )\leq \mathbb{P}(u \in {{{{{{{{{{C}}}}}}}}}}_{1,2} |u \not \sim v,I\left(\mathcal{G}\right)= a ) \\
   \times\mathbb{P}\bigg( |S^{(1)}_{uv} \cap \Gamma(v) | - |S^{(2)}_{uv} \cap \Gamma(v) | - X_3 \geq -a  \bigg)    +\mathbb{P}(X_1 > A).
\end{multline}
Let \(\theta := |S^{(1)}_{uv} \cap \Gamma(v) | - |S^{(2)}_{uv} \cap \Gamma(v) |\), and let \(X_4 \sim \mathrm{Bin}(A,p)\) be independent of every previously defined random variable. Then,
\begin{align}\label{asdfasdfasdfasdfasdfasdfasdfjfjfjjfjfjfjwfjewfuonhwrjoi}
\mathbb{P}(\theta - X_3 \geq -a) - \mathbb{P}(\theta -X_3 - X_4 \geq -a)& = \mathbb{P}(\theta +a\geq X_3> \theta+a-X_4)\notag\\
&= \sum_{i=1}^{A}\mathbb{P}(X_4 = i)\mathbb{P}(a\geq X_3 - \theta > a-i).
\end{align}

Since \(|S^{(1)}_{uv} \cap \Gamma(v) |\sim \mathrm{Bin}(|S^{(1)}_{uv}|,p)\), \(|S^{(2)}_{uv} \cap \Gamma(v)  |\sim \mathrm{Bin}(|S^{(2)}_{uv}|,p)\) and \(X_3\sim \mathrm{Bin}(|S^{{{{{*}}}}}_{uv}|-a -A,p)\) are independent random variables,  by Corollary \ref{single_use_corollarhy_for+__far_variance+calc_calc_cakldfdsfdfdsf}, 

\begin{align*}
\mathbb{P}(X_3 - \theta = j) &\leq \frac{1.12}{\sqrt{p(1-p)(|S^{(1)}_{uv}|+|S^{(2)}_{uv}|+|S^{{{{{*}}}}}_{uv}|-a -A)}}\\
&\leq \frac{1.12}{\sqrt{p(1-p)(|S^{(1)}_{uv}|+|S^{(2)}_{uv}|)}}\\
    &= \frac{1.12}{\sqrt{p(1-p)(n-2-|S^{{{{{*}}}}}_{uv}|)}}.
\end{align*}
 for any value of $j$. Therefore, for each $1 \leq i\leq A$,
 \begin{align*}
     \mathbb{P}(a\geq X_3 - \theta > a-i) \leq \frac{1.12i}{\sqrt{p(1-p)(n-2-|S^{{{{{*}}}}}_{uv}|)}}.
 \end{align*}
Therefore,
\begin{align*}
    \sum_{i=1}^{A}\mathbb{P}(X_4 = i)\mathbb{P}(a\geq X_3 - \theta > a-i) &\leq \frac{1.12}{\sqrt{p(1-p)(n-2-|S^{{{{{*}}}}}_{uv}|)}} \sum_{i=1}^A i \mathbb{P}(X_4=i)\\
    &= \frac{1.12}{\sqrt{p(1-p)(n-2-|S^{{{{{*}}}}}_{uv}|)}} \times \mathbb{E}[X_4]\\
    &= \frac{1.12A\sqrt{p}}{\sqrt{(1-p)(n-2-|S^{{{{{*}}}}}_{uv}|)}}\\
    &\leq \frac{2A\sqrt{p}}{\sqrt{n-2-|S^{{{{{*}}}}}_{uv}|}},
\end{align*}
where the final line follows from the fact that $p\leq 1/4$. Therefore, by Equation \ref{asdfasdfasdfasdfasdfasdfasdfjfjfjjfjfjfjwfjewfuonhwrjoi},
\begin{align}\label{foiwqenguwguwhfwf}
    \mathbb{P}(\theta - X_3 \geq -a) \leq \mathbb{P}(\theta -X_3 - X_4 \geq -a) +\frac{2A\sqrt{p}}{\sqrt{n-2-|S^{{{{{*}}}}}_{uv}|}} .
\end{align}
We know that \(X_3 + X_4 \sim \mathrm{Bin}(|S^{{{{{*}}}}}_{uv}|-a,p)\). Since \(p >\frac{p}{1+p}\), it is possible to construct a random variable, $Y_1$, such that \(Y_1 \leq X_3+X_4\) everywhere, and \(Y_1 \sim \mathrm{Bin}(|S^{{{{{*}}}}}_{uv}|-a,\frac{p}{1+p})\). Therefore,
\begin{align*}
    \mathbb{P}(\theta -X_3 - X_4 \geq -a)& \leq \mathbb{P}(\theta -  Y_1 \geq -a)
    =  \mathbb{P}(\theta - X_1 \geq -a).
\end{align*}
Therefore, by Equation \ref{foiwqenguwguwhfwf},
\begin{align*}
    \mathbb{P}(\theta - X_3 \geq -a) &\leq \mathbb{P}(\theta - X_1 \geq -a) +\frac{2A\sqrt{p}}{\sqrt{n-2-|S^{{{{{*}}}}}_{uv}|}}\\
    & = \mathbb{P}(u,v \in {{{{{{{{{{C}}}}}}}}}}_{1,2} |u \not \sim v,I\left(\mathcal{G}\right)= a )\notag
    +\frac{2A\sqrt{p}}{\sqrt{n-2-|S^{{{{{*}}}}}_{uv}|}},
\end{align*}
where the final equality follows from Equation \ref{why___am____I+++++++so++__++__angry}. Upon substituting this into Equation \ref{pepepepepepeppefhwdelfnjwrefjwbgjwrjwnjfkjknwdjvnjwk}, we find that
\begin{align*}
    &\mathbb{P}(u,v \in {{{{{{{{{{C}}}}}}}}}}_{1,2} |u \not \sim v,I\left(\mathcal{G}\right)= a )\\
    &\leq \mathbb{P}(u \in {{{{{{{{{{C}}}}}}}}}}_{1,2} |u \not \sim v,I\left(\mathcal{G}\right)= a )^2 \\
    &+\mathbb{P}(u \in {{{{{{{{{{C}}}}}}}}}}_{1,2} |u \not \sim v,I\left(\mathcal{G}\right)= a )\times
  \frac{2A\sqrt{p}}{\sqrt{n-2-|S^{{{{{*}}}}}_{uv}|}}   +\mathbb{P}(X_1 > A)\\
  &\leq \mathbb{P}(u \in {{{{{{{{{{C}}}}}}}}}}_{1,2} |u \not \sim v,I\left(\mathcal{G}\right)= a )^2 +
  \frac{2A\sqrt{p}}{\sqrt{n-2-|S^{{{{{*}}}}}_{uv}|}}   +\mathbb{P}(X_1 > A).
\end{align*}

Let \(Z_1 \sim \mathrm{Bin}\left(\frac{4\sqrt{n}}{\sqrt{p}},p\right)\). Since \(|S^{{{{{*}}}}}_{uv}|-a \leq \frac{4\sqrt{n}}{\sqrt{p}}\), \(p\geq\frac{p}{1+p}\) and \(X_1 \sim \mathrm{Bin}(|S^{{{{{*}}}}}_{uv}|-a,\frac{p}{1+p})\), \[
\mathbb{P}(Z_1>A)\geq \mathbb{P}(X_1>A).
\]
Let \(A = 12\sqrt{np} = 3\mathbb{E}[Z_1]\). Then, by a Chernoff bound,
\begin{align*}
    \mathbb{P}(Z_1>A)\leq \exp \left[{-4\sqrt{np}}\right] .
\end{align*}
 We know that \(|S^{{{{{*}}}}}_{uv}|\leq \frac{4\sqrt{n}}{\sqrt{p}}\leq \frac{4n}{\log n}\), since \(\frac{\log n}{n}\leq p\). Thus for sufficiently large $n$,
 \[
 n-2-|S^{{{{{*}}}}}_{uv}|\geq \frac{n}{2}.
 \]
Therefore,
\[
 \frac{2A\sqrt{p}}{\sqrt{n-2-|S^{{{{{*}}}}}_{uv}|}} = \mathcal{O}\left(\frac{\sqrt{p}}{\sqrt{n}}\right) = \mathcal{O}\left(p\right).
\]
Therefore,
\begin{multline}\label{Multinowaynfewlfjlf__E-E_E}
        \mathbb{P}(u,v \in {{{{{{{{{{C}}}}}}}}}}_{1,2} |u \not \sim v,I\left(\mathcal{G}\right)= a )
\leq \mathbb{P}(u \in {{{{{{{{{{C}}}}}}}}}}_{1,2} |u \not \sim v,I\left(\mathcal{G}\right)= a )^2 +  \mathcal{O}(p)+\exp \left[-4\sqrt{np}\right]
\end{multline}
By the law of total probability,
\begin{align}\label{totototootoprorooker__11__1}
    &\mathbb{P}(u,v \in {{{{{{{{{{C}}}}}}}}}}_{1,2}\big|I\left(\mathcal{G}\right)= a )\notag\\
  &  = (1-p)\mathbb{P}(u,v \in {{{{{{{{{{C}}}}}}}}}}_{1,2} \big|u \not \sim v,I\left(\mathcal{G}\right)= a )\notag\\
    &\quad+ p\mathbb{P}(u ,v\in {{{{{{{{{{C}}}}}}}}}}_{1,2} \big|u  \sim v,I\left(\mathcal{G}\right)= a )\notag\\
&\leq \mathbb{P}(u,v \in {{{{{{{{{{C}}}}}}}}}}_{1,2} \big|u \not \sim v,I\left(\mathcal{G}\right)= a ) + p.
\end{align}

Similarly,
\begin{align*}
    \mathbb{P}(u \in {{{{{{{{{{C}}}}}}}}}}_{1,2}\big|I\left(\mathcal{G}\right)= a )
      &  = (1-p)\mathbb{P}(u \in {{{{{{{{{{C}}}}}}}}}}_{1,2} \big|u \not \sim v,I\left(\mathcal{G}\right)= a )\\
       &\quad+ p\mathbb{P}(u \in {{{{{{{{{{C}}}}}}}}}}_{1,2} \big|u  \sim v,I\left(\mathcal{G}\right)= a )\\
      &\geq (1-p)\mathbb{P}(u \in {{{{{{{{{{C}}}}}}}}}}_{1,2} \big|u \not \sim v,I\left(\mathcal{G}\right)= a )\\
      &\geq \mathbb{P}(u \in {{{{{{{{{{C}}}}}}}}}}_{1,2} \big|u \not \sim v,I\left(\mathcal{G}\right)= a ) -p.
\end{align*}
Therefore,
\begin{align}\label{totototootoprorooker__11__2}
    \mathbb{P}(u \in {{{{{{{{{{C}}}}}}}}}}_{1,2}\big|u \not \sim v \land I\left(\mathcal{G}\right)= a )
    &\leq \bigg(\mathbb{P}(u \in {{{{{{{{{{C}}}}}}}}}}_{1,2} \big|I\left(\mathcal{G}\right)= a ) + p\bigg)^2\notag\\
    &= \mathbb{P}(u \in {{{{{{{{{{C}}}}}}}}}}_{1,2} \big|I\left(\mathcal{G}\right)= a )^2\notag\\
    &+2p\mathbb{P}(u \in {{{{{{{{{{C}}}}}}}}}}_{1,2} \big|I\left(\mathcal{G}\right)= a ) +p^2\notag\\
    &\leq \mathbb{P}(u \in {{{{{{{{{{C}}}}}}}}}}_{1,2} \big|I\left(\mathcal{G}\right)= a )^2 +3p.
    \end{align}
Therefore, by substituting Equations \ref{totototootoprorooker__11__1} and \ref{totototootoprorooker__11__2} into Equation \ref{Multinowaynfewlfjlf__E-E_E}, we find that 
\begin{equation*}
    \mathbb{P}(u,v \in {{{{{{{{{{C}}}}}}}}}}_{1,2} |I\left(\mathcal{G}\right)= a )
\leq \mathbb{P}(u \in {{{{{{{{{{C}}}}}}}}}}_{1,2} |I\left(\mathcal{G}\right)= a )^2 +  \mathcal{O}(p)+\exp \left[-4\sqrt{np}\right].
\end{equation*}
Since we are implicitly conditioning on $H$ in the above equation, this is the desired result.
\end{proof}

\begin{lemma}\label{SOMETHING__TH__TH+__++_+IN_PLAN}
Suppose that \(u \neq v \in C_{1,0}\), \(\frac{1}{4}\geq p\geq \frac{\log n}{n}\) and $n\geq 1048$. Then,
  \begin{align*}
    \mathbb{E}\mathbf{Cov}\left(\mathbbm{1}_{{{{{{{{{{{C}}}}}}}}}}_{1,2}}(u),\mathbbm{1}_{{{{{{{{{{{C}}}}}}}}}}_{1,2}}(v)\bigg|\mathcal{G}\right) & \leq \mathcal{O}(p)+\exp \left[-4\sqrt{np}\right] .
\end{align*}
\end{lemma}
\begin{proof}

We can write:
 \begin{align} \label{variance_day_2_red_pair_first_equation}
    \mathbf{Cov}\left(\mathbbm{1}_{{{{{{{{{{{C}}}}}}}}}}_{1,2}}(u),\mathbbm{1}_{{{{{{{{{{{C}}}}}}}}}}_{1,2}}(v)|\mathcal{G}\right)    &= \mathbb{P}(u,v \in {{{{{{{{{{C}}}}}}}}}}_{1,2}|\mathcal{G}) - \mathbb{P}(u \in {{{{{{{{{{C}}}}}}}}}}_{1,2}|\mathcal{G}) \mathbb{P}(v \in {{{{{{{{{{C}}}}}}}}}}_{1,2}|\mathcal{G})\notag\\
    &=   \mathbb{P}(u,v \in {{{{{{{{{{C}}}}}}}}}}_{1,2}|\mathcal{G}) - \mathbb{P}(u \in {{{{{{{{{{C}}}}}}}}}}_{1,2}|\mathcal{G})^2.
\end{align}
Therefore by Proposition \ref{cov_calc_in_appendix_C}, 
\begin{align*}
    \mathbf{Cov}\left(\mathbbm{1}_{{{{{{{{{{{C}}}}}}}}}}_{1,2}}(u),\mathbbm{1}_{{{{{{{{{{{C}}}}}}}}}}_{1,2}}(v)\bigg|\mathcal{G}\right)\leq \mathcal{O}(p)+\exp \left[-4\sqrt{np}\right],
\end{align*}
provided $\mathcal{G}$ is such that \(|S^{{{{{*}}}}}_{uv}|\leq \frac{4\sqrt{n}}{\sqrt{p}}\). If this is not the case, then we can use the following bound:
\[
\mathbf{Cov}\left(\mathbbm{1}_{{{{{{{{{{{C}}}}}}}}}}_{1,2}}(u),\mathbbm{1}_{{{{{{{{{{{C}}}}}}}}}}_{1,2}}(v)\bigg|\mathcal{G}\right)\leq 1.
\]
Therefore,
\begin{align*}
    \mathbb{E}\mathbf{Cov}\left(\mathbbm{1}_{{{{{{{{{{{C}}}}}}}}}}_{1,2}}(u),\mathbbm{1}_{{{{{{{{{{{C}}}}}}}}}}_{1,2}}(v)\bigg|\mathcal{G}\right) &\leq 
      \bigg(\mathcal{O}(p)+\exp \left[-4\sqrt{np}\right]\bigg)\mathbb{P}\left(|S^{{{{{*}}}}}_{uv}|\leq \frac{4\sqrt{n}}{\sqrt{p}}\right)\notag \\
     & + \mathbb{P}\left(|S^{{{{{*}}}}}_{uv}|> \frac{4\sqrt{n}}{\sqrt{p}}\right) \notag\\
     &\leq \mathcal{O}(p)+\exp \left[-4\sqrt{np}\right] + \mathbb{P}\left(|S^{{{{{*}}}}}_{uv}|> \frac{4\sqrt{n}}{\sqrt{p}}\right) .
\end{align*}
By Proposition \ref{tinyproopvariacne}, we have
\begin{align*}
    \mathbb{E}\mathbf{Cov}\left(\mathbbm{1}_{{{{{{{{{{{C}}}}}}}}}}_{1,2}}(u),\mathbbm{1}_{{{{{{{{{{{C}}}}}}}}}}_{1,2}}(v)\bigg|\mathcal{G}\right) & \leq \mathcal{O}(p)+\exp \left[-4\sqrt{np}\right] ,
\end{align*}
completing the proof.
\end{proof}

\begin{lemma}\label{THIRD_IN_PLAN_LEMMA} Suppose $n \geq \exp \left[ 3*10^6\right]$, $\frac{\log n}{n} \leq p \leq \frac{1}{4}$ and $\frac{10}{p} \geq \Delta \geq 1$. Then,
    \[
\mathbf{Var}\left(\mathbb{P}\left(v \in {{{{{{{{{{C}}}}}}}}}}_{1,2} \big|\mathcal{G}\right)\right) = \mathcal{O}\left(p + \frac{1}{np}\right).
\]
    
\end{lemma}
\begin{proof}

By Lemma \ref{short_lemma_on_importance_of_day_2_p_2_+definitions}, if \( u \not \sim v\), \(v \in {{{{{{{{{{C}}}}}}}}}}_{1,2}\) iff 
\begin{align*}
    |S^{(1)}_{uv} \cap \Gamma(u) | - |S^{(2)}_{uv} \cap \Gamma(u) | - |S^{{{{{*}}}}}_{uv}\cap \Gamma(u)\backslash\Gamma(v)) | + |\Gamma(u)\cap\Gamma(v) \cap S^{{{{{*}}}}}_{uv}| \geq 0.
\end{align*}
Define
\begin{align*}
    \mathcal{A} := \left[|S^{(1)}_{uv} \cap \Gamma(u) | - |S^{(2)}_{uv} \cap \Gamma(u) | - |S^{{{{{*}}}}}_{uv}\cap \Gamma(u)\backslash\Gamma(v)) | + |\Gamma(u)\cap\Gamma(v) \cap S^{{{{{*}}}}}_{uv}| \geq 0\right].
\end{align*}

It is easy to see that
\begin{align*}
     \mathbb{P}\left(v \in {{{{{{{{{{C}}}}}}}}}}_{1,2} \big|\mathcal{G}\right)  &= \mathbb{P}\left(v \in {{{{{{{{{{C}}}}}}}}}}_{1,2} \big|\mathcal{G},u\not\sim v \right)\mathbb{P}(u \not\sim v) +   \mathbb{P}\left(v \in {{{{{{{{{{C}}}}}}}}}}_{1,2} \big|\mathcal{G},u\sim v \right)\mathbb{P}(u \sim v) \\
    &= (1-p)\mathbb{P}\left(v \in {{{{{{{{{{C}}}}}}}}}}_{1,2} \big|\mathcal{G},u\not\sim v \right) +   p\mathbb{P}\left(v \in {{{{{{{{{{C}}}}}}}}}}_{1,2} \big|\mathcal{G},u\sim v \right)\\
     &\leq \mathbb{P}\left(v \in {{{{{{{{{{C}}}}}}}}}}_{1,2} \big|\mathcal{G},u\not\sim v \right) +   p\\
    &= \mathbb{P}(\mathcal{A}\big| \mathcal{G}) + p.
\end{align*}
Likewise 
\begin{align*}
     \mathbb{P}\left(v \in {{{{{{{{{{C}}}}}}}}}}_{1,2} \big|\mathcal{G}\right)  
    &= (1-p)\mathbb{P}\left(v \in {{{{{{{{{{C}}}}}}}}}}_{1,2} \big|\mathcal{G},u\not\sim v \right) +   p\mathbb{P}\left(v \in {{{{{{{{{{C}}}}}}}}}}_{1,2} \big|\mathcal{G},u\sim v \right)\\
    &\geq \mathbb{P}\left(v \in {{{{{{{{{{C}}}}}}}}}}_{1,2} \big|\mathcal{G},u\not\sim v \right) - p\mathbb{P}\left(v \in {{{{{{{{{{C}}}}}}}}}}_{1,2} \big|\mathcal{G},u\not\sim v \right)\\
    &\geq \mathbb{P}\left(v \in {{{{{{{{{{C}}}}}}}}}}_{1,2} \big|\mathcal{G},u\not\sim v \right) - p\\
    &= \mathbb{P}(\mathcal{A}\big| \mathcal{G}) - p.
\end{align*}
Therefore, 
\begin{align*}
    \mathbb{P}(\mathcal{A}\big| \mathcal{G}) - p\leq \mathbb{P}\left(v \in {{{{{{{{{{C}}}}}}}}}}_{1,2} \big|\mathcal{G}\right)\leq \mathbb{P}(\mathcal{A}\big| \mathcal{G}) + p.
\end{align*}

Furthermore, by Gaussian approximation (see Proposition \ref{Gaussian_approximation_for_first_variance_calc_of_2_colours} in Appendix \ref{appendix_E_2_colours}),

    \begin{align*}
   \bigg| \mathbb{P}\left(\mathcal{A}\big|\mathcal{G}\right) - \Phi\left(\mathcal{M}\right)\bigg|
   \leq \min \left\{1, \frac{C_{BE}}{\sqrt{p(1-p)(n-2-I(\mathcal{G}))}}\right\}.
    \end{align*}
    where \[
    \mathcal{M}:= \frac{|S^{(1)}_{uv}|p-|S^{(2)}_{uv}|p -(|S^{{{{{*}}}}}_{uv}|-I(\mathcal{G}))\frac{p}{1+p} +I(\mathcal{G})}{\sqrt{p(1-p)(|S^{(1)}_{uv}|+|S^{(2)}_{uv}|)+ \frac{p}{1+p}(1-\frac{p}{1+p})(|S^{{{{{*}}}}}_{uv}|-I(\mathcal{G}))}}.
    \]
    Furthermore, since \(0\leq \Phi(\mathcal{M}),\mathbb{P}\left(v \in {{{{{{{{{{C}}}}}}}}}}_{1,2} \big|\mathcal{G}\right) \leq 1\), we must have
    \begin{align*}
        \left|\mathbb{P}\left(v \in {{{{{{{{{{C}}}}}}}}}}_{1,2} \big|\mathcal{G}\right) - \Phi\left(\mathcal{M}\right)\right| \leq\min \left\{1,  p+\frac{C_{BE}}{\sqrt{p(1-p)(n-2-I(\mathcal{G}))}}\right\}.
    \end{align*}
Therefore, there is a random variable, \(\epsilon\), which takes values in $[-1,1]$ such that
\begin{align*}
    \mathbb{P}\left(v \in {{{{{{{{{{C}}}}}}}}}}_{1,2} \big|\mathcal{G}\right) = \Phi\left(\mathcal{M}\right) +\left( p+\frac{C_{BE}}{\sqrt{p(1-p)(n-2-I(\mathcal{G}))}}\right) \epsilon.
\end{align*}
Define \begin{align*}
     \mathcal{L}:= \min\left\{ 1, p+\frac{C_{BE}}{\sqrt{p(1-p)(n-2-I(\mathcal{G}))}}\right\}.
\end{align*}
Then, \begin{align}\label{thebowlnotfixed3}
    \mathbf{Var}\left(\mathbb{P}\left(v \in {{{{{{{{{{C}}}}}}}}}}_{1,2} \big|\mathcal{G}\right)\right) &= \mathbf{Var}\left(\Phi\left(\mathcal{M}\right) +\mathcal{L} \epsilon\right)\notag\\
    &\leq 2\mathbf{Var}\left(\Phi\left(\mathcal{M}\right)\right) + 2\mathbf{Var}\left(\mathcal{L} \epsilon\right).
\end{align}

We will proceed by finding bounds for both variances in the above sum, starting with \(\mathbf{Var}\left(\Phi\left(\mathcal{M}\right)\right)\).\\

Define the event 
\begin{align*}
    \mathcal{E} := \left[ \mathbb{E}\left[I(\mathcal{G})\right] - \sqrt{21}\sqrt[4]{pn} \leq I(\mathcal{G}) \leq \mathbb{E}\left[I(\mathcal{G})\right] +\sqrt{21}\sqrt[4]{pn}
    \right].
\end{align*}
Then, by Chebyshev's inequality and Proposition \ref{currentlygoinginappendixBvariance4parttechnical},
\[
\mathbb{P}(\mathcal{E})\geq 1-p.
\]
Since $\Phi$ is supported on $[0,1]$, after some basic algebraic manipulation (see Proposition \ref{Algebraic_manipulation_for_lemma_THIRD_IN_PLAN_appendix_C} in Appendix \ref{appendix_E_2_colours}), we find that 
\begin{align*}
    \mathbf{Var}\left(\Phi(\mathcal{M})\right) &\leq \mathbf{Var}(\Phi(\mathcal{M})|\mathcal{E}) + 3 \mathbb{P}(\neg\mathcal{E}).
\end{align*}
By Lemma \ref{variance_of_phi} in Appendix \ref{appendix_A} (using the fact that \(\Phi\) is a contraction), \(\mathbf{Var}(\Phi(\mathcal{M})|\mathcal{E}) \leq \mathbf{Var}(\mathcal{M}|\mathcal{E})\). Therefore,
\begin{align}\label{Ibrokethebowl__NEGative___OnE}
\mathbf{Var}\left(\Phi(\mathcal{M})\right) &\leq \mathbf{Var}(\mathcal{M}|\mathcal{E}) + 3 \mathbb{P}(\neg\mathcal{E})\notag\\
&= \mathbb{E}[{\mathcal{M}^2}|\mathcal{E}]-\mathbb{E}[\mathcal{M}|\mathcal{E}]^2+ \mathcal{O}(p).
\end{align}

We will proceed by bounding \(\mathbb{E}[{\mathcal{M}^2}|\mathcal{E}] \) and \(\mathbb{E}[\mathcal{M}|\mathcal{E}]^2\).\\

Define \(f(p):= p(1-p) - \frac{p}{1+p}\left(1-\frac{p}{1+p}\right)  = \mathcal{O}(p^2) \). We can express the denominator of $\mathcal{M}$ in the following way:
\begin{align*}
    &\sqrt{p(1-p)(|S^{(1)}_{uv}|+|S^{(2)}_{uv}|)+ \frac{p}{1+p}(1-\frac{p}{1+p})(|S^{{{{{*}}}}}_{uv}|-I(\mathcal{G}))}\\
    &= \sqrt{p(1-p)\left(|S^{(1)}_{uv}|+|S^{(2)}_{uv}|+|S^{{{{{*}}}}}_{uv}|\right) - f(p)|S^{{{{{*}}}}}_{uv}| - \frac{p}{(1+p)^2 }I(\mathcal{G})}\\
    &= \sqrt{p(1-p)(n-2) - f(p)|S^{{{{{*}}}}}_{uv}|-\frac{p}{(1+p)^2}I(\mathcal{G})}.
\end{align*}
Define \[n_{min} := 
\sqrt{p(1-p)(n-2) - f(p)n-\frac{p}{(1+p)^2}\mathbb{E}\left[I(\mathcal{G})\right] - \frac{\sqrt{21}p^{5/4}n^{1/4}}{(1+p)^2}}
\]
and 
\[
n_{max} := \sqrt{p(1-p)(n-2) -\frac{p}{(1+p)^2}\mathbb{E}\left[I(\mathcal{G})\right] +\frac{\sqrt{21}p^{5/4}n^{1/4}}{(1+p)^2}}.
\]
Then, conditioning on the event \(\mathcal{E}\), we find that 
\(
     \frac{\mathcal{N}}{n_{min}}\geq \mathcal{M} \geq \frac{\mathcal{N}}{n_{max}},
\)
where \(
\mathcal{N} := |S^{(1)}_{uv}|p-|S^{(2)}_{uv}|p -\frac{p}{1+p} |S^{{{{{*}}}}}_{uv}| + \frac{1}{1+p} I(\mathcal{G}).
\)
Therefore, 
\begin{align}\label{Ibrokethebowl__ONE}
    \mathbb{E}[{\mathcal{M}^2}|\mathcal{E}]-\mathbb{E}[\mathcal{M}|\mathcal{E}]^2 &\leq \mathbb{E}\left[\left(\frac{\mathcal{N}}{n_{min}}\right)^2\bigg|\mathcal{E}\right]-\mathbb{E}\left[\frac{\mathcal{N}}{n_{max}}\bigg|\mathcal{E}\right]^2\notag\\
    &= \mathbb{E}\left[\left(\frac{\mathcal{N}}{n_{min}}\right)^2\bigg|\mathcal{E}\right]-\mathbb{E}\left[\left(\frac{\mathcal{N}}{n_{max}}\right)^2\bigg|\mathcal{E}\right] + \mathbf{Var}\left(\frac{\mathcal{N}}{n_{max}}\bigg| \mathcal{E}\right)\notag\\
    &= \frac{n_{max}^2 - n_{min}^2}{n_{min}^2n_{max}^2} \mathbb{E}\left[    \mathcal{N}^2   \big| \mathcal{E} \right] + \frac{1}{n^2_{max}}\mathbf{Var}\left(\mathcal{N}\big| \mathcal{E}\right)\notag\\
    &\leq \frac{n_{max}^2 - n_{min}^2}{n_{min}^4} \mathbb{E}\left[    \mathcal{N}^2   \big| \mathcal{E} \right] + \frac{1}{n^2_{min}}\mathbf{Var}\left(\mathcal{N}\big| \mathcal{E}\right)\notag\\
    &\leq \frac{1}{\mathbb{P}(\mathcal{E})}\left(\frac{n_{max}^2 - n_{min}^2}{n_{min}^4} \mathbb{E}\left[    \mathcal{N}^2    \right] + \frac{1}{n^2_{min}}\mathbf{Var}\left(\mathcal{N}\right)\right),
\end{align}
where the final line follows from Propositions \ref{variance_bound_conditioned_on_an_event_AP_A} and \ref{expectation_bound_conditioned_on_an_event_AP_A}.
We will now proceed by bounding \(\mathbb{E}\left[\mathcal{N}^2\right]\). It is easy to see that
\begin{align*}
    \mathcal{N}^2 &= \left(p\left(|S^{(1)}_{uv}|-|S^{(2)}_{uv}| -|S^{{{{{*}}}}}_{uv}|\right) + \frac{p^2}{1+p}|S^{{{{{*}}}}}_{uv}|+ \frac{1}{1+p} I(\mathcal{G})\right)^2\\
    &\leq 3\left( p^2\left(|S^{(1)}_{uv}|-|S^{(2)}_{uv}| -|S^{{{{{*}}}}}_{uv}|\right)^2 +\frac{p^4}{(1+p)^2}|S^{{{{{*}}}}}_{uv}|^2 + \frac{1}{(1+p)^2} I(\mathcal{G})^2
 \right)\\
 &\leq 3\left( p^2\left(|S^{(1)}_{uv}|-|S^{(2)}_{uv}| -|S^{{{{{*}}}}}_{uv}|\right)^2 +{p^4}|S^{{{{{*}}}}}_{uv}|^2 +  I(\mathcal{G})^2
 \right),
    \end{align*}
where the penultimate line follows due to the AM-QM inequality. Therefore, by Proposition \ref{currentlygoinginappendixBvariance4parttechnical},
\begin{align}\label{Ibrokethebowl1}
    \mathbb{E}\left[\mathcal{N}^2\right]
    &\leq 3 p^2\mathbb{E}\left[\left(|S^{(1)}_{uv}|-|S^{(2)}_{uv}| -|S^{{{{{*}}}}}_{uv}|\right)^2\right] +3{p^4}\mathbb{E}\left[|S^{{{{{*}}}}}_{uv}|^2 \right]+  3\mathbb{E}\left[I(\mathcal{G})^2\right]\notag\\
    &= \mathcal{O}(pn +p^3n + p^{3/2}n^{1/2} )\notag\\
    &= \mathcal{O}(pn).
\end{align}
Now, we will bound the \(n_{max}^2 - n_{min}^2\) term, which we can express as
\begin{align}
    n_{max}^2 - n_{min}^2 &= f(p)n + \frac{2\sqrt{21}p^{5/4} n^{1/4}}{(1+p)^2}\notag 
   \leq \mathcal{O}(p^2)n + 2\sqrt{21}p^{5/4}n^{1/4}.
\end{align}
Since \(p >1/n\), we must have that \(p^{5/4}n^{1/4} \leq p^2n\). Therefore,
\begin{align}\label{Ibrokethebowl2}
    n_{max}^2 - n_{min}^2 &= \mathcal{O}(p^2n).
\end{align}
Now we will bound $n_{min}^2$. Since \(f(p)\leq p^2\),
\begin{align*}
    n_{min}^2 &\geq {p(1-p)(n-2)} - p^2n - p\mathbb{E}\left[I(G)\right] - \sqrt{21}p^{5/4}n^{1/4}.
\end{align*}
By Proposition \ref{currentlygoinginappendixBvariance4parttechnical}
\begin{align}\label{Ibrokethebowl3}
    n_{min}^2 &\geq  {p(1-p)(n-2)} - p^2n -2p^{5/2}n^{1/2}- \sqrt{21}p^{5/4}n^{1/4}\notag\\
    &\geq \frac{3}{4}(n-2)p - \frac{1}{4}np - \frac{1}{4}p\sqrt{n} - \sqrt{\frac{21}{2}}pn^{1/4}\notag\\
    &= pn\left(\frac{1}{2} - \frac{1}{4\sqrt{n}} - \sqrt{\frac{21}{2}}n^{-3/4} - \frac{2p}{n}\right)\notag\\
    &\geq pn\left(\frac{1}{2} - \frac{1}{4\sqrt{n}} - \sqrt{\frac{21}{2}}n^{-3/4} - \frac{2}{n}\right)\notag\\
    &= \Omega(pn).
\end{align}
By Equations \ref{Ibrokethebowl1}, \ref{Ibrokethebowl2} and \ref{Ibrokethebowl3},

\begin{align}\label{Ibrokethebowl5}
    \frac{n_{max}^2 - n_{min}^2}{n_{min}^4} \mathbb{E}\left[    \mathcal{N}^2    \right] = \mathcal{O}(p).
\end{align}
Now we move on to bounding \(\mathbf{Var}\left(\mathcal{N}\right)\). 
\begin{align*}\mathbf{Var}\left(\mathcal{N}\right) &= \mathbf{Var}\left(p\left(|S^{(1)}_{uv}|-|S^{(2)}_{uv}| -|S^{{{{{*}}}}}_{uv}|\right) + \frac{p^2}{1+p}|S^{{{{{*}}}}}_{uv}|+ \frac{1}{1+p} I(\mathcal{G})\right)\\
&\leq 3\mathbf{Var}\left(p\left(|S^{(1)}_{uv}|-|S^{(2)}_{uv}| -|S^{{{{{*}}}}}_{uv}|\right)\right) + 3\mathbf{Var}\left(\frac{p^2}{1+p}|S^{{{{{*}}}}}_{uv}|\right)+ 3 \mathbf{Var}\left(\frac{1}{1+p} I(\mathcal{G})\right)\\
&= 3p^2\mathbf{Var}\left(|S^{(1)}_{uv}|-|S^{(2)}_{uv}| -|S^{{{{{*}}}}}_{uv}|\right) + \frac{3p^4}{(1+p)^2}\mathbf{Var}\left(|S^{{{{{*}}}}}_{uv}|\right)+ \frac{3}{(1+p)^2} \mathbf{Var}\left( I(\mathcal{G})\right)\\
&\leq 3p^2\mathbf{Var}\left(|S^{(1)}_{uv}|-|S^{(2)}_{uv}| -|S^{{{{{*}}}}}_{uv}|\right) + 3p^4 \mathbf{Var}\left(|S^{{{{{*}}}}}_{uv}|\right)+ {3} \mathbf{Var}\left( I(\mathcal{G})\right).
\end{align*}
By Propositions \ref{whi_why+_do_I+ke_kp_keep_lesbel} and \ref{currentlygoinginappendixBvariance4parttechnical},
\begin{align*}
 \mathbf{Var}\left(\mathcal{N}\right) &=   \mathcal{O}(np^2) + \mathcal{O}(p^{7/2}n^{1/2}) + \mathcal{O}(p^{3/2}n^{1/2})\\
 &= \mathcal{O}(np^2 + n^{1/2}p^{3/2})\\
 &=  \mathcal{O}(n^{1/2}p^{3/2}(\sqrt{np}+1)).
\end{align*}
Since \(\sqrt{np}\geq \sqrt{\log n} \geq 1\), $\sqrt{np}+1 = \mathcal{O}(\sqrt{np}).$ Therefore, 
\(
 \mathbf{Var}\left(\mathcal{N}\right)     = \mathcal{O}(np^2).
\)
By Equation \ref{Ibrokethebowl3} we have:
\begin{align}\label{Ibrokethebowl6}
    \frac{1}{n^2_{min}}\mathbf{Var}\left(\mathcal{N}\right) = \mathcal{O}(p).
\end{align}
By Equations \ref{Ibrokethebowl__ONE}, \ref{Ibrokethebowl5} and \ref{Ibrokethebowl6},
\begin{align*}
\mathbb{E}[{\mathcal{M}^2}|\mathcal{E}]-\mathbb{E}[\mathcal{M}|\mathcal{E}]^2 = \frac{1}{\mathbb{P}(\mathcal{E}) }\times \mathcal{O}(p) = \mathcal{O}\left(\frac{p}{1-p}\right)= \mathcal{O}(p),
\end{align*}
since \(1\geq 1-p\geq \frac{3}{4}\). Therefore, by Equation \ref{Ibrokethebowl__NEGative___OnE}, 
\begin{align}\label{thebowlnotfixed2}
    \mathbf{Var}\left(\Phi(\mathcal{M})\right) = \mathcal{O}(p).
\end{align}
Now we move on to bounding \( \mathbf{Var}\left(\mathcal{L} \epsilon\right)\). By Proposition \ref{Algebraic_manipulation_for_lemma_THIRD_IN_PLAN_appendix_C} in Appendix \ref{appendix_E_2_colours},
\begin{align}\label{thebowlnotfixed1}
    \mathbf{Var}\left(\mathcal{L} \epsilon \right) &\leq \mathbf{Var}(\mathcal{L} \epsilon|\mathcal{E}) + 3 \mathbb{P}(\neg\mathcal{E}) = \mathbf{Var}(\mathcal{L} \epsilon|\mathcal{E}) +\mathcal{O}(p).
\end{align}
By Proposition \ref{currentlygoinginappendixBvariance4parttechnical}, conditioning on \(\mathcal{E}\),
\begin{align*}
    n-2-I(\mathcal{G})&\geq n- 2 - 2p^{3/2}n^{1/2}-\sqrt{21}\sqrt[4]{pn}\\
    &\geq n-2-2\sqrt{n} - \sqrt{21}\sqrt[4]{n}\\
    &\geq n/2,
\end{align*}
where the penultimate line holds because $p\leq 1$. Therefore, conditioning on $\mathcal{E}$,
\begin{align*}
  \mathcal{L}&\leq  p+\frac{\sqrt{2}C_{BE}}{\sqrt{np(1-p)}} = \mathcal{O}\left(p + \frac{1}{\sqrt{np}}\right).
\end{align*}
Therefore, \[
\mathbf{Var}(\mathcal{L} \epsilon|\mathcal{E}) = \mathcal{O}\left(\left(p + \frac{1}{\sqrt{np}}\right)^2\right) = \mathcal{O}\left(p^2 + \frac{\sqrt{p}}{\sqrt{n}} + \frac{1}{np}\right).
\]
Since \(p\leq 1\) and $p\geq 1/n$, $p^2 +\frac{\sqrt{p}}{\sqrt{n}} = \mathcal{O}(p) $. Therefore, \[
\mathbf{Var}(\mathcal{L} \epsilon|\mathcal{E})  = \mathcal{O}\left(p + \frac{1}{np}\right).
\]
Therefore, by Equations \ref{thebowlnotfixed3}, \ref{thebowlnotfixed2} and \ref{thebowlnotfixed1}, we find that
\[
\mathbf{Var}\left(\mathbb{P}\left(v \in {{{{{{{{{{C}}}}}}}}}}_{1,2} \big|\mathcal{G}\right)\right) = \mathcal{O}\left(p + \frac{1}{np}\right).
\]
\end{proof}
\begin{lemma}\label{Variance_calc_of_the_REEDDEEDFD_(((((()))))((()(}
$n \geq \exp \left[ 3*10^6\right]$, $\frac{\log n}{n} \leq p \leq \frac{1}{4}$ and $\frac{10}{p} \geq \Delta \geq 1$. Then, for $u \neq v \in C_{1,0}$,

    \begin{align*}
     \mathbf{Cov}\left(\mathbbm{1}_{{{{{{{{{{{C}}}}}}}}}}_{1,2}}(u),\mathbbm{1}_{{{{{{{{{{{C}}}}}}}}}}_{1,2}}(v)\right) \leq \mathcal{O}(p) + \mathcal{O}\left(\frac{1}{np}\right).
     \end{align*}
\end{lemma}
\begin{proof}
It is easy to see that:
\begin{align*}
    \mathbf{Cov}\left(\mathbb{E}\left[\mathbbm{1}_{{{{{{{{{{{C}}}}}}}}}}_{1,2}}(u)\right],\mathbb{E}\left[\mathbbm{1}_{{{{{{{{{{{C}}}}}}}}}}_{1,2}}(v)\big|\mathcal{G}\right]
   \right) = \mathbf{Var}\left(\mathbb{P}\left(v \in {{{{{{{{{{C}}}}}}}}}}_{1,2} \big|\mathcal{G}\right)\right) .
\end{align*}
Therefore, by the law of total covariance,
\begin{align*}
   \mathbf{Cov}\left(\mathbbm{1}_{{{{{{{{{{{C}}}}}}}}}}_{1,2}}(u),\mathbbm{1}_{{{{{{{{{{{C}}}}}}}}}}_{1,2}}(v)\right) = \mathbb{E}\left[
   \mathbf{Cov}\left(\mathbbm{1}_{{{{{{{{{{{C}}}}}}}}}}_{1,2}}(u),\mathbbm{1}_{{{{{{{{{{{C}}}}}}}}}}_{1,2}}(v)\big|\mathcal{G}\right)\right] + \mathbf{Var}\left(\mathbb{P}\left(v \in {{{{{{{{{{C}}}}}}}}}}_{1,2} \big|\mathcal{G}\right)\right) .
\end{align*}
Therefore, by Lemmas \ref{SOMETHING__TH__TH+__++_+IN_PLAN} and \ref{THIRD_IN_PLAN_LEMMA}, \begin{align*}
     \mathbf{Cov}\left(\mathbbm{1}_{{{{{{{{{{{C}}}}}}}}}}_{1,2}}(u),\mathbbm{1}_{{{{{{{{{{{C}}}}}}}}}}_{1,2}}(v)\right) \leq \mathcal{O}(p) + \exp \left[-4\sqrt{np}\right] + \mathcal{O}\left(\frac{1}{np}\right).
\end{align*}
Since \(e^{-\sqrt{x}}\leq x\) for $x\geq 1$, and $np \geq 1$, we have 
\begin{align*}
     \mathbf{Cov}\left(\mathbbm{1}_{{{{{{{{{{{C}}}}}}}}}}_{1,2}}(u),\mathbbm{1}_{{{{{{{{{{{C}}}}}}}}}}_{1,2}}(v)\right) \leq \mathcal{O}(p) + \mathcal{O}\left(\frac{1}{np}\right).
     \end{align*}
\end{proof}

\begin{lemma}\label{Variance_calc_of_the_blues_(((((()))))((()(} Suppose 
$n \geq \exp \left[ 3*10^6\right]$, $\frac{\log n}{n} \leq p \leq \frac{1}{4}$ and $\frac{10}{p} \geq \Delta \geq 1$. Then, for $u \neq v \in C_{2,0}$,
    \begin{align*}
     \mathbf{Cov}\left(\mathbbm{1}_{{{{{{{{{{{C}}}}}}}}}}_{1,2}}(u),\mathbbm{1}_{{{{{{{{{{{C}}}}}}}}}}_{1,2}}(v)\right) \leq \mathcal{O}(p) + \mathcal{O}\left(\frac{1}{np}\right).
     \end{align*}
\end{lemma}
Since the proof of Lemma \ref{Variance_calc_of_the_REEDDEEDFD_(((((()))))((()(} did not use the fact that $|C_{1,0}|\geq |C_{2,0}|$, Lemma \ref{Variance_calc_of_the_REEDDEEDFD_(((((()))))((()(} implies Lemma \ref{Variance_calc_of_the_blues_(((((()))))((()(} by symmetry.\\

\begin{lemma}\label{mainRESULTstep4VariaNCEdayTWO}
 Suppose 
$n \geq \exp \left[ 3*10^6\right]$, $\frac{\log n}{n} \leq p \leq \frac{1}{4}$ and $\frac{10}{p} \geq \Delta \geq 1$. Then,
\begin{align*}
      \mathbf{Var}\left(|{{{{{{{{{{C}}}}}}}}}}_{1,2}|\right)  =\mathcal{O}\left(n^2p\right) + \mathcal{O}\left(\frac{n}{p}\right).
  \end{align*}  
\end{lemma}
\begin{proof}

  \begin{align}\label{PPOODDODJEJEJEJFEJIFJEIJFIEJIFJEIFJIEFIJEIFJIEFJIEJIFJIEJFIJEIFJI}
      \mathbf{Var}\left(|{{{{{{{{{{C}}}}}}}}}}_{1,2}|\right) &= \mathbf{Var}\left(|{{{{{{{{{{C}}}}}}}}}}_{1,2}\cap C_{1,0}| + |{{{{{{{{{{C}}}}}}}}}}_{1,2}\cap C_{2,0}|\right)\notag\\
      &\leq 2 \mathbf{Var}\left(|{{{{{{{{{{C}}}}}}}}}}_{1,2}\cap C_{1,0}| \right) + 2 \mathbf{Var}\left(|{{{{{{{{{{C}}}}}}}}}}_{1,2}\cap C_{2,0}| \right).
  \end{align}  
  Let $c \in \{1,2\}$. Then,
  \begin{align*}
      \mathbf{Var}\left(|{{{{{{{{{{C}}}}}}}}}}_{1,2}\cap C_{c,0}| \right) 
      = \sum_{u \in C_{c,0}}\mathbf{Var}\left( \mathbbm{1}_{u \in {{{{{{{{{{C}}}}}}}}}}_{1,2}}\right) +  \sum_{u \neq v \in C_{c,0}}  \mathbf{Cov}\left(\mathbbm{1}_{{{{{{{{{{{C}}}}}}}}}}_{1,2}}(u),\mathbbm{1}_{{{{{{{{{{{C}}}}}}}}}}_{1,2}}(v)\right)
  \end{align*}
  It is easy to see that \(\mathbf{Var}\left( \mathbbm{1}_{u \in {{{{{{{{{{C}}}}}}}}}}_{1,2}}\right) = \mathbb{E}\left[\mathbbm{1}_{u \in {{{{{{{{{{C}}}}}}}}}}_{1,2}}\left(1-\mathbbm{1}_{u \in {{{{{{{{{{C}}}}}}}}}}_{1,2}}\right)\right]\leq \frac{1}{4}\). Furthermore, by either Lemma \ref{Variance_calc_of_the_REEDDEEDFD_(((((()))))((()(} or \ref{Variance_calc_of_the_blues_(((((()))))((()(}, \(\mathbf{Cov}\left(\mathbbm{1}_{{{{{{{{{{{C}}}}}}}}}}_{1,2}}(u),\mathbbm{1}_{{{{{{{{{{{C}}}}}}}}}}_{1,2}}(v)\right) \leq \mathcal{O}(p) + \mathcal{O}\left(\frac{1}{np}\right)\). Therefore,
  \begin{align*}
      \mathbf{Var}\left(|{{{{{{{{{{C}}}}}}}}}}_{1,2}\cap C_{c,0}| \right) 
     & \leq \frac{|C_{c,0}|}{4} +  \frac{|C_{c,0}|^2}{4}\left(\mathcal{O}(p) + \mathcal{O}\left(\frac{1}{np}\right)\right)\\
      &= \mathcal{O}(n) + \mathcal{O}(n^2p) + \mathcal{O}\left(\frac{n}{p}\right)\\
      &= \mathcal{O}\left(n^2p\right) + \mathcal{O}\left(\frac{n}{p}\right),
  \end{align*}
  where the final line follows because $n^2p\geq n$. By substituting this into Equation \ref{PPOODDODJEJEJEJFEJIFJEIJFIEJIFJEIFJIEFIJEIFJIEFJIEJIFJIEJFIJEIFJI}, we find \[
  \mathbf{Var}\left(|{{{{{{{{{{C}}}}}}}}}}_{1,2}| \right) =\mathcal{O}\left(n^2p\right) + \mathcal{O}\left(\frac{n}{p}\right).
  \]
\end{proof}
\section{Step 5 - Day 2 Part 3}
\begin{lemma}\label{currentlylemma9.1}
     Suppose \(\frac{10}{p} \geq \Delta \geq \max \left\{\dfrac{1}{\sqrt{p}} \exp\left[A\sqrt{\log \left(\dfrac{1}{p}\right)}\right] , Bp^{-3/2} n^{-1/2} \right\}\) (where $A$ and $B$ are constants defined in \ref{COrollary_after_RED.BIAS}) and \(\frac{\log n}{n}\leq p \leq \frac{1}{4}\). Then, 
\begin{align*}
    \mathbb{P}\left(|C_{2,2}|\leq \frac{n}{2} - 4*10^{-12}pn\Delta\right) 
    &= 1- \mathcal{O}\left(\frac{1}{p\Delta^2}\right) - \mathcal{O}\left(\frac{1}{np^3\Delta^2}\right) .
\end{align*}
\end{lemma}
\begin{proof}
By Lemmas \ref{main.lemma.expectation} and \ref{mainRESULTstep4VariaNCEdayTWO}, if $n\geq \exp \left[ 6*10^{13}\right]$,
\[
\mathbb{E}\left[ |C_{1,2}|\right] \geq \frac{n}{2} + 5*10^{-12}*pn\Delta
\]
and
\begin{align*}
      \mathbf{Var}\left(|{{{{{{{{{{C}}}}}}}}}}_{1,2}|\right)  =\mathcal{O}\left(n^2p\right) + \mathcal{O}\left(\frac{n}{p}\right).
  \end{align*} 
Therefore, by Chebyshev's inequality,
\begin{align*}
    \mathbb{P}\left(|C_{1,2}|\leq \frac{n}{2} + 4*10^{-12}pn\Delta\right) & = \frac{\mathcal{O}\left(n^2p\right) + \mathcal{O}\left(\frac{n}{p}\right)}{p^2n^2\Delta^2} \\
    &= \mathcal{O}\left(\frac{1}{p\Delta^2}\right) + \mathcal{O}\left(\frac{1}{np^3\Delta^2}\right) .
\end{align*}
Since \(|C_{2,2}| = n - |C_{1,2}|\), we must have:
 \begin{align*}
    \mathbb{P}\left(|C_{2,2}|\leq \frac{n}{2} - 4*10^{-12}pn\Delta\right) 
    &= 1- \mathcal{O}\left(\frac{1}{p\Delta^2}\right) - \mathcal{O}\left(\frac{1}{np^3\Delta^2}\right) .
\end{align*}
\end{proof}

\section{Step 6 - Day 3}
We now conclude the proof of Theorem \ref{the_power_of_few_2_colours} by following the proof in \cite{tran2023powerfewphenomenonsparse}. 
\begin{lemma}\label{currently10.1}
    For any $a>0$ and $b\in (0,1/2)$ satisfying:
    \[
    ba^2 >3(\log 2)/2,
    \]
    if $|C_{2,2}|\leq \frac{n}{2} - \frac{an}{\sqrt{pn}}$, then $|C_{2,3}|\leq bn$ with probability at least $1-e^{-\Omega (n)}$.
\end{lemma}
The proof to this Lemma can be found in \cite[Lemma 2.4]{tran2023powerfewphenomenonsparse}.
\section{Step 7 - Day 4 and beyond}
\begin{lemma}\label{currently11.1}
    Consider the majority dynamics process on $G\sim G(n,p)$, where $ p\geq \frac{(1+\lambda)\log n}{n}.$ Assume that $|C_{2,2}|\leq bn$ for some constant $b<1/2$. Then, with probability at least $1-2n^{-\lambda/2}$, there is a time $t^* = \mathcal{O}(\log_{pn} n) $ such that $|C_{2,t^*}| = 0$
\end{lemma}
The proof to this Lemma can be found in \cite[Lemma 2.6]{tran2023powerfewphenomenonsparse}.
\section{Step 8 - Wrapping up the proof of Theorem \ref{the_power_of_few_2_colours}}
Consider the majority dynamics process on $G\sim G(n,p)$, where $\frac{1}{4}\geq p \geq (1+\lambda)\frac{\log n}{n}$ and $|C_{1,0}| = \frac{n}{2}+\Delta$. Furthermore, let $\frac{10}{p}\geq \Delta \geq \max \left\{\dfrac{1}{\sqrt{p}} \exp\left[A\sqrt{\log \left(\dfrac{1}{p}\right)}\right] , Bp^{-3/2} n^{-1/2} \right\}$ (where $A$ and $B$ are sufficiently large constants to apply the following arguments). By Lemma \ref{currentlylemma9.1},
\begin{align}\label{somerandomequatinofff}
   |C_{2,2}|\leq \frac{n}{2} - 4*10^{-12}pn\Delta,
\end{align}
except with a fail probability of at most $\mathcal{O}\left(\frac{1}{p\Delta^2}\right) + \mathcal{O}\left(\frac{1}{np^3\Delta^2}\right)$. Conditioning on Equation \ref{somerandomequatinofff}, we apply Lemma \ref{currently10.1} with $a = 4*10^{-12}pn\Delta$ and $b = 1/3$. Then,

\begin{align}\label{fuwhwhgnvocmofowrowofcn+++o....hofehdfeffeofhwohf}
    |C_{2,2}|\leq \frac{n}{3},
\end{align}
except with a fail probability of at most $e^{-\Omega(n)}$. Conditioning on Equation \ref{fuwhwhgnvocmofowrowofcn+++o....hofehdfeffeofhwohf}, we apply Lemma \ref{currently11.1} with $b = 1/3$ to finish the proof. The total fail probability is at most \[
\mathcal{O}\left(\frac{1}{p\Delta^2}\right) + \mathcal{O}\left(\frac{1}{np^3\Delta^2}\right) + e^{-\Omega(n)} +2n^{-\lambda/2} = \mathcal{O}\left( \frac{1}{p\Delta^2} + \frac{1}{np^3\Delta^2} +n^{-\lambda/2}\right).
\]
\section{Concluding Remarks}
We have proved that whenever $\Delta = \max \left\{\frac{1}{\sqrt{p}} \exp\left[\Omega\left(\sqrt{\log \left(\frac{1}{p}\right)}\right)\right] , \omega\left(p^{-3/2} n^{-1/2} \right)\right\}$, unanimity occurs with a probability of $1 - o(1)$. We believe that this result holds for smaller values of $\Delta$. Linh Tran and Van Vu conjectured that this result holds for $\Delta = \omega \left(p^{-1/2}\right)$ \cite{tran2023powerfewphenomenonsparse}. If their conjecture is true, we believe that a strong analysis of the first $\log_{pn}(n)$ days would be required to prove it. 

\subsection*{Note in Proof: } While finalizing the writeup of this paper, we were made aware of a new paper by Kim and Tran \cite{kim2025new}, which proves the following result:
\begin{theorem}
     Let $V$ be a set of $n$ vertices with arbitrary partition $V = C_{1} \cup C_{2}$, and suppose $|C_{1}| = \frac{n}{2} + \Delta $. Let $G$ be a $G(n,p)$ random graph drawn over $V$. 
     
     Suppose $n^{-1/2}\log^{1/4}n\geq p\geq Cn^{-1}\log^2 n$.  There is a universal constant $C$ such that for every $\epsilon> 0 $ fixed, there are constants $C_{\epsilon},N_\epsilon$ such that if $n\geq N_\epsilon $ and $\Delta  \geq C_\epsilon p^{-3/2}n^{-1/2}\log n$, then Colour 1 wins within $\mathcal{O}(\log_{pn}n)$ days with probability at least $1-\epsilon$.
\end{theorem}

\section{Acknowledgement}
The research was conducted while the author was a student taking part in the University of Cambridge Summer Research in Mathematics (SRIM) programme, and was supervised by Marcelo Campos. The author would like to thank Marcelo Campos for many helpful discussions and for suggesting the problem. 

The author was supported by the Trinity College Summer Studentship Scheme
fund.
\appendix
\section{Some Preliminary Results in Probability}
\label{appendix_A}
\setcounter{equation}{0}
\renewcommand{\theequation}{A.\arabic{equation}}

\begin{theorem}\label{in_paper_A.1}\normalfont{(Berry-Esseen).}\textit{ There is a universal constant, }$C_{BE}$\textit{, such that for any $n \in \mathbb{N}$, if \(X_1,X_2,\dots, X_n\) are random variables with zero means, variances \(\sigma_1^2,\sigma_2^2,\dots, \sigma_n^2 > 0\) and absolute third moments \(\mathbb{E}\left[|X_i|^3 \right] = \rho_i < \infty\), we have:} 
\begin{equation*}
    \sup_{z \in \mathbb{R}} \bigg| \mathbb{P}\left( \frac{\sum_{i=1}^n X_i}{\sqrt{\sum_{i=1}^n\sigma_i^2 }}\leq z \right)  - \Phi(z)\bigg| \leq C_{BE} \sum_{i=1}^n \rho_i \left( \sum_{i=1}^n\sigma_i^2  \right)^{-3/2}.
\end{equation*}
The original proof was by Esseen \cite{doi:10.1080/03461238.1956.10414946} and gave \(C_{BE} = 7.59\). Since then, this constant has been improved many times. In this paper, we will be using \(C_{BE} = 0.56\), which was proved in the latest work by Shevtosova \cite{article_Shevtosova}.

\end{theorem}
\begin{corollary}\label{in_paper_A.2}
Let \(p\in (0,1)\) and \(X_1,X_2,\dots,X_n\) be independent random variables such that for all \(i\), \(X_i \sim \mathrm{Ber}(p)\) or \(X_i \sim - \mathrm{Ber}(p)\). Let \(X = X_1+X_2+\cdots+X_n\), and \(\mu_X = \mathbb{E}[X]\). Then, 
\begin{equation*}
    \sup_{x \in \mathbb{R}} \bigg| \mathbb{P}(X -\mu_X \leq x) - \Phi\left(\frac{x}{\sigma \sqrt{n}}\right)\bigg| \leq \frac{C_{BE}(1-2\sigma^2)}{\sigma \sqrt{n}} \text{ , where } \sigma = \sqrt{p(1-p)}
\end{equation*}
\end{corollary}
    This implies the following lemma, which bounds the probability that two binomial random variables are a specific number of units apart.

\begin{lemma}\label{in_paper_A.3}
Let \(p \in (0,1)\) be a constant and \(\sigma = \sqrt{p(1-p)}\). Suppose \(X_1 \sim \mathrm{Bin}(n_1,p)\) and \(X_2 \sim \mathrm{Bin}(n_2,p)\) are independent random variables. Then, for any positive integer \(d \),
\[
\mathbb{P}(X_1 = X_2 + d ) \leq \frac{1.12(1-2\sigma^2)}{\sigma \sqrt{n_1+n_2}}.
\]
\end{lemma}

\begin{corollary}\label{single_use_corollarhy_for+__far_variance+calc_calc_cakldfdsfdfdsf}
Let \(p \in (0,1)\) be a constant and \(\sigma = \sqrt{p(1-p)}\). Suppose \(X_1 \sim \mathrm{Bin}(n_1,p)\), \(X_2 \sim \mathrm{Bin}(n_2,p)\) and \(X_2 \sim \mathrm{Bin}(n_2,p)\) are independent random variables. Then, for any positive integer \(d \),
\[
\mathbb{P}(X_1 - X_2 +X_3 = d ) \leq \frac{1.12(1-2\sigma^2)}{\sigma \sqrt{n_1+n_2+n_3}}.
\] 
\end{corollary}
Since \(X_1+X_3 \sim \mathrm{Bin}(n_1+n_3,p)\), Corollary \ref{single_use_corollarhy_for+__far_variance+calc_calc_cakldfdsfdfdsf} follows immediately from Lemma \ref{in_paper_A.3}
\begin{lemma} \label{taken_from_other_paper} 
Let \(X_1\sim \mathrm{Bin}(n,p)\) and \(X_2\sim \mathrm{Bin}(m,p)\) be independent random variables, where \(p \in (\frac{\log n}{n},1)\) and \(n\geq m \in \mathbb{N}\). Define \(A:= X_1-X_2\). Then, for $n \geq 520$, 
For every \(d \in \mathbb{Z}\), \[
|\mathbb{P}(A = d+1)-\mathbb{P}(A =d)| \leq \frac{20C_{BE}}{np(1-p)}.
\]

\end{lemma}
\begin{proof}
We will first prove that for every $ d \in \mathbb{Z}$, \[
|\mathbb{P}(X_1 = d+1)-\mathbb{P}(X_1 =d)| \leq \frac{20C_{BE}}{np(1-p)}.
\]
    By Corollary \ref{in_paper_A.2}, for any $y\in \mathbb{R}$, \[\mathbb{P}(X_1 \leq y) \geq \Phi\left(\frac{y-\mathbb{E}[X_1]}{\sigma\sqrt{n}}\right) - \frac{C_{BE}}{\sigma\sqrt{n}}, \]
where \(\sigma = \sqrt{p(1-p)}\). Therefore, \[
\mathbb{P}(X_1 < d) \geq \limsup_{y < d}\bigg( \Phi\left(\frac{y-\mathbb{E}[X_1]}{\sigma\sqrt{n}}\right) - \frac{C_{BE}}{\sigma\sqrt{n}}\bigg) = \Phi\left(\frac{d-\mathbb{E}[X_1]}{\sigma\sqrt{n}}\right) - \frac{C_{BE}}{\sigma\sqrt{n}}.
\]
By Corollary \ref{in_paper_A.2}, \[\mathbb{P}(X_1 \leq d) \leq \Phi\left(\frac{d -\mathbb{E}[X_1]}{\sigma\sqrt{n}}\right) + \frac{C_{BE}}{\sigma\sqrt{n}}.
\]
Therefore, \begin{equation}\label{ghghghghghgh}
\mathbb{P}(X_1 = d ) \leq \frac{2C_{BE}}{\sigma\sqrt{n}} = \frac{2C_{BE}}{\sqrt{np(1-p)}}.\end{equation}

Consider \(F(d):= \mathbb{P}(X_1 = d+1)-\mathbb{P}(X_1 =d)\). After some algebraic manipulation, we find that \(F(d)-F(d-1)\) is positive iff \(|d+1/2-p(n+1)|<\sqrt{1+4p(1-p)(n+1)}/2\). Therefore, it suffices to verify the claim for values of \(d\) such that \(|d-np|\leq 2\sqrt{np(1-p)}\). For these values, we have:
\begin{align*}
F(d) = \mathbb{P}(X_1=d)\left[ -1+\frac{\mathbb{P}(X_1=d+1)}{\mathbb{P}(X_1=d)}\right] &=  \mathbb{P}(X_1=d)\left[ -1+\frac{\binom{n}{d+1}p^{d+1}(1-p)^{n-d-1}}{\binom{n}{d}p^d(1-p)^{n-d}}\right]\\
&= \mathbb{P}(X_1=d)\left[ -1+\frac{(n-d)p}{(d+1)(1-p)}\right].
\end{align*}
By Equation \ref{ghghghghghgh}, 
\[
|\mathbb{P}(X_1=d)|\leq \frac{2C_{BE}}{\sqrt{np(1-p)}}.
\]
Furthermore, 
\begin{align*}
-1+\frac{(n-d)p}{(d+1)(1-p) }\leq -1 + \frac{(n-d)p}{d(1-p)} = \frac{np-d}{d(1-p)}.
\end{align*}
This is a decreasing function of $d$, so is less than or equal to  

\[
\frac{2\sqrt{np(1-p)}}{(1-p)(np-\sqrt{np(1-p)}} = \frac{2}{\sqrt{1-p}(\sqrt{np}-2)} < \frac{10}{\sqrt{np(1-p)}},
\]
for $n \geq 520$ (noting that $np \geq \log n$). Therefore,
\[
F(d) < \frac{20C_{BE}}{np(1-p)}.
\]

We now return to considering $A$. 
\begin{align*}
\mathbb{P}(A = d+1)-\mathbb{P}(A=d)&= \sum_{i}\left(\mathbb{P}(X_1 = d+1+i)-\mathbb{P}(X_2=d+i)\right)\mathbb{P}(X_2 =i)\\
&\leq \sum_{i} \frac{20C_{BE}}{np(1-p)}\mathbb{P}(X_2=i)\\
&= \frac{20C_{BE}}{np(1-p)}.
\end{align*}
This completes the proof.

\end{proof}

\begin{lemma}\label{two_bin_equal} Suppose $X_1,X_2 \sim \mathrm{Bin}(n,p)$ are independent random variables, for $n \geq 20$ and $p \in [\frac{\log n}{n},1-\frac{10}{n}]$. Then,
\[
\mathbb{P}(X_1=X_2)\geq \frac{1}{6\sqrt{np(1-p)}}.
\]
  
\end{lemma}
\begin{proof}
Since \(X_1 - X_2\) has mean \(0\) and standard deviation \(\sigma = \sqrt{2np(1-p)}\), by Chebyschev's inequality,
\[
\mathbb{P}(|X_1-X_2|\leq \sqrt{3}\sigma)\geq \frac{2}{3}.
\]
Therefore,
\[
\sum_{i = -\lfloor \sqrt{3}\sigma\rfloor}^{\lfloor \sqrt{3}\sigma\rfloor} \mathbb{P}(X_1-X_2 = i) \geq \frac{2}{3}.
\]
Since the distribution of \(X_1-X_2\) is unimodal with the mode being 0, we find that \[
\mathbb{P}(X_1-X_2 = 0)\geq \frac{2/3}{2\lfloor \sqrt{3}\sigma\rfloor +1}\geq\frac{2/3}{2\sqrt{3}\sigma+1}.
\]
It is easy to see that \(np(1-p)\geq 10(1-\frac{10}{n})\geq 5\). Using this, we find that \[
\frac{2/3}{2\lfloor \sqrt{3}\sigma\rfloor +1}\geq\frac{2/3}{2\sqrt{3}\sigma+1} \geq \frac{1}{6\sigma}.
\]
Therefore, \[
\mathbb{P}(X_1=X_2)\geq \frac{1}{6\sqrt{np(1-p)}}.
\]
\end{proof}

The following corollary follows easily from Lemma \ref{two_bin_equal} and Lemma \ref{taken_from_other_paper}.
\begin{corollary}\label{cttorollaryththththth}
    Let \(X_1\) and \(X_2\) be independent \(\mathrm{Bin}(n,p)\) distributions, where\hfill\break \({p \in [\frac{\log n}{n},1- \frac{10}{n})}\) and \(n \geq 520\). For every \(d \in \mathbb{Z}\),
    \[
    \mathbb{P}(X_1-X_2= d) \geq \frac{1}{6\sqrt{np(1-p)}} - \frac{20C_{BE}d}{np(1-p)} .
    \]
\end{corollary}
\begin{proposition}\label{variance_bound_conditioned_on_an_event_AP_A}
Let $\mathcal{E}$ be an event with $\mathbb{P}(\mathcal{E})>0$, and $X$ a random variable. Then,
\begin{equation*}
     \mathbf{Var}\left(X\big| \mathcal{E}\right)  \leq \frac{\mathbf{Var}\left(X\right)  }{\mathbb{P}(\mathcal{E})}.
\end{equation*}
\end{proposition}
\begin{proof}
    By the law of total variance,
    \begin{align*}
        \mathbf{Var}\left(X\right) &=\mathbb{E}\left[ \mathbf{Var}\left(X\big| \mathbbm{1}_{\mathcal{E}}\right)\right] + \mathbf{Var}\left(\mathbb{E}\left[X\big| \mathbbm{1}_{\mathcal{E}}\right]\right)\\
        &\geq \mathbf{Var}\left(X\big| \mathcal{E}\right)\mathbb{P}(\mathcal{E}) +\mathbf{Var}\left(X\big| \neg\mathcal{E}\right)\mathbb{P}(\neg\mathcal{E})  \\
        &\geq \mathbf{Var}\left(X\big| \mathcal{E}\right)\mathbb{P}(\mathcal{E}).
    \end{align*}
    Therefore,
    \begin{align*}
        \mathbf{Var}\left(X\big| \mathcal{E}\right)  \leq \frac{\mathbf{Var}\left(X\right)  }{\mathbb{P}(\mathcal{E})}.
    \end{align*}
\end{proof}
\begin{proposition}\label{expectation_bound_conditioned_on_an_event_AP_A}
    Let $\mathcal{E}$ be an event with $\mathbb{P}(\mathcal{E})>0$, and $X$ a non-zero random variable. Then, \[
    \mathbb{E}[X| \mathcal{E}] \leq \frac{\mathbb{E}[X]}{\mathbb{P}(\mathcal{E})}.
    \]
    
\end{proposition}
\begin{proof}
By the law of total expectation,
\begin{align*}
    \mathbb{E}\left[X\right] &= \mathbb{E}[X| \mathcal{E}] \mathbb{P}\left(\mathcal{E}\right) + \mathbb{E}[X| \neg\mathcal{E}] \mathbb{P}\left(\neg\mathcal{E}\right)\\
    &\geq \mathbb{E}[X| \mathcal{E}] \mathbb{P}\left(\mathcal{E}\right).
\end{align*}
Dividing both sides by \(\mathbb{P}(\mathcal{E})\) completes the proof

\end{proof}
\begin{lemma}\label{variance_of_phi}
\begin{enumerate}[label=(\roman*)]
    \item \(|\Phi(a) - \Phi(b)| \leq |a-b|\) for any $a,b\in \mathbb{R}$.
    \item Let $W$ be a real-valued random variable with finite variance. Then, \[
    \mathbf{Var}(\Phi(W))\leq \mathbf{Var}(W).
    \]
\end{enumerate}
\end{lemma}
\begin{proof}[Proof of Part (i). ]Let \(a,b \in \mathbb{R}\). Wlog $a>b$. Then, by the intermediate value theorem, 
\[
\frac{\Phi(a) - \Phi(b)}{a-b} = \Phi'(\theta)
\]
for some \(\theta \in (b,a)\). We know that
\[
\Phi'(\theta) = \frac{1}{\sqrt{2\pi}}e^{-\frac{\theta^2}{2}} \leq 1.
\]
Therefore,\[
|\Phi(a) - \Phi(b)| \leq |a-b|
\]
for any $a,b\in \mathbb{R}$.
    
\end{proof}
\begin{proof}[Proof of Part (ii). ]
Let \(W_1\) and \(W_2\) be independent random variables with the same distribution as $W$. By part (i), \((\Phi(W_1) - \Phi(W_2))^2 \leq (W_1-W_2)^2\). Therefore,\[
\mathbf{Var}(W) = \frac{1}{2}\mathbb{E}\left[\left(W_1-W_2\right)^2\right]\geq \frac{1}{2}\mathbb{E}\left[\left(\Phi(W_1)-\Phi(W_2)\right)^2\right] =\mathbf{Var}(\Phi(W)).
\]

\end{proof}
\section{Proof of Key Lemmas in Step 2 of Theorem \ref{the_power_of_few_2_colours}}
\label{appendix_B_2_colours}
\setcounter{equation}{0}
\renewcommand{\theequation}{B.\arabic{equation}}
\begin{proof}[Proof of Lemma \ref{CLTPaperLemma10}(i). ]
This follows trivially from the definition of $Z_v$.
\end{proof}
\begin{proof}[{Proof of Lemma \ref{CLTPaperLemma10}(ii).} ]
    Part (ii) of the lemma follows straightforwardly from the facts that $Z_v + \mu_v \in \{-1,1\}$, that $Z_v$ has mean $0$, and that $Z_v$ only depends on $\Gamma_v$. 
\end{proof}
\begin{proof}[{Proof of Lemma \ref{CLTPaperLemma10}(iii) and (iv).} ]
Suppose $\emptyset \neq S \subset \mathcal{E}$. Write $S$ as $S = S_1 \cup S_2,$ where $S_1 = S\backslash (C_{1,0}\times C_{2,0})$ and $S_2 = S \cap (C_{1,0}\times C_{2,0})$. In this way, $S$ consists of \(|S_1| \) edges from \(v\) to a vertex of the same colour and \(|S_2|\) edges from $v$ to a vertex of a different colour. Now, define the random variable \begin{align*}
    W_{v;S} := \begin{cases}
  \Bin\big(|C_{1,0}| - 1 - |S_1|,p\big) -\Bin\big( |C_{2,0}| - |S_2|\big)& \text{if } v\in C_{1,0},\\ 
  \Bin\big(|C_{2,0}| - 1 - |S_1|,p\big) -\Bin\big( |C_{1,0}| - |S_2|\big)& \text{if } v \in C_{2,0}.
    \end{cases}
\end{align*}
Then, by conditioning on the edges of $S$, we have (for $S\neq \emptyset$),
\begin{eqnarray}\label{fjowgoiwrnognwngowrjewfjweiofjioewfjiewjofijewiofjoiewjoingononoiwgoijoijiowjwoiowneovniovnwrvwoivwiogowlaaljs}
\widehat{Z_v}(S) &=& \mathbb{E}\left[Z_v (\vec{x}) \Phi_{S} (\vec{x}) \right]\notag\\
&=& \sum_{I \subseteq S_1} \sum_{J \subseteq S_2} p^{|I|+|J|} (1-p)^{|S|-|I|-|J|} \bigg[2 \mathbb{P}\Big(W_{v;S} \geq |J|-|I|-L(v)\Big) -1 \bigg]\notag\\
& & \qquad \times \left( \dfrac{2-2p}{2\sqrt{p(1-p)}} \right)^{|I|+|J|} \left(\dfrac{-2p}{2\sqrt{p(1-p)}} \right)^{|S|-|I|-|J|}\notag\\
&=& 2 \left(-\sqrt{p(1-p)} \right)^{|S|} \sum_{I \subseteq S_1}\sum_{J \subseteq S_2} (-1)^{|I|+|J|} \mathbb{P}\Big(W_{v;S} \geq |J|-|I|-L(v)\Big).
\end{eqnarray}
Since $S \neq \emptyset$, let $e \in S$.  If $e \in S_1$, then by conditioning on whether $e \in I \cup J$, we have
\begin{align*}
 & \sum_{I \subseteq S_1}\sum_{J \subseteq S_2} (-1)^{|I|+|J|} \mathbb{P}\Big(W_{v;S} \geq |J|-|I|- L(v)\Big)\\
 & = \sum_{I \subseteq S_1\backslash\{e\}}\sum_{J \subseteq S_2} (-1)^{|I|+|J|} \Bigg[ \mathbb{P}\Big(W_{v;S} \geq |J|-|I|-L(v)\Big) - \mathbb{P}\Big(W_{v;S} \geq |J|-|I| - L(v) - 1 \Big)\Bigg]\\
 &  = -\sum_{I \subseteq S_1\backslash\{e\}}\sum_{J \subseteq S_2} (-1)^{|I|+|J|} \Bigg[ \mathbb{P}\Big(W_{v;S} = |J|-|I|-L(v)-1\Big) \Bigg].
\end{align*}
Similarly, if \(e \in S_2\), then:
\begin{align*}
& \sum_{I \subseteq S_1}\sum_{J \subseteq S_2} (-1)^{|I|+|J|} \mathbb{P}\Big(W_{v;S} \geq |J|-|I|- L(v)\Big)\\
 &  = \sum_{I \subseteq S_1}\sum_{J \subseteq S_2\backslash\{e\}} (-1)^{|I|+|J|} \Bigg[ \mathbb{P}\Big(W_{v;S} \geq |J|-|I|-L(v)\Big) - \mathbb{P}\Big(W_{v;S} \geq |J|-|I| - L(v) + 1 \Big)\Bigg]\\
 &  = \sum_{I \subseteq S_1\backslash\{e\}}\sum_{J \subseteq S_2} (-1)^{|I|+|J|} \Bigg[ \mathbb{P}\Big(W_{v;S} = |J|-|I|-L(v)\Big) \Bigg].
\end{align*}
By Lemma \ref{in_paper_A.3}, \(\mathbb{P}\Big(W_{v;S} = |J|-|I|-L(v)\Big)\) and \( \mathbb{P}\Big(W_{v;S} = |J|-|I|-L(v)-1\Big)\) are both $\mathcal{O}\left(\frac{1}{\sqrt{(n-|S|-1)p(1-p)}}\right)$. Since \(|S|\leq 10k^2\leq n/2\), both these probabilities are $\mathcal{O}\left(\frac{1}{\sqrt{np(1-p)}}\right)$. Therefore, in both cases,
\begin{align*}
\bigg|\sum_{I \subseteq S_1}\sum_{J \subseteq S_2} (-1)^{|I|+|J|} \mathbb{P}\Big(W_{v;S} \geq |J|-|I|- L(v)\Big)\bigg| & = \mathcal{O}\left(\frac{2^{|S|}}{\sqrt{np(1-p)}}\right).
\end{align*}
Therefore, by Equation \ref{fjowgoiwrnognwngowrjewfjweiofjioewfjiewjofijewiofjoiewjoingononoiwgoijoijiowjwoiowneovniovnwrvwoivwiogowlaaljs},
\begin{align*}
\widehat{Z_v}(S) & =  \mathcal{O}\left(\frac{\left(2\sqrt{p(1-p)} \right)^{|S|}}{\sqrt{np(1-p)}}\right) =  \mathcal{O}\left(\frac{\left(2\sqrt{p(1-p)} \right)^{|S|-1}}{\sqrt{n}}\right) =  \mathcal{O}\left(\frac{1}{\sqrt{n}}\right),
\end{align*}

proving part (iii) of the lemma. We prove part (iv) by picking some other edge $e'\in S$ and conditioning on whether or not $e' \in I\cup J$ to obtain:
\begin{align*}
    |\widehat{Z_v}| &\leq 2\left(\sqrt{p(1-p)}\right)^{|S|} \sum_{I \subseteq S_1}\sum_{J \subseteq S_2} \sup_L \bigg| \mathbb{P}\left(W_{v;S} = L+1 \right) - \mathbb{P}\left(W_{v;S} = L \right)\bigg| \\
    & = 2\left(2\sqrt{p(1-p)}\right)^{|S|} \sup_L \bigg| \mathbb{P}\left(W_{v;S} = L+1 \right) - \mathbb{P}\left(W_{v;S} = L \right)\bigg|.
\end{align*}
Therefore, by Lemma \ref{taken_from_other_paper}, the supremum is $\mathcal{O}\left(\frac{1}{np(1-p)}\right)$. Therefore, 
\begin{align*}
    |\widehat{Z_v}| & = \mathcal{O}\left( \frac{\left(2\sqrt{p(1-p)}\right)^{|S|-2}}{n}\right) = \mathcal{O} \left(\frac{1}{n}\right).
    \end{align*}
\end{proof}
\begin{proof}[{Proof of Lemma \ref{CLTPaperLemma10}(v).} ]
For this, we first note that by part (ii), this sum reduces to only two terms:
\begin{eqnarray*}
\sum_{e \in E} \Big(\widehat{Z_{u_1}} (e) - \widehat{Z_{u_2}}(e) \Big) \widehat{Z_v} (e) &=& \widehat{Z_{u_1}} (u_1 v) \widehat{Z_v} (u_1 v) - \widehat{Z_{u_2}}(u_2 v) \widehat{Z_v} (u_2 v)\\
&=& \widehat{Z_v} (u_1 v) \Big( \widehat{Z_{u_1}} (u_1 v) + \widehat{Z_{u_2}}(u_2 v) \Big) - \widehat{Z_{u_2}}(u_2 v) \Big( \widehat{Z_v} (u_2 v) + \widehat{Z_v} (u_1 v) \Big).
\end{eqnarray*}
Suppose $u_1' \in C_{1,0}$ and $u_2' \in C_{2,0}$ are additional vertices.  From our computation in (iii), we have
\begin{eqnarray*}
\widehat{Z_{u_1}}(u_1 u_1') &=& 2 \sqrt{p(1-p)} \cdot \mathbb{P}\Big(Bin(|C_{1,0}|-2,p)-Bin(|C_{2,0}|,p) = -1\Big),\\
\widehat{Z_{u_2}}(u_2 u_2') &=& 2 \sqrt{p(1-p)} \cdot \mathbb{P}\Big(Bin(|C_{2,0}|-2,p)-Bin(|C_{1,0}|,p) = -1\Big),\\
\widehat{Z_{u_1}}(u_1 u_2) = \widehat{Z_{u_2}}(u_1 u_2) &=& -2 \sqrt{p(1-p)} \cdot \mathbb{P}\Big(Bin(|C_{1,0}|-1,p)-Bin(|C_{2,0}|-1,p) = 0\Big).
\end{eqnarray*}

By Lemma \ref{taken_from_other_paper}, the above three probabilities are all within $\mathcal{O}\left(\dfrac{1}{np(1-p)}\right)$ of each other. This means that $\widehat{Z_{u_1}}(u_1v) + \widehat{Z_{u_2}}(u_2v)$ and $\widehat{Z_v}(u_2v) + \widehat{Z_v}(u_1v)$ are both at most $\mathcal{O}\left( \dfrac{1}{n\sqrt{p(1-p)}}\right)$.  Combining this with part (iii) then establishes the desired result.
\end{proof}
\begin{proof}[{Proof of Lemma \ref{CLTPaperLemma10}(vi).} ]
We will use the fact that $Z_v + \mu_v \in \{-1,1\}$. We obtain:
\begin{eqnarray*}
Z_v ^L &=& (Z_v + \mu_v - \mu_v) ^L \\
&=& \sum_{i=0} ^{L} {\binom{L}{i}} (Z_v +\mu_v) ^{i} (-\mu_v)^{L-i}\\
&=& Z_v \dfrac{(1-\mu_v)^L - (-1)^L (1+\mu_v)^L}{2} + (1-\mu_v ^2) \dfrac{(1-\mu_v)^{L-1} + (-1)^L (1+\mu_v)^{L-1}}{2}. 
\end{eqnarray*}

\end{proof}
\section{Proof of Key Lemmas in Step 3 of Theorem \ref{the_power_of_few_2_colours}}
\label{appendix_C_2_colours}
\setcounter{equation}{0}
\renewcommand{\theequation}{C.\arabic{equation}}

\begin{proof}[Proof of Proposition \ref{beta__proposition} (i). ]
It is easy to see that 
\begin{align}\label{badcalc1}
    \mathbb{E}[|\hat{R}_{v_1}|] &= \sum_{u \in V\backslash \{v_1\}} \mathbb{E}[\mathbbm{1}_{u \in \hat{R}_{v_1}}]\notag\\
    &= \sum_{u \in V\backslash \{v_1\}} \mathbb{P}(u \in \hat{R}_{v_1})\notag\\
    &= \sum_{u \in V\backslash \{v_1\}} \mathbb{P}(u \in C_{1,1}| u \sim v_1).
\end{align}
Similarly,
\begin{align}\label{badcalc2}
    \mathbb{E}[|C_{1,1}|] &= \sum_{u \in V}\mathbb{E}[\mathbbm{1}_{u \in C_{1,1}}]\notag\\
    &= \sum_{u \in V} \mathbb{P}(u \in C_{1,1})\notag\\
    &=\mathbb{P}(v_1 \in C_{1,1}) + \sum_{u \in V\backslash \{v_1\}} \mathbb{P}(u \in C_{1,1}|u \sim v_1)\mathbb{P}(u \sim v_1)\notag\\
    &+ \sum_{u \in V\backslash \{v_1\}}\mathbb{P}(u \in C_{1,1}| u \not \sim v_1) \mathbb{P}(u \not \sim v_1)\notag\\
    &=\mathbb{P}(v_1 \in C_{1,1}) + \sum_{u \in V\backslash \{v_1\}} p \mathbb{P}(u \in C_{1,1}|u \sim v_1)\notag\\
    &+ \sum_{u \in V\backslash \{v_1\}}(1-p)\mathbb{P}(u \in C_{1,1}| u \not \sim v_1) .
\end{align}
Therefore, using Equations \ref{badcalc1} and \ref{badcalc2}, we find that:
\begin{align}
    \beta_1 &= -\mathbb{P}(v_1 \in C_{1,1}) + (1-p)\sum_{u \in V\backslash \{v_1\}}\left( \mathbb{P}(u \in C_{1,1}|u \sim v_1) - \mathbb{P}(u \in C_{1,1}|u \not \sim v_1)\right)\notag\\
    &= -\mathbb{P}(v_1 \in C_{1,1}) + (1-p)\bigg( \mathbb{P}(v_2 \in C_{1,1}| v_1 \sim v_2) - \mathbb{P}(v_2 \in C_{1,1} | v_2 \not \sim v_1) \bigg) \notag\\
    &+(1-p)\sum_{u \in V \backslash \{v_1,v_2\}}\bigg( p \mathbb{P}(u \in C_{1,1}| u \sim v_1, u \sim v_2)\notag\\ 
    &+ (1-p)\mathbb{P}(u \in C_{1,1}| u \sim v_1, u \not \sim v_2) - p \mathbb{P}(u \in C_{1,1}| u \not \sim v_1, u \sim v_2)\notag\\
    &- (1-p) \mathbb{P}(u \in C_{1,1}| u \not \sim v_1, u \not \sim v_2 ) \bigg). \notag  
\end{align}
Since \(\mathbb{P}(v_1 \in C_{1,1})\leq 1\) and \(\mathbb{P}(v_2 \in C_{1,1}| v_1 \sim v_2) - \mathbb{P}(v_2 \in C_{1,1} | v_2 \not \sim v_1)\geq 0\) (because $v_1 \in C_{1,0}$), we have:
\begin{align}\label{badcalc_important_1_sec}
    \beta_1 & \geq -1 +(1-p)\sum_{u \in V \backslash \{v_1,v_2\}}\bigg( p \mathbb{P}(u \in C_{1,1}| u \sim v_1, u \sim v_2)\notag\\ 
    &+ (1-p)\mathbb{P}(u \in C_{1,1}| u \sim v_1, u \not \sim v_2) - p \mathbb{P}(u \in C_{1,1}| u \not \sim v_1, u \sim v_2)\notag\\
    &- (1-p) \mathbb{P}(u \in C_{1,1}| u \not \sim v_1, u \not \sim v_2)\bigg)  
\end{align}
%
%
%
%
%
%
%
%
%
It is easy to see that 
\begin{align}\label{badcalc3}
    \mathbb{E}[|\hat{R}_{v_2}|] &= \sum_{u \in V\backslash \{v_2\}} \mathbb{E}[\mathbbm{1}_{u \in \hat{R}_{v_2}}]\notag\\
    &= \sum_{u \in V\backslash \{v_2\}} \mathbb{P}(u \in \hat{R}_{v_2})\notag\\
    &= \sum_{u \in V\backslash \{v_2\}} \mathbb{P}(u \in C_{1,1}| u \sim v_2)
\end{align}

By the same logic as Equation \ref{badcalc2},
\begin{align}\label{badcalc4}
    \mathbb{E}[|C_{1,1}|] &=\mathbb{P}(v_2 \in C_{1,1}) + \sum_{u \in V \backslash \{v_2\}}  p \mathbb{P}(u \in C_{1,1}|u \sim v_2)\notag\\
    &+ \sum_{u \in V\backslash \{v_1\}}(1-p)\mathbb{P}(u \in C_{1,1}| u \not \sim v_2)
\end{align}
Therefore, by Equations \ref{badcalc3} and \ref{badcalc4},
\begin{align}\label{badbadbasbas}
\beta_2 &= \mathbb{P}(v_2 \in C_{1,1}) + (1-p)\sum_{u \in V \backslash \{v_2\}} \bigg(  \mathbb{P}(u \in C_{1,1}|u \not \sim v_2) - \mathbb{P}(u \in C_{1,1}|u \sim v_2)\bigg)\\
&= \mathbb{P}(v_2 \in C_{1,1}) + (1-p)\bigg( \mathbb{P}(v_1 \in C_{1,1}| v_1 \not \sim v_2) - \mathbb{P}(v_1 \in C_{1,1} | v_1 \sim v_2)\bigg)\notag\\
    &+(1-p)\sum_{u \in V \backslash \{v_1,v_2\}}\bigg( -p \mathbb{P}(u \in C_{1,1}| u \sim v_1, u \sim v_2)\notag\\ 
    &+ p\mathbb{P}(u \in C_{1,1}| u \sim v_1, u \not \sim v_2) - (1-p) \mathbb{P}(u \in C_{1,1}| u \not \sim v_1, u \sim v_2)\notag\\
    &+ (1-p) \mathbb{P}(u \in C_{1,1}| u \not \sim v_1, u \not \sim v_2)\bigg)\notag .
\end{align}
Since \(\mathbb{P}(v_2 \in C_{1,1})\leq 1\) and \(\mathbb{P}(v_1 \in C_{1,1}| v_1 \not \sim v_2) - \mathbb{P}(v_1 \in C_{1,1} | v_1 \sim v_2)\leq 1\) (because $v_2 \in C_{2,0}$), we have that: 
\begin{align}\label{badcalc_important_2_sec}
    \beta_2 &\leq 2+(1-p)\sum_{u \in V \backslash \{v_1,v_2\}}\bigg( -p \mathbb{P}(u \in C_{1,1}| u \sim v_1, u \sim v_2)\notag\\ 
    &+ p\mathbb{P}(u \in C_{1,1}| u \sim v_1, u \not \sim v_2) - (1-p) \mathbb{P}(u \in C_{1,1}| u \not \sim v_1, u \sim v_2)\notag\\
    &+ (1-p) \mathbb{P}(u \in C_{1,1}| u \not \sim v_1, u \not \sim v_2)\bigg).
\end{align}
By Equations \ref{badcalc_important_1_sec} and \ref{badcalc_important_2_sec},
\begin{align*}
    \beta_2 - \beta_1 &\leq 3 + (1-p)\sum_{u \in V \backslash \{v_1,v_2\}}\bigg( -2p \mathbb{P}(u \in C_{1,1}| u \sim v_1, u \sim v_2)\notag\\ 
    &-(1-2p)\mathbb{P}(u \in C_{1,1}| u \sim v_1, u \not \sim v_2) - (1-2p) \mathbb{P}(u \in C_{1,1}| u \not \sim v_1, u \sim v_2 )\notag\\
    &+ 2(1-p) \mathbb{P}(u \in C_{1,1} | u \not \sim v_1, u \not \sim v_2 ) \bigg).
\end{align*}
Using the fact that \(\mathbb{P}(u \in C_{1,1}| u \sim v_1, u \sim v_2) =\mathbb{P}(u \in C_{1,1}| u \not \sim v_1, u \not \sim v_2) \), we can simplify the above equation as follows:

\begin{align}\label{rightnowdoing1}
    \beta_2 - \beta_1 &\leq 3 + (1-p)(1-2p)\sum_{u \in V \backslash \{v_1,v_2\}}\bigg( 2 \mathbb{P}(u \in C_{1,1}| u \sim v_1, u \sim v_2)\notag\\ 
    &-\mathbb{P}(u \in C_{1,1}| u \sim v_1, u \not \sim v_2) - \mathbb{P}(u \in C_{1,1}| u \not \sim v_1, u \sim v_2)\bigg)
\end{align}

For $u \in V \backslash \{v_1,v_2\}$, define \(\alpha_1 := |\Gamma(u) \cap (C_{1.0} \backslash \{v_1\})|\) and \(\alpha_2 := |\Gamma(u) \cap (C_{2,0}\backslash \{v_2\})|\). Then \(\alpha_1 \sim \mathrm{Bin}(|C_{1,0}\backslash \{u\}|-1,p)\) and \(\alpha_2 \sim \mathrm{Bin}(|C_{2,0}\backslash \{u\}|-1,p)  \) are independent random variables. We will now proceed by splitting into cases based on the colour of $u$. \hfill\break

\textit{Case 1 : \( u \in C_{1,0}\). } In this case, \begin{align*}
\mathbb{P}(u \in C_{1,1}| u \sim v_1, u \sim v_2) = \mathbb{P}(\alpha_1 - \alpha_2\geq 0)\end{align*}
and\[\mathbb{P}(u \in C_{1,1}| u \sim v_1, u \not \sim v_2) = \mathbb{P}(\alpha_1 - \alpha_2 \geq  -1)\] and \[\mathbb{P}(u \in C_{1,1}| u \not \sim v_1, u \sim v_2) = \mathbb{P}(\alpha_1 - \alpha_2 \geq 1).\] Therefore,
\begin{align*}
&\mathbb{P}(u \in C_{1,1}| u \sim v_1, u \sim v_2) - \mathbb{P}(u \in C_{1,1}| u \sim v_1, u \not \sim v_2) 
= -\mathbb{P}(\alpha_1-\alpha_2 = -1),
\end{align*}
and \begin{align*}
\mathbb{P}(u \in C_{1,1}| u \sim v_1, u \sim v_2)- \mathbb{P}(u \in C_{1,1}| u \not \sim v_1, u \sim v_2) = \mathbb{P}(\alpha_1-\alpha_2=1).
\end{align*}
Therefore,
\begin{multline*}
     2 \mathbb{P}(u \in C_{1,1}| u \sim v_1, u \sim v_2)\notag 
    -\mathbb{P}(u \in C_{1,1}| u \sim v_1, u \not \sim v_2)\\
    - \mathbb{P}(u \in C_{1,1}| u \not \sim v_1, u \sim v_2) = \mathbb{P}(\alpha_1-\alpha_2=1)-\mathbb{P}(\alpha_1-\alpha_2 = -1).
\end{multline*}
By Lemma \ref{taken_from_other_paper}, 
\begin{align*}
|\mathbb{P}(\alpha_1-\alpha_2=1)-\mathbb{P}(\alpha_1-\alpha_2 = -1) | & \leq \frac{40C_{BE}}{(|C_{1,0}|-2)p(1-p)}\\
&\leq \frac{80C_{BE}}{np(1-p)}.
\end{align*}
\textit{Case 2: \( u \in C_{2,0}\). } In this case, \[\mathbb{P}(u \in C_{1,1}| u \sim v_1, u \sim v_2) = \mathbb{P}(\alpha_1 - \alpha_2\geq -1)\]and \[\mathbb{P}(u \in C_{1,1}| u \sim v_1, u \not \sim v_2) = \mathbb{P}(\alpha_1 - \alpha_2 \geq  -2)\] and \[\mathbb{P}(u \in C_{1,1}| u \not \sim v_1, u \sim v_2) = \mathbb{P}(\alpha_1 - \alpha_2 \geq 0).\] Therefore,
\begin{align*}
&\mathbb{P}(u \in C_{1,1}| u \sim v_1, u \sim v_2) - \mathbb{P}(u \in C_{1,1}| u \sim v_1, u \not \sim v_2) 
= -\mathbb{P}(\alpha_1-\alpha_2 = -2),
\end{align*}
and \begin{align*}
\mathbb{P}(u \in C_{1,1}| u \sim v_1, u \sim v_2)- \mathbb{P}(u \in C_{1,1}| u \not \sim v_1, u \sim v_2) = \mathbb{P}(\alpha_1-\alpha_2=0).
\end{align*}
Therefore,
\begin{multline*}
     2 \mathbb{P}(u \in C_{1,1}| u \sim v_1, u \sim v_2)\notag 
    -\mathbb{P}(u \in C_{1,1}| u \sim v_1, u \not \sim v_2)\\
    - \mathbb{P}(u \in C_{1,1}| u \not \sim v_1, u \sim v_2) = \mathbb{P}(\alpha_1 - \alpha_2 = 0 )-\mathbb{P}(\alpha_1-\alpha_2 = -2).
\end{multline*}
By Lemma \ref{taken_from_other_paper}, \begin{align*}|\mathbb{P}(\alpha_1-\alpha_2=0)-\mathbb{P}(\alpha_1-\alpha_2 = -2) |\leq& \frac{40C_{BE}}{(|C_{1,0}|-1)p(1-p)}\\
\leq&  \frac{80C_{BE}}{np(1-p)}.\end{align*}Therefore, upon substituting these inequalities into Equation \ref{rightnowdoing1}, we find:
\begin{align*}
    \beta_2-\beta_1 &\leq 3 + (1-p)|1-2p| \sum_{u \in V \backslash \{v_1,v_2\}} \frac{80C_{BE}}{np(1-p)}\\
    &\leq 3+|1-2p|\frac{80C_{BE}}{p}\\
    &\leq 3 + \frac{80C_{BE}}{p}\\
    &\leq \frac{80C_{BE}+3}{p}.
    \end{align*}

    \end{proof}

\begin{proof}[Proof of Proposition \ref{beta__proposition} (ii). ]
Equation \ref{badbadbasbas} states that:
\begin{align*}
    \beta_2 &= \mathbb{P}(v_2 \in C_{1,1}) + (1-p)\sum_{u \in V \backslash \{v_2\}} \bigg(  \mathbb{P}(u \in C_{1,1}|u \not \sim v_2) - \mathbb{P}(u \in C_{1,1}|u \sim v_2)\bigg).
\end{align*}

Pick some $u \in V\backslash \{v_2\}$. Define \(\alpha_1 := |\Gamma(u)\cap C_{1,0}|\) and \(\alpha_2 := |\Gamma(u) \cap (C_{2,0}\backslash \{v_2\})|\). Then, \(\alpha_1 \sim \mathrm{Bin}(|C_{1,0}\backslash\{u\}|,p)\) and \(\alpha_2 \sim\mathrm{Bin}(|C_{2,0}\backslash\{u\}|-1,p)\) are independent random variables. We will now proceed by case analysis on the colour of $u$.\\\\
\textit{Case 1: \( u \in C_{1,0}\). } In this case, \[\mathbb{P}(u \in C_{1,1}|u \not \sim v_2) = \mathbb{P}(\alpha_1-\alpha_2 \geq 0)\]
and \[\mathbb{P}(u \in C_{1,1}|u \sim v_2) = \mathbb{P}(\alpha_1-\alpha_2 \geq 1).\] Therefore, \begin{align*}
    \mathbb{P}(u \in C_{1,1}|u \not \sim v_2) - \mathbb{P}(u \in C_{1,1}|u \sim v_2) &= \mathbb{P}(\alpha_1-\alpha_2 = 0)\\
    \leq \frac{1.12}{\sqrt{p(1-p)(n-2)}} & \leq \frac{2}{\sqrt{p(1-p)(n-1)}},
\end{align*}
by Lemma \ref{in_paper_A.3} in Appendix \ref{appendix_A}.\\

\textit{Case 2: \( u \in C_{2,0}\). } In this case, \[\mathbb{P}(u \in C_{1,1}|u \not \sim v_2) = \mathbb{P}(\alpha_1-\alpha_2 \geq 1)\]
and 
\[
\mathbb{P}(u \in C_{1,1}|u \sim v_2) = \mathbb{P}(\alpha_1-\alpha_2 \geq 2).
\]
Therefore, \begin{align*}
    \mathbb{P}(u \in C_{1,1}|u \not \sim v_2) - \mathbb{P}(u \in C_{1,1}|u \sim v_2) &= \mathbb{P}(\alpha_1-\alpha_2 = 1)\\
    &\leq \frac{1.12}{\sqrt{p(1-p)(n-2)}} \\
    &\leq \frac{2}{\sqrt{p(1-p)(n-1)}},
\end{align*}
by Lemma \ref{in_paper_A.3} in Appendix \ref{appendix_A}.\\

Therefore, upon substituting into Equation \ref{badbadbasbas}, we find:
\begin{align*}
    \beta_2 &\leq \mathbb{P}(v_2 \in C_{1,1})+ (1-p)\sum_{u \in V \backslash \{v_2\}}\frac{2}{\sqrt{p(1-p)(n-1)}}\\
    &\leq 1 + (1-p)(n-1)\frac{2}{\sqrt{p(1-p)(n-1)}}\\
    &= 1 + \frac{2\sqrt{(1-p)(n-1)}     }{\sqrt{p}}\\
    &\leq \frac{3\sqrt{(1-p)(n-1)}     }{\sqrt{p}}.
\end{align*}

\end{proof}
\section{Proof of Key Lemmas in Step 4 of Theorem \ref{the_power_of_few_2_colours}}
\label{appendix_D_2_colours}
\setcounter{equation}{0}
\renewcommand{\theequation}{D.\arabic{equation}}
\begin{proof}[Proof of Proposition \ref{currentlygoinginappendixBvariance4parttechnical}(i). ]
It is easy to see that 
\begin{align}\label{fjiwe+++fhewfwefwef}
|S^{(1)}_{uv}| - |S^{(2)}_{uv}| - |S^{{{{{*}}}}}_{uv}| &= 2|S^{(1)}_{uv}| - \left( |S^{(1)}_{uv}|+|S^{(2)}_{uv}| + |S^{{{{{*}}}}}_{uv}| \right)\notag\\
&=  2|S^{(1)}_{uv}| - n+2. 
\end{align}
Therefore,
\begin{align}\label{xxccddfdfwefjwefwffw}
    \mathbf{Var}\left(|S^{(1)}_{uv}| - |S^{(2)}_{uv}| - |S^{{{{{*}}}}}_{uv}| \right) = 4 \mathbf{Var}\left(|S^{(1)}_{uv}|\right).
\end{align}

A careful inspection of the definitions of $S^{(1)}_{uv}$ and $\hat{R}_v$ allows one to see that  $S^{(1)}_{uv}$ is what the definition of $\hat{R}_v$ would be if we were considering $G\backslash \{v\}$ instead of $G$. Therefore, by Lemma \ref{variance_lemma_two_colours},
\begin{align*}
    \mathbf{Var}(|S^{(1)}_{uv}|) \leq \frac{7(n-1)}{12}\leq n.
\end{align*}
Therefore, by Equation \ref{xxccddfdfwefjwefwffw}, 
\begin{align}
    \mathbf{Var}\left(|S^{(1)}_{uv}| - |S^{(2)}_{uv}| - |S^{{{{{*}}}}}_{uv}| \right) \leq {4n}.\notag
\end{align}

\end{proof}
\begin{proof}[Proof of Proposition \ref{currentlygoinginappendixBvariance4parttechnical} (ii).]
We will proceed by finding bounds for \(\mathbb{E}\left[|S^{(1)}_{uv}| - |S^{(2)}_{uv}| - |S^{{{{{*}}}}}_{uv}| \right]\). We will start by finding an upper bound.\\

Let \(w \in V\backslash \{u,v\}\). Then,
\begin{align*}
\mathbb{P}\left(w \in S^{(1)}_{uv}\right)\leq \mathbb{P}\left( |\Gamma(w)\cap C_{2,0}| - |\Gamma(w)\cap C_{1,0}\backslash \{u,v\}| \leq 1\right).
\end{align*}
It is easy to see that 
\[
|\Gamma(w)\cap C_{2,0}| - |\Gamma(w)\cap C_{1,0}\backslash \{u,v\}| =  \sum_{s\in C_{2,0} } \mathbbm{1}_{s \sim w} + \sum_{s \in  C_{1,0} \backslash \{u,v\}} -\mathbbm{1}_{s \sim w} ,
\]
which is a sum of independent random variables which are distributed as \( \mathrm{Ber}(p)\) or $-\mathrm{Ber}(p)$. Therefore, by Corollary \ref{in_paper_A.2} in Appendix \ref{appendix_A},
\begin{align*}
    \mathbb{P}(w \in S^{(1)}_{uv})& \leq \Phi\left(\frac{1-p|C_{2,0}|+p(|C_{1,0}|-2)}{\sqrt{(n-2)p(1-p)}}\right) + \frac{C_{BE}}{\sqrt{p(1-p)(n-2)}}\\
   & =  \Phi\left(\frac{1+p(|C_{1,0}|-|C_{2,0}|)}{\sqrt{(n-2)p(1-p)}}\right) + \frac{C_{BE}}{\sqrt{p(1-p)(n-2)}}\\
  &= \Phi\left(\frac{1 \pm 2p\Delta-2p}{\sqrt{(n-2)p(1-p)}}\right) + \frac{C_{BE}}{\sqrt{p(1-p)(n-2)}}\\
    &\leq \Phi\left(\frac{1+2p(\Delta-1)}{\sqrt{(n-2)p(1-p)}}\right) + \frac{C_{BE}}{\sqrt{p(1-p)(n-2)}}\\
    &= \frac{1}{2}+\Phi_0\left(\frac{1+2p(\Delta-1)}{\sqrt{(n-2)p(1-p)}}\right) + \frac{C_{BE}}{\sqrt{p(1-p)(n-2)}}.
\end{align*}
Since \(\Phi_0(a)\leq \frac{a}{\sqrt{2\pi}}\) for any $a>0$, we have:
\begin{align*}
    \mathbb{P}(w \in S^{(1)}_{uv})& \leq\frac{1}{2}+\frac{1}{\sqrt{2\pi}}\left(\frac{1+2p(\Delta-1)}{\sqrt{(n-2)p(1-p)}}\right) + \frac{C_{BE}}{\sqrt{p(1-p)(n-2)}}\\
    &\leq \frac{1}{2}+ \left(\frac{1}{{\sqrt{(n-2)p(1-p)}}}\right) \left(\frac{21}{\sqrt{2\pi}} + C_{BE}\right)\\
    &\leq \frac{1}{2}+\frac{9}{{\sqrt{np(1-p)}}},
\end{align*}
where the penultimate line holds because \(p(\Delta-1)\leq p\Delta\leq 10\). Therefore,
\begin{align*}
    \mathbb{E}\left[ |S^{(1)}_{uv}|\right] &= \sum_{w \in V\backslash\{u,v\}} \mathbb{E}\left[\mathbbm{1}_{w \in S^{(1)}_{uv}}\right]\\
   &= \sum_{w \in V\backslash\{u,v\}}\mathbb{P}\left(w \in S^{(1)}_{uv}\right)\\
   &\leq \frac{n-2}{2}+\frac{9(n-2)}{{\sqrt{np(1-p)}}}\\
   &\leq \frac{n-2}{2}+\frac{9\sqrt{n}}{{\sqrt{p(1-p)}}}.
\end{align*}
Therefore, by Equation \ref{fjiwe+++fhewfwefwef},
\begin{align}\label{thegrindneverstopsgrinch1}
     \mathbb{E}\left[|S^{(1)}_{uv}| - |S^{(2)}_{uv}| - |S^{{{{{*}}}}}_{uv}|\right]\leq \frac{18\sqrt{n}}{{\sqrt{p(1-p)}}}.
\end{align}
Now, we will find a lower bound for \(\mathbb{E}\left[|S^{(1)}_{uv}| - |S^{(2)}_{uv}| - |S^{{{{{*}}}}}_{uv}| \right]\).

It is easy to see that
\begin{align*}
\mathbb{P}\left(w \in S^{(1)}_{uv}\right)\geq \mathbb{P}\left( |\Gamma(w)\cap C_{2,0}| - |\Gamma(w)\cap C_{1,0}\backslash \{u,v\}| \leq 0\right).
\end{align*}
It is easy to see that 
\[
|\Gamma(w)\cap C_{2,0}| - |\Gamma(w)\cap C_{1,0}\backslash \{u,v\}| =  \sum_{s\in C_{2,0} } \mathbbm{1}_{s \sim w} + \sum_{s \in  C_{1,0} \backslash \{u,v\}} -\mathbbm{1}_{s \sim w} ,
\]
which is a sum of independent random variables which are distributed as \( \mathrm{Ber}(p)\) or $-\mathrm{Ber}(p)$. Therefore, by Corollary \ref{in_paper_A.2} in Appendix \ref{appendix_A},
\begin{align*}
    \mathbb{P}(w \in S^{(1)}_{uv})& \geq  \Phi\left(\frac{p(|C_{1,0}|-2)-p|C_{2,0}|)}{\sqrt{(n-2)p(1-p)}}\right) - \frac{C_{BE}}{\sqrt{p(1-p)(n-2)}}\\
  &= \Phi\left(\frac{ \pm 2p\Delta-2p}{\sqrt{(n-2)p(1-p)}}\right) - \frac{C_{BE}}{\sqrt{p(1-p)(n-2)}}\\
    &\geq \Phi\left(\frac{-2p(\Delta+1)}{\sqrt{(n-2)p(1-p)}}\right) - \frac{C_{BE}}{\sqrt{p(1-p)(n-2)}}\\
    &= \frac{1}{2}-\Phi_0\left(\frac{2p(\Delta + 1)}{\sqrt{(n-2)p(1-p)}}\right) - \frac{C_{BE}}{\sqrt{p(1-p)(n-2)}}.
\end{align*}
Since \(\Phi_0(a)\leq \frac{a}{\sqrt{2\pi}}\) for any $a>0$, we have:
\begin{align*}
    \mathbb{P}(w \in S^{(1)}_{uv})& \geq\frac{1}{2}-\frac{1}{\sqrt{2\pi}}\left(\frac{2p(\Delta+1)}{\sqrt{(n-2)p(1-p)}}\right) + \frac{C_{BE}}{\sqrt{p(1-p)(n-2)}}\\
    &\geq \frac{1}{2}- \left(\frac{1}{{\sqrt{(n-2)p(1-p)}}}\right) \left(\frac{21}{\sqrt{2\pi}} + C_{BE}\right)\\
    &\geq \frac{1}{2}-\frac{9}{{\sqrt{np(1-p)}}},
\end{align*}
where the penultimate line holds because \(p(\Delta+1)\leq p\Delta+\frac{1}{2}\leq 10.5\). Therefore,
\begin{align*}
    \mathbb{E}\left[ |S^{(1)}_{uv}|\right] &= \sum_{w \in V\backslash\{u,v\}} \mathbb{E}\left[\mathbbm{1}_{w \in S^{(1)}_{uv}}\right]\\
   &= \sum_{w \in V\backslash\{u,v\}}\mathbb{P}\left(w \in S^{(1)}_{uv}\right)\\
   &\geq \frac{n-2}{2}-\frac{9(n-2)}{{\sqrt{np(1-p)}}}\\
   &\geq \frac{n-2}{2}-\frac{9\sqrt{n}}{{\sqrt{p(1-p)}}}.
\end{align*}
Therefore, by Equation \ref{fjiwe+++fhewfwefwef},
\begin{align}\label{thegrindneverstopsgrinch2}
     \mathbb{E}\left[|S^{(1)}_{uv}| - |S^{(2)}_{uv}| - |S^{{{{{*}}}}}_{uv}|\right]\geq \frac{-18\sqrt{n}}{{\sqrt{p(1-p)}}}.
\end{align}
Therefore, by Equations \ref{thegrindneverstopsgrinch1} and \ref{thegrindneverstopsgrinch2},
\begin{align*}
    \bigg| \mathbb{E}\left[|S^{(1)}_{uv}| - |S^{(2)}_{uv}| - |S^{{{{{*}}}}}_{uv}|\right]\bigg|\leq \frac{18\sqrt{n}}{{\sqrt{p(1-p)}}}.
\end{align*}
It is easy to see that 
\begin{align}\label{ffffjiwe+jiwe+jiwe+}
\mathbb{E}\left[\left(|S^{(1)}_{uv}| - |S^{(2)}_{uv}| - |S^{{{{{*}}}}}_{uv}| \right)^2\right] &= \mathbf{Var}\left(|S^{(1)}_{uv}| - |S^{(2)}_{uv}| - |S^{{{{{*}}}}}_{uv}| \right)\notag \\
&+ \mathbb{E}\left[\left(|S^{(1)}_{uv}| - |S^{(2)}_{uv}| - |S^{{{{{*}}}}}_{uv}| \right)\right]^2.
\end{align}
Therefore, using part (i) of the proposition, we find that 
\begin{align*}
    \mathbb{E}\left[\left(|S^{(1)}_{uv}| - |S^{(2)}_{uv}| - |S^{{{{{*}}}}}_{uv}| \right)^2\right] \leq 4n + \frac{324n}{p(1-p)}\leq \frac{328n}{p(1-p)}\leq \frac{328n}{p} .
\end{align*}

\end{proof}
\begin{proof}[Proof of Proposition \ref{currentlygoinginappendixBvariance4parttechnical} (iii). ]

By definition, 
\begin{align*}
    I(\mathcal{G}) = \sum_{w \in S^{{{{{*}}}}}_{uv}} \mathbbm{1}_{w \sim u \land w \sim v}  \sim \mathrm{Bin}\left(|S^{{{{{*}}}}}_{uv}|,p^2\right).
\end{align*}
%
%
Therefore, by the law of total variance,
\begin{align*}
    \mathbf{Var}\left(I(\mathcal{G})\right) & = \mathbb{E}\left[\mathbf{Var}\left(I(\mathcal{G})\big| | S^{{{{{*}}}}}_{uv} | \right) \right] + \mathbf{Var}\left(\mathbb{E}\left[I(\mathcal{G})\big| | S^{{{{{*}}}}}_{uv} |\right]\right)\\
    &= \mathbb{E}\left[|S^{{{{{*}}}}}_{uv}|p^2(1-p^2)\right] + \mathbf{Var}\left(|S^{{{{{*}}}}}_{uv}|p^2\right)\\
    &= p^2(1-p^2)\mathbb{E}\left[|S^{{{{{*}}}}}_{uv}|\right] + p^4 \mathbf{Var}\left(|S^{{{{{*}}}}}_{uv}|\right)\\
    &\leq p^2\mathbb{E}\left[|S^{{{{{*}}}}}_{uv}|\right] + p^4 \mathbf{Var}\left(|S^{{{{{*}}}}}_{uv}|\right)
\end{align*}
By Proposition \ref{whi_why+_do_I+ke_kp_keep_lesbel}, \(\mathbb{E}[|S^{{{{{*}}}}}_{uv}|]  \leq \frac{2\sqrt{n}}{\sqrt{p}}\) and \(\mathbf{Var}(|S^{{{{{*}}}}}_{uv}|) \leq \frac{289\sqrt{n}}{\sqrt{p}}\). Therefore,
\begin{align*}
    \mathbf{Var}\left(I(\mathcal{G})\right) & \leq 2 p^{3/2}n^{1/2} + 289 p^{7/2}n^{1/2}.\end{align*}
Since $p\leq \frac{1}{4}$,
    \begin{align*}
    \mathbf{Var}\left(I(\mathcal{G})\right) & \leq 2 p^{3/2}n^{1/2} + \frac{289}{4^2} p^{3/2}n^{1/2}\\
    &\leq 21 p^{3/2}n^{1/2}.
   \end{align*} 

\end{proof}
\begin{proof}[Proof of Proposition \ref{currentlygoinginappendixBvariance4parttechnical} (iv). ]
We can write:
\begin{align*}
    \mathbb{E}[I(\mathcal{G})]& = \sum_{w \in V\backslash \{u,v\} } \mathbb{P}\left(w \in I(\mathcal{G})\right)\\
    &= \sum_{w \in V\backslash \{u,v\} } \mathbb{P}\left(w \in S^{{{{{*}}}}}_{uv}, w \in \Gamma(u), w \in \Gamma(v)\right).
\end{align*}
\(S^{{{{{*}}}}}_{uv}\) is defined solely using the edges of $G[V\backslash\{u,v\}]$. Since the edges of $G[V\backslash\{u,v\}]$, \(\Gamma(u)\) and \(\Gamma(v)\) are independent,
\begin{align*}
     \mathbb{E}\left[I(\mathcal{G})\right]& = \sum_{w \in V\backslash \{u,v\} } \mathbb{P}\left(w \in S^{{{{{*}}}}}_{uv}\right)\mathbb{P}\left( w \in \Gamma(u)\right)\mathbb{P}\left( w \in \Gamma(v)\right)\\
     &= \sum_{w \in V\backslash \{u,v\} } \mathbb{P}\left(w \in S^{{{{{*}}}}}_{uv}\right)p^2\\
     &= p^2 \mathbb{E}\left[|S^{{{{{*}}}}}_{uv}|\right]\\
     &\leq 2p^{3/2}n^{1/2} \text{, by Proposition \ref{whi_why+_do_I+ke_kp_keep_lesbel}.}
\end{align*}
\end{proof}
\begin{proof}[Proof of Proposition \ref{currentlygoinginappendixBvariance4parttechnical} (v). ]
Combining parts (iii) and (iv) of the proposition, we find that
\begin{align*}
     \mathbb{E}\left[I(\mathcal{G})^2\right] &= \mathbf{Var} \left(  I( \mathcal{G} ) \right) + \mathbb{E} \left[I(\mathcal{G})\right]^2\\
     &\leq 21p^{3/2}n^{1/2} + 4p^3 n,
\end{align*}
    completing the proof of part (v).
\end{proof}
\begin{proof}[Proof of Proposition \ref{currentlygoinginappendixBvariance4parttechnical} (vi). ]
By Proposition \ref{whi_why+_do_I+ke_kp_keep_lesbel}, \(\mathbb{E}[|S^{{{{{*}}}}}_{uv}|]  \leq \frac{2\sqrt{n}}{\sqrt{p}}\) and \(\mathbf{Var}(|S^{{{{{*}}}}}_{uv}|) \leq \frac{289\sqrt{n}}{\sqrt{p}}\). Therefore,
\begin{align*}
    \mathbb{E}\left[|S^{{{{{*}}}}}_{uv}|^2\right] &= \mathbf{Var}\left(|S^{{{{{*}}}}}_{uv}|\right) + \mathbb{E}\left[|S^{{{{{*}}}}}_{uv}|\right]^2\\
    &\leq \frac{289\sqrt{n}}{\sqrt{p}} + \frac{4n}{p}\\
    &\leq \frac{289{n}}{{p}} + \frac{4n}{p}\\
    &\leq \frac{293n}{p}.
\end{align*}
\end{proof}

\section{Auxiliary Calculations for Step 4 of Theorem \ref{the_power_of_few_2_colours}}
\label{appendix_E_2_colours}
\setcounter{equation}{0}
\renewcommand{\theequation}{C.\arabic{equation}}

\begin{proposition}\label{Gaussian_approximation_for_first_variance_calc_of_2_colours}
Suppose $\frac{\log n}{n} \leq p\leq \frac{1}{4}$ and let \(u\neq v \in C_{1,0}\). Let 
\begin{align*}
    \mathcal{A} := \left[|S^{(1)}_{uv} \cap \Gamma(u) | - |S^{(2)}_{uv} \cap \Gamma(u) | - |S^{{{{{*}}}}}_{uv}\cap \Gamma(u)\backslash\Gamma(v)) | + I\left(\mathcal{G}\right) >0\right].
\end{align*}
    Then,
    \begin{align*}
   \left| \mathbb{P}\left(\mathcal{A}\big|\mathcal{G}\right) - \Phi\left(\frac{|S^{(1)}_{uv}|p-|S^{(2)}_{uv}|p -(|S^{{{{{*}}}}}_{uv}|-a)\frac{p}{1+p} +a}{\sqrt{p(1-p)(|S^{(1)}_{uv}|+|S^{(2)}_{uv}|)+ \frac{p}{1+p}(1-\frac{p}{1+p})(|S^{{{{{*}}}}}_{uv}|-a)}}\right)\right| \leq\\
   \frac{C_{BE}}{\sqrt{p(1-p)(n-2-I(\mathcal{G}))}}.
    \end{align*}
\end{proposition}
\begin{proof}
    \(|S^{(1)}_{uv} \cap \Gamma(u) |\sim \mathrm{Bin}(|S^{(1)}_{uv}|,p)\), \(|S^{(2)}_{uv} \cap \Gamma(u) |\sim \mathrm{Bin}(|S^{(2)}_{uv}|,p)\) and \(|S^{{{{{*}}}}}_{uv}\cap \Gamma(u)\backslash\Gamma(v)) |\sim \mathrm{Bin}(|S^{{{{{*}}}}}_{uv}|-I(\mathcal{G}),\frac{p}{1+p})\) are independent random variables. Since Binomial random variables can be expressed as sums of independent Bernoulli random variables, we may express 
    \[
    |S^{(1)}_{uv} \cap \Gamma(u) | - |S^{(2)}_{uv} \cap \Gamma(u) | - |S^{{{{{*}}}}}_{uv}\cap \Gamma(u)\backslash\Gamma(v)) | + |\Gamma(u)\cap\Gamma(v) \cap S^{{{{{*}}}}}_{uv}|
    \]
    as a sum of positive and negative Bernoulli random variables. Therefore, we can apply Lemma \ref{in_paper_A.1} in Appendix \ref{appendix_A} to find that
\begin{align}\label{dogdogcatcatcorpusdougliduglidodod}
   \left| \mathbb{P}\left(\mathcal{A}\big|\mathcal{G}\right) - \Phi\left(\frac{|S^{(1)}_{uv}|p-|S^{(2)}_{uv}|p -(|S^{{{{{*}}}}}_{uv}|-I(\mathcal{G}))\frac{p}{1+p} +I(\mathcal{G})}{\sqrt{p(1-p)(|S^{(1)}_{uv}|+|S^{(2)}_{uv}|)+ \frac{p}{1+p}(1-\frac{p}{1+p})(|S^{{{{{*}}}}}_{uv}|-I(\mathcal{G}))}}\right)\right|\notag\\
   \leq  C_{BE}\frac{(|S^{(1)}_{uv}+S^{(2)}_{uv}|)p(1-p)(p^2+(1-p)^2)+(|S^{{{{{*}}}}}_{uv}|-I(\mathcal{G}))\left(\frac{p}{1+p}\right)\left(\frac{1}{1+p}\right)\left(\frac{p^2+1}{(1+p)^2}\right)}{\sqrt{p(1-p)(|S^{(1)}_{uv}|+|S^{(2)}_{uv}|)+ \frac{p}{1+p}(1-\frac{p}{1+p})(|S^{{{{{*}}}}}_{uv}|-I(\mathcal{G}))}^3}.
    \end{align}
    Since \(p^2+(1-p)^2 \leq 1\) and \(\frac{p^2+1}{(1+p)^2} \leq 1\), we can further bound Equation \ref{dogdogcatcatcorpusdougliduglidodod} to find:
    \begin{align*}
        &\left| \mathbb{P}\left(\mathcal{A}\big|\mathcal{G}\right) - \Phi\left(\frac{|S^{(1)}_{uv}|p-|S^{(2)}_{uv}|p -(|S^{{{{{*}}}}}_{uv}|-I(\mathcal{G}))\frac{p}{1+p} +I(\mathcal{G})}{\sqrt{p(1-p)(|S^{(1)}_{uv}|+|S^{(2)}_{uv}|)+ \frac{p}{1+p}(1-\frac{p}{1+p})(|S^{{{{{*}}}}}_{uv}|-I(\mathcal{G}))}}\right)\right|\notag\\
   &\leq  C_{BE}\frac{(|S^{(1)}_{uv}+S^{(2)}_{uv}|)p(1-p)+(|S^{{{{{*}}}}}_{uv}|-I(\mathcal{G}))\left(\frac{p}{1+p}\right)\left(\frac{1}{1+p}\right)}{\sqrt{p(1-p)(|S^{(1)}_{uv}|+|S^{(2)}_{uv}|)+ \frac{p}{1+p}(1-\frac{p}{1+p})(|S^{{{{{*}}}}}_{uv}|-I(\mathcal{G}))}^3}\\
   &= \frac{C_{BE}}{\sqrt{p(1-p)(|S^{(1)}_{uv}|+|S^{(2)}_{uv}|)+ \frac{p}{1+p}(1-\frac{p}{1+p})(|S^{{{{{*}}}}}_{uv}|-I(\mathcal{G}))}}.
    \end{align*}
Furthermore, since \(\frac{p}{1+p}\leq p \leq \frac{1}{2}\), we must have \(\frac{p}{1+p}\left(1-\frac{p}{1+p}\right)\leq p(1-p)\) (as $p(1-p)$ is a decreasing function of $p$ for $0\leq p \leq \frac{1}{2}$). Therefore,
\begin{align*}
    &\left| \mathbb{P}\left(\mathcal{A}\big|\mathcal{G}\right) - \Phi\left(\frac{|S^{(1)}_{uv}|p-|S^{(2)}_{uv}|p -(|S^{{{{{*}}}}}_{uv}|-I(\mathcal{G}))\frac{p}{1+p} +I(\mathcal{G})}{\sqrt{p(1-p)(|S^{(1)}_{uv}|+|S^{(2)}_{uv}|)+ \frac{p}{1+p}(1-\frac{p}{1+p})(|S^{{{{{*}}}}}_{uv}|-I(\mathcal{G}))}}\right)\right|\notag\\
    &\leq \frac{C_{BE}}{\sqrt{p(1-p)(|S^{(1)}_{uv}|+|S^{(2)}_{uv}| +|S^{{{{{*}}}}}_{uv}|-I(\mathcal{G}))}}\\
    &= \frac{C_{BE}}{\sqrt{p(1-p)(n-2-I(\mathcal{G}))}}.
\end{align*}
    
\end{proof}
\begin{proposition}\label{Algebraic_manipulation_for_lemma_THIRD_IN_PLAN_appendix_C}
    Let \(X\) be a random variable taking values in \([-1,1]\) and \(\mathcal{E}\) be an event. Then,
    \begin{align*}
    \mathbf{Var}\left(X\right) &\leq \mathbf{Var}(X|\mathcal{E}) + 3 \mathbb{P}(\neg\mathcal{E}).
\end{align*}
\end{proposition}
\begin{proof}
    By the law of total variance,
\begin{align*}
\mathbf{Var}\left(X\right) &= \mathbb{E}\left[ \mathbf{Var}(X|\mathbbm{1}_\mathcal{E}) \right] + \mathbf{Var}\left(\mathbb{E}\left[ X |\mathbbm{1}_\mathcal{E} \right] \right)\\
&= \mathbf{Var}(X|\mathcal{E})\mathbb{P}(\mathcal{E}) +  
\mathbf{Var}(X|\neg \mathcal{E})\mathbb{P}(\neg \mathcal{E})+ \mathbf{Var}\left(\mathbb{E}\left[ X |\mathbbm{1}_\mathcal{E} \right] \right)
\end{align*}
Since \(-1\leq X\leq 1\), \[
\mathbf{Var}(X|\neg \mathcal{E})\leq 1  .
\]
Furthermore,
\begin{align*}
    \mathbf{Var}\left(\mathbb{E}\left[ X |\mathbbm{1}_\mathcal{E} \right] \right) &= \big|\mathbb{E}\left[X|\mathcal{E} \right]-\mathbb{E}\left[ X| \neg\mathcal{E} \right] \big|\mathbb{P}(\mathcal{E})\mathbb{P}\left(\neg\mathcal{E}\right)\\
    &\leq 2 \mathbb{P}(\neg \mathcal{E}).
\end{align*}
Therefore,
\begin{align*}
    \mathbf{Var}\left(X\right) &\leq \mathbf{Var}(X|\mathcal{E}) + 3 \mathbb{P}(\neg\mathcal{E}).
\end{align*}

\end{proof}
\bibliographystyle{amsplain}
\bibliography{bibliography}
\end{document}